\documentclass[oneside,english]{amsart}
\usepackage[latin9]{inputenc}
\usepackage{float}
\usepackage{amsthm}
\usepackage{amssymb}
\usepackage{graphicx}

\makeatletter

\providecommand{\tabularnewline}{\\}

\numberwithin{equation}{section}
\numberwithin{figure}{section}
\theoremstyle{plain}

\newtheorem{thm}{\protect\theoremname}
\theoremstyle{definition}
 
\newtheorem{defn}[thm]{\protect\definitionname}
\theoremstyle{plain}

\newtheorem{prop}[thm]{\protect\propositionname}
\theoremstyle{plain}

\newtheorem{lem}[thm]{\protect\lemmaname}
\theoremstyle{plain}

\newtheorem{cor}[thm]{\protect\corollaryname}

\theoremstyle{remark}
\newtheorem{rem}[thm]{\protect\remarkname}

\makeatother

\usepackage{babel}
\providecommand{\corollaryname}{Corollary}
\providecommand{\definitionname}{Definition}
\providecommand{\lemmaname}{Lemma}
\providecommand{\propositionname}{Proposition}
\providecommand{\remarkname}{Remark}
\providecommand{\theoremname}{Theorem}

\begin{document}
\title{Arrow Spaces: An Approach to Inner Product Spaces and Affine Geometry}
\date{\today}
\author{Hussin Albahboh$^{1}$, Harry Gingold$^{2}$ and Jocelyn Quaintance$^{3}$}
\dedicatory{1-Department of Mathematics, West Virginia University, Morgantown,
WV 26506, USA. 3-Department of Computer Science, University of Pennsylvania,
Philadelphia, PA, 19104, USA}
\begin{abstract}
Given a postulated set of points, an algebraic system of axioms is proposed
for an ``arrow space'".  An arrow is defined to be
an ordered set of two points $\left<T, H\right>$, named respectively Tail and Head.
The set of arrows is an arrow space. 
The arrow space is axiomatically endowed with
an arrow space ``pre-inner product" which is analogous to the inner product of a
Euclidean vector space.
Using this arrow space pre-inner product, various properties of the arrow space are derived 
and contrasted with the properties of a Euclidean vector space.
The axioms of a vector space and its associated inner product are derived as
theorems that follow from the axioms of an arrow space since vectors are rigorously shown 
to be equivalence classes of arrows. Applications of using an arrow space to solve
geometric problems in affine geometry are provided. 
\end{abstract}

\maketitle


\section{Introduction}
An arrow is a fundamental object of mathematics and physics that is manifested graphically as a line segment with a direction. Indeed, the first appearance of a an arrow heads back 
to about circa 62,000 years before the common era. Another fundamental concept of mathematics and physics is the vector.
 Vectors are abstract algebraic quantities that are often represented by arrows. In physics, there are applications where
 vectors are forces and the arrow is a pictorial portrayal of this vector force acting on a point mass. 
 We learn from Crow \cite{key-15} that in 1687 Issac Newton used the main diagonal of the parallelogram 
to represent the resultant of two forces, with the understanding that the addition operation of two such forces {\it must} be commutative. 
 This lead to the widely accepted requirement that the addition of vectors is commutative. But as we will soon see, the addition of arrows is not.

\medskip
Consider three ``points'' $A$,$B$, and $C$. Let an arrow be an ordered pair of ``points"
$\left<T,H\right>$, where $T$ stands for Tail and $H$ stands for Head.  
Two arrows can be added together if and only if the head of the first is equal to the tail of the second.
Consider the five {\it distinct} arrows $\overrightarrow{AB}$, $\overrightarrow{BC}$, $\overrightarrow{AC}$, 
$\overrightarrow{BA}$, $\overrightarrow{BB}$, and $\overrightarrow{AA}$. 
Define the addition operation $+_{A}$ of two arrows $\overrightarrow{AB}+_{A}\overrightarrow{BC}:=_A\overrightarrow{AC}$.
Notice that  $\overrightarrow{BC}+_{A}\overrightarrow{AB} \neq_{A} \overrightarrow{AC}$ since $\overrightarrow{BC}+_{A}\overrightarrow{AB}$ is
undefined. Thus $+_{A}$ is {\it not} commutative. The nature of this noncommutativity is subtle since both $\overrightarrow{AB}+_{A}\overrightarrow{BA}$ and $\overrightarrow{BA}+_{A}\overrightarrow{AB}$ are defined with $\overrightarrow{AB}+_{A}\overrightarrow{BA} =_{A} \overrightarrow{AA}$ and
$\overrightarrow{BA}+_{A}\overrightarrow{BB} =_{A} \overrightarrow{BB}$, yet $\overrightarrow{AA}\neq_{A}\overrightarrow{BB}$.
This difference in the addition operation highlights the fact that the vector, although geometrically represented by arrow, cannot and should not be construed to be the same entity as the arrow.
As we eventually see in Section \ref{sec:7.The-Equivalence-C}, the arrow is in fact a precursor of the vector since the vector represents an equivalence class of arrows.

\medskip
Although $+_A$ is limited to certain pairs of arrows, the addition of arrows, (when well defined), clearly encodes then starting point and the
ending point within its resulting summation.  As an example of the meaning of this sentence, 
take three points $A$, $B$, and $C$ and form Triangle $ABC$. We describe various paths around 
Triangle $ABC$ via arrow addition as show by the three paths illustrated in Figure 0.1.
\begin{align}
\label{[eq:NONCOMMUTATIVE]}
\overrightarrow{AB}+_{A}\overrightarrow{BC}+_{A}\overrightarrow{CA}&=_A\overrightarrow{AA},
\quad\overrightarrow{BA}+_{A}\overrightarrow{AC}+_{A}\overrightarrow{CA}=_A\overrightarrow{BB}\notag\\
\overrightarrow{CB}+_{A}\overrightarrow{BA}+_{A}\overrightarrow{AC}&=_A\overrightarrow{CC}.
\end{align}
\begin{figure}[H]
\label{introfig2}
\includegraphics[width=4.0in,height=2.0in]{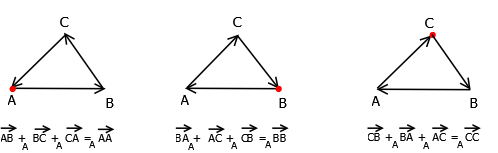}
\end{figure}

The first begins and ends at Vertex $A$, the second begins and ends at Vertex $B$, while the third begins and ends at Vertex $C$.
Vector addition does not provide such precise information.
Let $V$,$W$, and $Z$ be three vectors in a linear space over the field $\mathbb{R}$. Let $+_{V}$ be the addition operation between vectors. Let $0$ denote the identity element of the given linear space. 
The equation $V+_{V}W+_{V}Z=0$ is a manifestation of the three arrow relations in (\ref{[eq:NONCOMMUTATIVE]}). 
If the points $A$,$B$, and $C$ are all distinct, then the vector equation $V+_{V}W+_{V}Z=0$  
cannot tell which of the three relations in (\ref{[eq:NONCOMMUTATIVE]}) is our ultimate destination.
Such information is often useful.
Numerous problems in mathematical physics and in geometry make 
it imperative to know what is the first arrow in a sum of arrows and which one is the last since
the goal is to determine the final destination given an initial point of departure
 Our approach of an arrow as an ordered set of points $\left<T, H\right>$ provides the framework to specifically determine such a final destination. 
 
 \medskip
The collection of such ordered points, 
which we call an arrow space, also imparts a geometric framework to solve many problems of affine geometry without the necessity of a group action.  
This is because our definition of arrow is equivalent
to the approach in \cite{key-10}, \cite{key-20}, where an arrow 
$\overrightarrow{AB}$ representing a vector $V$ is obtained by $V$ acting on the point $A$ via translation, namely $B = A + V$. 
This action is the heart of the definition of an affine space, (see Definition 2.1 in \cite{key-20}), and mimics the action of a force on a point. 
The treatment of \cite{key-10}, \cite{key-20} is consistent with an axiom advanced in the first half of the 20th by K. O. Friedrichs \cite{key-21}, namely
given a vector $V$ and a point $A$ there exists a unique point $B$ such that the arrow $\overrightarrow{AB}$ corresponds to the vector $V$.
 
\medskip
One of the main goals of this article is to propose a rigorous axiomatic setting of arrow spaces and their properties which is not found in either current linear algebra textbooks or the vast collection of vector analysis literature listed by Crow \cite{key-15}. Our axiomatic treatment uses concepts and nomenclature from set theory, algebra, and inner product spaces. In this axiomatic treatment, vectors originate naturally from arrows rather than vice versa. The fundamental nature of arrows having length and direction is being brought out in full force by the modern notion of an arrow space pre-inner product defined on any pair of arrows. The axioms of the set of real numbers together with the supremum axiom are adopted in order guarantee that the real number line has no ``holes". Hence, calculus and its theorems are readily available for use when discussing properties of the arrow space and its associated pre-inner product. The modern set of axioms of vector analysis are derived  as theorems from our axioms of an arrow space since vectors are equivalence classes of arrows, where two arrows are equivalent if and only if they same length and same direction; the concept of same direction is
captured by requiring that a certain arrow pre-inner be equal to $1$. 

\medskip
The notion of a vector as an equivalence class of arrows which share the same length and direction is predated by 
Giusto Bellavitis. According to Crow \cite{key-15}, in 1835 Bellavitis publishes his first exposition on systems of equipollences.
This system has some features in common with the now traditional vector analysis, as is suggested in his definition of equipollent; two straight lines are called equipollent if they are equal, parallel and directed in the same sense. His lines behave in exactly the same manner as complex numbers behave, but it is important to note that he viewed his lines as essentially geometric entities, not as geometric representations of algebraic entities.


\medskip
We stress that the entities relied upon in our axiomatic treatment are algebraic. Geometric entities like point, line, ball, sphere, etc. are used  for the representation of algebraic entities and making them tangible. The algebraic axiomatic foundation of a linear vector space and of a Hilbert space have been scrutinized over numerous decades and withstood the test of time. Thus one may view this current work as an attempt to propose an expanded Euclidean geometry that is supplemented by arrows and vectors and has much in common with linear algebra, with inner product spaces, and with their confirmed foundation. 

\medskip
The order of work in this paper is as follows.
Section \ref{sec:Arrow-Spaces} rigorously defines the notion of an arrow space $\mathcal{P}_{A}$.
Starting with a postulated set of points $\mathcal{P}$, we define
arrows as ordered pairs of points $\overrightarrow{AB} =(A,B)$.
We say that two points $A,\;B$ are equal, denoted by $A=_{P}B$,
if $A$ and $B$ refer to a single point.  Based on this definition of equal
points, equal arrows can be defined as $\overrightarrow{AB}=_{A}\overrightarrow{CD}$
iff $A=_{P}C$ and $B=_{P}D$. An arrow addition, denoted
by $+_{A}$, is then introduced. Two arrows can be added if and only if they
have a point in common, namely
$\overrightarrow{AB}+_{A}\overrightarrow{BC}=_{A}\overrightarrow{AC}$.
The definition of $+_{A}$ implies if $A\neq_{P}B$, then
$\overrightarrow{AB}+_{A}\overrightarrow{BA}\neq_{A}\overrightarrow{BA}+_{A}\overrightarrow{AB}$.
Furthermore, motivated by the axioms and properties
of inner product spaces, we define an arrow pre-inner product, denoted
by $\left<-,-\right>_{A}$, on $\mathcal{P}_{A}$. This enables us to
define the measure of an arrow, denoted $||.||_{A}$,
an arrow scalar multiplication $(t)\;\overrightarrow{AB}$, a
line, betweenness of points, and more. 

\medskip
Sections \ref{sec:Comparing-Arrows-With} and  \ref{sec:Section4} are devoted to making 
comparisons between arrows spaces and vectors spaces.  Section \ref{sec:Comparing-Arrows-With}
emphasizes the differences between arrows and vectors by focusing on how the operation of arrow addition $+_{A}$
deviates from that of vector addition.  Surprisingly, arrow addition
is non-commutative.  A similar comparison is then applied to the respective operations of scalar multiplication. Section \ref{sec:Section4} focuses on the similarities between arrow spaces and vector spaces. We begin by showing that $+_{A}$, like its vector counterpart, is associative, and that both vector spaces and arrow spaces have the notion of an additive identity and an additive inverse.  We also show that arrow scalar multiplication is also associative and then go on to showcase those properties of arrow scalar multiplication which have analogs in terms of vector space scalar multiplication.

 \medskip
In Section \ref{sec:4.Eq.Clss} we use arrow scalar multiplication to define the notion of a line in an arrow space $\mathcal{P}_A$ and show that given any two distinct points $A$ and $B$, there exists
a unique line, denoted by $l_{AB}$, containing $A$ and $B$; see Theorem \ref{thm:WAS.Axiom (3)}. 
We then restrict our attention to the set of points that
lie on a line $l$, namely the set of points $\mathcal{P}_l$.  For $\mathcal{P}_{l_A}$, the arrow space associated
with the line $l$, we define
a relation $\Re_l$
as follows: we say that $\overrightarrow{AB}\;\Re_l\;\overrightarrow{CD}$
if and only if either $A=_{P}B$ and $C=_{P}D$, or
\[
||\overrightarrow{AB}||_{A}=||\overrightarrow{CD}||_{A}\,\,\, \text{and}\,\,\,
\left<\frac{\overrightarrow{AB}}{||\overrightarrow{AB}||_{A}},\;\frac{\overrightarrow{CD}}{||\overrightarrow{CD}||_{A}}\right>_{A}=1.
\]
From the second condition of this relation we can extract 
a definition of parallelism which mimics Euclid's notion of parallel lines; see Subsection 5.3.
But first we prove that $\Re_l$ is an equivalence relation by showing that
if $\overrightarrow{AB}\;\Re_l\;\overrightarrow{CD}$ and $\overrightarrow{EF}\;\Re_l\;\overrightarrow{GH}$, then
$\left<\overrightarrow{AB},\overrightarrow{EF}\right>_{A}=\left<\overrightarrow{CD},\overrightarrow{GH}\right>_{A}$; see Theorem
\ref{thm:WAS.AXIOM (5)}.
We then exploit the definition of a line and the equivalence relation $\Re_l$ to prove that 
given any arrow $\overrightarrow{AB}$ and any point $P$ in $P_{l_{A}}$, there exists
a unique parallel arrow $\overrightarrow{PK}$ (likewise a unique arrow $\overrightarrow{LP}$)
such that $\overrightarrow{AB}\;\Re_l\;\overrightarrow{PK}$ (likewise
$\overrightarrow{AB}\;\Re_l\;\overrightarrow{LP}$); see Theorem \ref{thm:WAS.AXIOM (6)}.

\medskip 
In Section \ref{sec:6.An-Equivalence-Relation} we return to an arbitrary arrow space $\mathcal{P}_A$ and directly define the relation $\Re$ in the context of $\mathcal{P}_A$; see Definition 42. But in order to prove that $\Re$ is an equivalence relation we have to postulate Theorem \ref{thm:WAS.AXIOM (5)} as Axiom 4. 
The obtained equivalence classes, denoted by $[\,-\,]$, will become vectors and turns $\mathcal{P}_A$ into a vector space
$\mathcal{P}_v$. Of course we need to show that $\mathcal{P}_v$ satisfies all the axioms of a vector space. This means we need to convert arrow addition into vector addition, $+_{V}$. Key to this conversion is the existence of a unique parallel arrow through a fixed given point, namely the analog of Theorem \ref{thm:WAS.AXIOM (5)}, which must now be postulated as Axiom 5.  By using Axiom 5 we define $+_{V}$ as follows: given two vectors $[\overrightarrow{AB}]$ and $[\overrightarrow{CD}]$, and an arbitrary point $P$, let $\overrightarrow{LP}$ be the unique arrow such that 
$\overrightarrow{AB}\;\Re\;\overrightarrow{LP}$, and let $\overrightarrow{PK}$ be the unique arrow such that 
$\overrightarrow{CD}\;\Re\;\overrightarrow{PK}$.
(The existence and uniqueness of the two arrows $\overrightarrow{LP}$
and $\overrightarrow{PK}$ are guaranteed by Axiom 5). Then 
\[
[\overrightarrow{AB}]+_{V}[\overrightarrow{CD}]=_{V}[\overrightarrow{LP}]+_{V}[\overrightarrow{PK}]=_{V}[\overrightarrow{LP}+_{A}\overrightarrow{PK}]=_{V}[\overrightarrow{LK}];
\]
see Figure \ref{figintro}.
\begin{figure}[ht]
\label{figintro}
\includegraphics[scale=0.5]{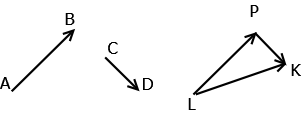}\caption{Adding equivalence classes of arrows.}
\end{figure}
Also arrow scalar multiplication is used to define a vector
scalar multiplication as follows:
\[
t\;[\overrightarrow{AB}]=_{V}[(t)\;\overrightarrow{AB}].
\]

\medskip
In Section \ref{sec:7.The-Equivalence-C} we prove that the set
of equivalence classes of arrows, namely $\mathcal{P}_{v}$, with
the operations of  vector addition and vector scalar multiplication as defined above fulfills all 
the axioms of vector space. Then in Section \ref{sectionaffinegeomex} we demonstrate how the tools of this article can be 
applicable to the field of affine geometry by solving two problems through the context of the vector space $\mathcal{P}_v$ associated with the arrow space $\mathcal{P}_A$.
 The first problem is to show that for a given line $l_{OG}$, for every
point $P\notin l_{OG}$, there exists a unique point $W\in l_{OG}$
such that $\left<\overrightarrow{WO},\overrightarrow{WP}\right>_{A}=0$; see Theorem
\ref{thm:POINT.LINE}.  Then the Cauchy Schwartz inequality for arrows spaces is a corollary of 
Theorem \ref{thm:POINT.LINE}.
The other application is related to the barycentric coordinates of an affine space 
[\cite{key-10}, Page 22]. We show that given a set $\{P_{i}\}_{i=1}^{n}$ of
distinct points in $\mathcal{P}$ and a finite set of real numbers
$\{\lambda_{i}\}_{i=1}^{n}$ such that $\sum_{i=1}^{n}\lambda_{i}=1$,
for a fixed coordinate free origin $O$, there exists a unique
point $M$ such that $\sum_{i=1}^{n}[(\lambda_{i})\overrightarrow{OP_{i}}]=_{V}[\overrightarrow{OM}]$. Furthermore,
the point $M$ is independent from the choice of the origin $O$.

\section{\label{sec:Arrow-Spaces}Arrow Spaces}

In this section we rigorously define the notion of an arrow and an
arrow space. The definition of an arrow and its associated arrow space
depends on a postulated set of points.

\medskip
\noindent {\bf Axiom 0}. There exists a set of points $\mathcal{P}$.

\medskip
We will label individual points with Roman letters and with a slight
abuse of notation denote $\mathcal{P}=\{A,\;B,\;C,...\}$. This labeling
convention will allow us to denote equality among the elements of
$\mathcal{P}$.
\begin{defn}
Let $\mathcal{P}=\{A,\;B,\;C,...\}$ be a set of points. Let $A,\;B\in\mathcal{P}$.
We define $A=_{P}B$ if and only if $A$ and $B$ refer to a single point. Otherwise,
we write $A\neq_{P}B$ meaning that $A$ and $B$ refer to two distinct
points.
\end{defn}

We can now define an arrow as an ordered pair of points.
\begin{defn}
\label{def:ArrowDef.}Let $\mathcal{P}=\{A,\;B,\;C,...\}$ be a set
of points. Let

\[
\mathcal{P\times P}=\{(A,B)\;|\;A,\;B\in\mathcal{P}\}
\]
be the Cartesian product of $\mathcal{P}$. Given any two points $A,\;B\in\mathcal{P}$,
we define an {\it arrow}, denoted by $\overrightarrow{AB}$, to be the ordered
pair $(A,B)$. The two points $A$ and $B$ are the {\it tail} and the {\it head}
of the arrow $\overrightarrow{AB}$ respectively. If $A=_{P}B$,
then $(A,B)=(A,A)$ and we denote the associated arrow by $\overrightarrow{AA}$.
We call the set of all arrows, whose tails and heads are the points
of a set $\mathcal{P}$, an {\it arrow space} and denote it by $\mathcal{P}_{A}$.
\end{defn}

In the following definition we define equality among arrows.
\begin{defn}
\label{def:equal arrows}Let $\overrightarrow{AB},\;\overrightarrow{CD}$
be two arrows in $\mathcal{P}_{A}$. We put $\overrightarrow{AB}=_{A}\overrightarrow{CD}$
if and only if $A=_{P}C$ and $B=_{P}D$. If $A\neq_{P}C$ or $B\neq_{P}D$,
we say that the two arrows are different and write $\overrightarrow{AB}\neq_{A}\overrightarrow{CD}$.
\end{defn}

Next comes a technical definition, the negation of an arrow.
\begin{defn}
\label{def:(-)Minuse.arrowAB}Given the arrow $\overrightarrow{AB}$,
we define $-\;\overrightarrow{AB}$ as $-\;\overrightarrow{AB}=\overrightarrow{BA}$; see Figure 3.1.
\end{defn}

Note that if $A\neq_{P}B$, then $\overrightarrow{AB}\neq_{A}-\;\overrightarrow{AB}$.

\medskip
In order to use arrows as a tool for solving various problems in affine
geometry, and since arrows are actually precursors to vectors (see Section \ref{sec:6.An-Equivalence-Relation}),
we want to be able to manipulate them in a way that is reminiscent
of the way we manipulate vectors. This means we need to define
the binary operation of arrow addition and the notion of scalar multiplication
acting on an arrow. Arrow addition is rather straightforward
as seen by the following definition.
\begin{defn}
\label{def:ARROWSADDITIONHEADTAIL}Let $\overrightarrow{AB}$ and $\overrightarrow{BC}$
be any two arrows in $\mathcal{P}_{A}$. We define {\it arrow addition}, denoted
by $+_{A}$, of $\overrightarrow{AB}$ and $\overrightarrow{BC}$
as $\overrightarrow{AB}+_{A}\overrightarrow{BC}=_{A}\overrightarrow{AC}$.

\begin{figure}[ht]
\includegraphics[scale=0.3]{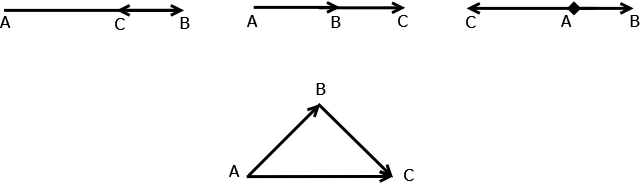}
\caption{\label{fig:AB+BC=00003DAC}An illustration of Definition \ref{def:ARROWSADDITIONHEADTAIL} where  $\protect\overrightarrow{AB}+_{A}\protect\overrightarrow{BC}=_{A}\protect\overrightarrow{AC}.$}
\end{figure}
\end{defn}

However, to define the notion scalar multiplication in an arrow space
we need metric notions of length and distance, along with the Euclidean
notion of an angle. These crucial notions are algebraically captured
by a postulated arrow pre-inner product, namely a symmetric, positive definite
``bilinear'' mapping $\left<-,-\right>_{A}:\mathcal{P}_{A}\times\mathcal{P}_{A}\longrightarrow\mathbb{R}$.
This arrow pre-inner product will be a tool for defining the measure of
an arrow, the definition of a line, and notion of betweenness for points that lie on
a line. 

\medskip
\noindent {\bf Axiom 1}. There exists a mapping  $\left<-,-\right>_{A}:\mathcal{P}_{A}\times\mathcal{P}_{A}\longrightarrow\mathbb{R}$
such that
\begin{itemize}
\item[1. ] ({\it positive definiteness})

\begin{equation}
\left<\overrightarrow{AB},\overrightarrow{AB}\right>_{A}\geq0\;\;
\text{and}\;\left<\overrightarrow{AB},\overrightarrow{AB}\right>_{A}=0\;\text{if and only if}\;A=_{P}B,\label{eq:AXIOMz}
\end{equation}

\item[2.] ({\it symmetry})
\begin{equation}
\left<\overrightarrow{AB},\overrightarrow{CD}\right>_{A}=\left<\overrightarrow{CD},\overrightarrow{AB}\right>_{A},\label{eq:AXIOMb}
\end{equation}

\item[3.] ({\it arrow addition linearity})
\begin{equation}
\left<\overrightarrow{AB}+_{A}\overrightarrow{BC},\overrightarrow{LM}\right>_{A} =
\left<\overrightarrow{AB},\overrightarrow{LM}\right>_{A} +
 \left<\overrightarrow{BC},\overrightarrow{MR}\right>_{A},\label{eq:AXIOMc0}
\end{equation}

\item[4.]({\it negation rule})
\begin{equation}
\left<-\;\overrightarrow{AB},\overrightarrow{CD}\right>_{A}= (-1)\;\left<\overrightarrow{AB},\overrightarrow{CD}\right>_{A}.\label{eq:A01-1a}
\end{equation}

\end{itemize}

Regarding Axiom 1, we make two observations.
First, Equation (\ref{eq:AXIOMb}), when combined with Equation (\ref{eq:AXIOMc0}), implies that
\begin{align}
\left<\overrightarrow{AB}+_{A}\overrightarrow{BC},\overrightarrow{LM}+_{A}\overrightarrow{MR}\right>_{A} 
& =\left<\overrightarrow{AB},\overrightarrow{LM}\right>_{A}+\left<\overrightarrow{AB},\overrightarrow{MR}\right>_{A}\notag\\
 &+ \left<\overrightarrow{BC},\overrightarrow{LM}\right>_{A}+\left<\overrightarrow{BC},\overrightarrow{MR}\right>_{A}.\label{eq:AXIOMc}
 \end{align}
Secondly, Equation (\ref{eq:AXIOMb}), when combined with Equation (\ref{eq:A01-1a}), implies that
\begin{equation}  
 \left<\overrightarrow{AB},-\;\overrightarrow{CD}\right>_{A}
 =(-1)\;\left<\overrightarrow{AB},\overrightarrow{CD}\right>_{A} = \left<-\;\overrightarrow{AB},\overrightarrow{CD}\right>_{A}. \label{eq:A01-1}
\end{equation}

\medskip
Intuitively, we associate $\overrightarrow{AA}$ with the ``zero''
arrow, one of the infinitely many ``zero" arrows. Thus we would like pre-inner product of Axiom 1 to behave
correctly with respect to zero, namely that $\left<\overrightarrow{AA},-\right>=0$.
This is indeed the case as evidenced by the following proposition.
\begin{prop}
\label{prop:AACDZERO}For any two arrows $\overrightarrow{AA}$ and $\overrightarrow{CD}$ of 
$\mathcal{P}_A$, we have $\left<\overrightarrow{AA},\overrightarrow{CD}\right>_{A}=0$.
\end{prop}

\begin{proof}
Let $\overrightarrow{AA}$ and $\overrightarrow{CD}$ be any two arrows.
Since
\[
\overrightarrow{AA}=_{A}\overrightarrow{AD}+_{A}\overrightarrow{DA}\qquad\text{and}\qquad
\;\overrightarrow{CD}=_{A}\overrightarrow{CA}+_{A}\overrightarrow{AD},
\]
Definition \ref{def:(-)Minuse.arrowAB} and Equations (\ref{eq:AXIOMc}) and (\ref{eq:A01-1})  imply that 

\begin{align*}
\left<\overrightarrow{AA},\overrightarrow{CD}\right>_{A} & =\left<\overrightarrow{AD}+_{A}\overrightarrow{DA},\overrightarrow{CA}+_{A}\overrightarrow{AD}\right>_{A}\\
 & =\left<\overrightarrow{AD},\overrightarrow{CA}\right>_{A}+\left<\overrightarrow{AD},\overrightarrow{AD}\right>_{A}+\left<\overrightarrow{DA},\overrightarrow{CA}\right>_{A}+\left<\overrightarrow{DA},\overrightarrow{AD}\right>_{A}\\
 & =\left<\overrightarrow{AD},\overrightarrow{CA}\right>_{A}+\left<\overrightarrow{AD},\overrightarrow{AD}\right>_{A}-\left<\overrightarrow{AD},\overrightarrow{CA}\right>_{A}-\left<\overrightarrow{AD},\overrightarrow{AD}\right>_{A}=0.
\end{align*}
\end{proof}

The arrow pre-inner product of Axiom 1 provides a way of defining the
measure (or length) of any arrow.
\begin{defn}
\label{def:MeasureOf-any-arrow}For any arrow $\overrightarrow{AB} \in \mathcal{P}_A$,
we define a {\it measure}, denoted $||-||_{A}$, as follows: $||\overrightarrow{AB}||_{A}=\sqrt{\left<\overrightarrow{AB},\overrightarrow{AB}\right>_{A}}$. 
If $||\overrightarrow{AB}||_{A}=1$, then we call this arrow a {\it unit
arrow}.
\end{defn}

The following lemma is an immediate consequence of Equation (\ref{eq:AXIOMz})
and Definition \ref{def:MeasureOf-any-arrow}.
\begin{lem}
\label{lem:||AB||=00003D0 IFF}For any arrow $\overrightarrow{AB}\in \mathcal{P}_A$,
we have $||\overrightarrow{AB}||_{A}=0$ if and only if $A=_{P}B$.
\end{lem}

We take advantage of the arrow pre-inner product and the measure to define the notion of scalar
multiplication in an arrow space. To avoid confusion with the negation operation of Definition \ref{def:(-)Minuse.arrowAB},
we will {\it always} surround the scalar multiple with parenthesis.
\begin{defn} \label{def:(Scalar-Multiplication-of}Let $\overrightarrow{AB}$
be any arrow in $\mathcal{P}_{A}$ and $t\in\mathbb{R}$. If $A=_{P}B$,
or $t=0$, we put $(t)\overrightarrow{AB}=\overrightarrow{AA}$ .
If $A\neq_{P}B$ and $t\neq0$, we put $(t)\overrightarrow{AB}=_{A}\overrightarrow{AD}$,
where $D$ is a point in $\mathcal{P}$ such that\\

\begin{itemize}
\item[1.] $||\overrightarrow{AD}||_{A}=|t|\;||\overrightarrow{AB}||_{A}$,\\

\item[2.]
$\left<\overrightarrow{AB},\overrightarrow{AD}\right>_{A}
= ||\overrightarrow{AB}||_{A}||\overrightarrow{AD}||_{A}$
if $t>0$,\\ and $\left<\overrightarrow{AB},\overrightarrow{AD}\right>_{A}
= -||\overrightarrow{AB}||_{A}||\overrightarrow{AD}||_{A}$
if $t<0$.
\end{itemize}
\end{defn}

Observe that Definition \ref{def:(Scalar-Multiplication-of} only
deals with the existence of the point $D$. We will see in Section
\ref{sec:4.Eq.Clss} that this $D$ is in fact unique; see Theorem \ref{thm:WAS.AXIOM (6)}.

\medskip The arrow pre-inner product, along with Definition \ref{def:(Scalar-Multiplication-of} (2), suggests the following definition which is an
algebraic quantification for when two arrow have the same direction.
\begin{defn}
\label{arrowdirdef}
Let  $A$, $B$, $C$, and $D$ be four distinct points of $\mathcal{P}$. We say $\overrightarrow{AB}$ has the {\it same direction} as 
$\overrightarrow{CD}$ if and only if $\left<\overrightarrow{AB},\overrightarrow{AD}\right>_{A}
= ||\overrightarrow{AB}||_{A}||\overrightarrow{AD}||_{A}$, or equivalently if and only if 
$\left<\frac{\overrightarrow{AB}}{|\overrightarrow{AB}||_{A}},\frac{\overrightarrow{CD}}{|\overrightarrow{CD}||_{A}}\right> = 1$.
We say $\overrightarrow{AB}$ has the {\it opposite direction} as 
$\overrightarrow{CD}$ if and only if $\left<\overrightarrow{AB},\overrightarrow{AD}\right>_{A}
= -||\overrightarrow{AB}||_{A}||\overrightarrow{AD}||_{A}$, or equivalently if and only if 
$\left<\frac{\overrightarrow{AB}}{|\overrightarrow{AB}||_{A}},\frac{\overrightarrow{CD}}{|\overrightarrow{CD}||_{A}}\right> = -1$.
Let $O$, $P$, $Q$, and $R$ be four (not necessarily distinct) points of $\mathcal{P}$. We say $\overrightarrow{OP}$ is {\it perpendicular} to, or 
forms a {right angle} with $\overrightarrow{QR}$ if and only if $\left<\overrightarrow{OP},\overrightarrow{QR}\right>_{A} = 0$.
\end{defn}

\begin{rem}
Definition \ref{arrowdirdef} implicitly measures when ``angle" between two nontrivial arrows is zero (same direction), when it is
$\pi$ (opposite direction), and when it is $\pi/2$ (right angle). 
These are the only three situations of angle measurement that are necessary for the axiomatic presentation of an arrow space presented 
in this paper.
A thorough treatment of angles between two arrows is discussed in our LAA paper and in the Gingold/Salah paper.
\end{rem}

\medskip
Since we now have the notion of scalar multiplication, we require
that the arrow pre-inner product of Axiom 1 behaves correctly with respect
to this operation. In other words, we have the following axiom.

\medskip
\noindent {\bf Axiom 2}. ({\it scalar multiplication linearity})
For $\overrightarrow{AB}$ and $\overrightarrow{CD}$
any two arrows of $\mathcal{P}_A$ and $t\in\mathbb{R}$, we have
\begin{equation}
\left<(t)\;\overrightarrow{AB},\overrightarrow{CD}\right>_{A} = t\left<\overrightarrow{AB},\overrightarrow{CD}\right>_{A}.\label{eq:A01s}
\end{equation}

Observe that Equation (\ref{eq:A01s}), when combined with Equation (\ref{eq:AXIOMb}), implies that
\begin{equation}
\left<(t)\;\overrightarrow{AB},(s)\;\overrightarrow{CD}\right>_{A}=t\;s\left<\overrightarrow{AB},\overrightarrow{CD}\right>_{A},\label{eq:A01}
\end{equation}
a fact that will be extremely useful for the proofs in the paper.

\medskip
The following lemma shows the relationship between length, scalar multiplication, and negation.
\begin{lem}
\label{lem:tAB}For any arrow $\overrightarrow{AB}\in \mathcal{P}_A$ and any
$t\in\mathbb{R}$ we have

\begin{itemize}
\item[1.]
\begin{equation}
||(t)\;\overrightarrow{AB}||_{A}=|t|\;||\overrightarrow{AB}||_{A},\label{eq:LEM00}
\end{equation}

\item[2.]
\begin{equation}
||\overrightarrow{AB}||_{A}=||\overrightarrow{BA}||_{A}.\label{eq:LEM11}
\end{equation}
\end{itemize}
\end{lem}

\begin{proof}
1) Definition \ref{def:MeasureOf-any-arrow} and Equation (\ref{eq:A01}) imply that
\[
||(t)\;\overrightarrow{AB}||_{A}^{2}  = \left<(t)\;\overrightarrow{AB},(t)\;\overrightarrow{AB}\right>_{A}
=t^{2}\left<\overrightarrow{AB},\overrightarrow{AB}\right>_{A}=t^{2}||\overrightarrow{AB}||_{A}^{2}.
\]

Taking the positive square root of both sides preceding equation gives Equation (\ref{eq:LEM00}).  

\medskip
 2) Definitions \ref{def:(-)Minuse.arrowAB} and \ref{def:MeasureOf-any-arrow}, along with Equation (\ref{eq:A01-1}), imply that 
\[
||\overrightarrow{BA}||_{A}^{2}  = \left<\overrightarrow{BA},\overrightarrow{BA}\right>_{A}
 =\left<-\;\overrightarrow{AB},-\;\overrightarrow{AB}\right>_{A}
 =(-1)\;(-1)\left<\overrightarrow{AB},\overrightarrow{AB}\right>_{A}
 =||\overrightarrow{AB}||_{A}^{2}.
\]
If we take the positive square root of both sides of the preceding equation we obtain Equation (\ref{eq:LEM11}).
\end{proof}

We end this section with a theorem which relates scalar multiplication with equality of arrows.
\begin{thm}
\label{thm:aABbAB}Let $\overrightarrow{AB}$ be an arrow in $\mathcal{P}_A$ such that
$A\neq_{P}B$. If $(a)\;\overrightarrow{AB}=_{A}(b)\;\overrightarrow{AB}$
for some real numbers $a$ and $b$, then we must have $a=b$.
\end{thm}

\begin{proof}
Let $\overrightarrow{AB}$ be an arrow such that $A\neq_{P}B$ and
$(a)\;\overrightarrow{AB}=_{A}(b)\;\overrightarrow{AB}$ for some real
numbers $a$ and $b$. If $a=0$, then it follows by Definition \ref{def:(Scalar-Multiplication-of}
that $(a)\;\overrightarrow{AB}=_{A}(0)\;\overrightarrow{AB}=_{A}\overrightarrow{AA} =
(b)\;\overrightarrow{AB}$. Since
$A\neq_{P}B$,
it follows again by Definition \ref{def:(Scalar-Multiplication-of} that
$b=0$, which means that $a=b$. Now suppose that $a\neq0$ and $b\neq0$.
Then Definition \ref{def:MeasureOf-any-arrow} implies that
\begin{equation}
||(a)\;\overrightarrow{AB}||_{A}^{2}=\left<(a)\;\overrightarrow{AB},(a)\;\overrightarrow{AB}\right>_{A}.\label{eq:m}
\end{equation}
Lemma \ref{lem:tAB} implies that $||(a)\;\overrightarrow{AB}||_{A}^{2}=|a|^{2}\;||\overrightarrow{AB}||_{A}^{2}$.
Also, it is given that $(a)\;\overrightarrow{AB}=_{A}(b)\;\overrightarrow{AB}$.
Thus, the Equation (\ref{eq:m}) can be rewritten as 
\[
|a|^{2}\;||\overrightarrow{AB}||_{A}^{2}=\left<(a)\;\overrightarrow{AB},(b)\;\overrightarrow{AB}\right>_{A}
= ab\;\left<\overrightarrow{AB},\overrightarrow{AB}\right>_{A}.
\]
where the last equality follows from by Equation (\ref{eq:A01}).
Since $\left<\overrightarrow{AB},\overrightarrow{AB}\right>_{A}=||\overrightarrow{AB}||_{A}^{2}$
(by Definition \ref{def:MeasureOf-any-arrow}), the above equation implies that $|a|^{2}=a^2 = ab$, or that
 $a=b$.
\end{proof}

\section{\label{sec:Comparing-Arrows-With}How Arrows Differ From Vectors}
For this section and the next we assume that we have an arrow space $\mathcal{P}_A$ 
and an arbitrary real vector space $V$.  We want to examine the differences between $\mathcal{P}_A$ and 
$V$ and conclude that arrows are different entities than vectors. 
As a case in point, we show that arrow addition is restricted
to certain pairs of arrows (the head of the first arrow must be the same as the
tail of the second arrow) and cannot be performed on any two arbitrary 
arrows of $\mathcal{P}_{A}$; see Proposition \ref{arrowadditionnotwelldefined}.
This is a fundamental difference from vector addition which is {\it always} defined for any two vectors in $V$.
We also show in Proposition \ref{prop:arr.add not comutv} that arrow addition is not commutative, 
a fact which draws a clear distinction
between arrows and vectors. More differences between arrows and vectors
will be further demonstrated.

\medskip We begin our analysis of the differences between arrows and vectors by recalling the well
known fact that for any vector
$v\in V$, it is always true that $(-1)\;v=-\;v$. However, this analog will not hold for arrows as
witnessed by the fact that the two expressions
$-\;\overrightarrow{AB}$ and $(-1)\;\overrightarrow{AB}$ in $\mathcal{P}_{A}$
are not the same. 

\begin{prop}
\label{prop:(-1)AB.NOT=00003D-AB}For any $A\neq_{P}B$, we have $-\;\overrightarrow{AB}\neq_{A}(-1)\;\overrightarrow{AB}$.
\end{prop}

\begin{figure}[ht]
\includegraphics[scale=0.5]{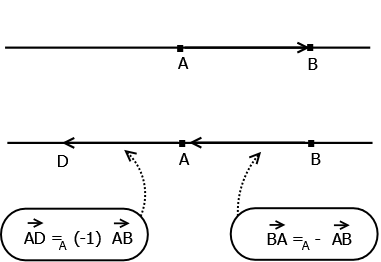}
\caption{ An illustration of Proposition \ref{prop:(-1)AB.NOT=00003D-AB}. Note that $(-1)\;\protect\overrightarrow{AB}$ and  $-\;\protect\overrightarrow{AB}$ have different heads and tails.}
\end{figure}

\begin{proof}
Definition \ref{def:(-)Minuse.arrowAB} implies that $-\;\overrightarrow{AB}=_{A}\overrightarrow{BA}$,
while Definition \ref{def:(Scalar-Multiplication-of} implies that $(-1)\;\overrightarrow{AB}=_{A}\overrightarrow{AD}$, where $D\in\mathcal{P}$. 
Since $A\neq_{P}B$, an application of Definition
\ref{def:equal arrows} shows that $\overrightarrow{BA}\neq_{A}\overrightarrow{AD}$,
which means that $-\;\overrightarrow{AB}\neq_{A}(-1)\;\overrightarrow{AB}$.
\end{proof}

\medskip
Next we turn our attention to the ways in which arrow addition
differs from that of vector addition.
The addition of arrows, as given by Definition \ref{def:ARROWSADDITIONHEADTAIL},
seems natural. If we have an arrow $\overrightarrow{AB}$
which has $B$ as its head, and at the same time $B$ is the tail of
another arrow $\overrightarrow{BC}$, then the resultant
arrow obtained from adding these two arrows is the arrow
that has its tail as the tail of the first arrow
and its head as the head of the second arrow, namely
$\overrightarrow{AC}$; see Figure 2.1.
But this natural definition has some drawbacks.
The first drawback is that arrow addition is not defined for arbitrary pairs
of arrows.

\begin{prop}\label{arrowadditionnotwelldefined}
The operation $+_{A}$ given by Definition \ref{def:ARROWSADDITIONHEADTAIL}
is not well defined for arbitrary elements of $\mathcal{P}_A\times \mathcal{P}_A$.
\end{prop}

\begin{proof}
This is an immediate consequence of Definition \ref{def:ARROWSADDITIONHEADTAIL} 
since $\overrightarrow{AB}+_{A}\overrightarrow{DC}$
is not well defined if $B\neq_{P}D$.
\end{proof}

Proposition \ref{arrowadditionnotwelldefined} demonstrates a fundamental difference between arrow addition and 
vector addition since vector addition is defined for any element of $V\times V$. Another fundamental difference between these
two operations involves commutativity.  Vector addition is commutative binary operation.
However, this is surprisingly not the case for arrow addition as demonstrated by the following proposition.
For this reason alone, arrows and vectors should not be thought of as a single notion.
\begin{prop}
\label{prop:arr.add not comutv}The operation $+_{A}$ given by Definition \ref{def:ARROWSADDITIONHEADTAIL} is not commutative.
\end{prop}


\begin{proof}
If $A$ and $B$ are two distinct points of $\mathcal{P}$, then Definition \ref{def:ARROWSADDITIONHEADTAIL}
implies that
\[
\overrightarrow{AB}+_{A}\overrightarrow{BA}=_{A}\overrightarrow{AA}
\qquad\text{and}\qquad
\overrightarrow{BA}+_{A}\overrightarrow{AB}=_{A}\overrightarrow{BB}.
\]
 If $\overrightarrow{AB}+_{A}\overrightarrow{BA}=_{A}\overrightarrow{BA}+_{A}\overrightarrow{AB}$,
then we would have $\overrightarrow{AA}=_{A}\overrightarrow{BB}$
which is impossible by Definition \ref{def:equal arrows} as $A\neq_{P}B$.
Therefore, $\overrightarrow{AB}+_{A}\overrightarrow{BA}\neq_{A}\overrightarrow{BA}+_{A}\overrightarrow{AB}$
implying that $+_{A}$ is noncommutative.
\end{proof}

\medskip 
We now turn our attention to differences between the corresponding operations of scalar multiplication.
 As was the case for arrow addition, we will discover that certain notions involving scalar multiplication in an arrow
space are not well defined.
First we show that the property $s(v+w) = sv + sw$, where $s \in \mathbb{R}$ and $v, w\in V$, does not necessarily
hold in $\mathcal{P}_A$.

\begin{thm}
Let $A,\;B$, and $C$ be three distinct points of $\mathcal{P}$.
Consider the two arrows $\overrightarrow{AB}$ and $\overrightarrow{BC}$
in $\mathcal{P}_{A}$. For any $s\neq1$, 
$(s)\;(\overrightarrow{AB}+_{A}\overrightarrow{BC})$ and $(s)\;\overrightarrow{AB}+_{A}(s)\;\overrightarrow{BC}$
are not equivalent. In particular, 
$(s)\;\overrightarrow{AB}+_{A}(s)\;\overrightarrow{BC}$ is not well defined.
\end{thm}

\begin{proof}
First, if $s=0$, then Definitions \ref{def:ARROWSADDITIONHEADTAIL} and
\ref{def:(Scalar-Multiplication-of} imply that
\begin{equation}
(s)\;(\overrightarrow{AB}+_{A}\overrightarrow{BC})=_{A}(0)\;\overrightarrow{AC}=_{A}\overrightarrow{AA}.\label{eq:E}
\end{equation}
On the other hand, let us ``evaluate" expression $(s)\;\overrightarrow{AB}+_{A}(s)\;\overrightarrow{BC}$.
Definition \ref{def:(Scalar-Multiplication-of} implies that
$(0)\;\overrightarrow{AB}=_{A}\overrightarrow{AA}$
and that $(0)\;\overrightarrow{BC}=_{A}\overrightarrow{BB}$.
Since $A\neq_{P}B$, $\overrightarrow{AA}+_{A}\overrightarrow{BB}$
is not well defined in the sense of Definition \ref{def:ARROWSADDITIONHEADTAIL}.

\medskip
 Now assume $s\neq0,\;1$.
 It follows by Definitions \ref{def:ARROWSADDITIONHEADTAIL} and
\ref{def:(Scalar-Multiplication-of} that
\begin{equation}
(s)\;(\overrightarrow{AB}+_{A}\overrightarrow{BC})=_{A}(s)\;\overrightarrow{AC}=_{A}\overrightarrow{AM},\label{eq:EE}
\end{equation}
for some points $C$ and $M$ in $\mathcal{P}$. On the other
hand, Definition \ref{def:(Scalar-Multiplication-of} implies that
that $(s)\;\overrightarrow{AB}=_{A}\overrightarrow{AD}$ for some point
$D$ in $\mathcal{P}$ with $D\neq_{P}B$ (as $s\neq1$; see Theorem 23), and that $(s)\;\overrightarrow{BC}=_{A}\overrightarrow{BE}$
for some point $E$ in $\mathcal{P}$. Thus, $(s)\;\overrightarrow{AB}+_{A}(s)\;\overrightarrow{BC}=_{A}\overrightarrow{AD}+_{A}\overrightarrow{BE}$,
and by Definition \ref{def:ARROWSADDITIONHEADTAIL}, as $D\neq_{P}B$, the right side is undefined. 
\end{proof}

\medskip
Next we show that the property $(s+t)v = sv + tv$, where $s,t \in \mathbb{R}$ and $v\in V$, also fails
to hold in $\mathcal{P}_A$.

\begin{thm}
Given an arrow $\overrightarrow{AB}$ in $\mathcal{P}_A$ with $A\neq_{P}B$, let $s,\;t\in\mathbb{R}$
with $s\neq0$. The two expressions $(s+t)\;(\overrightarrow{AB})$
and $(s)\;\overrightarrow{AB}+_{A}(t)\;\overrightarrow{AB}$ are not equivalent. However, if $s=0$, then $(s+t)\;(\overrightarrow{AB})=_{A}(s)\;\overrightarrow{AB}+_{A}(t)\;\overrightarrow{AB}$.
\end{thm}

\begin{proof}
First assume that $s=0$. Definition \ref{def:(Scalar-Multiplication-of} implies that
\begin{equation}
(t)\overrightarrow{AB}=_{A}\overrightarrow{AD}\label{eq:LAB}
\end{equation}
for some point $D$ in $\mathcal{P}$ and that
\begin{equation}
(s)\overrightarrow{AB}=_{A}\overrightarrow{AA}.\label{eq:LABS}
\end{equation}
Equations (\ref{eq:LAB}) and (\ref{eq:LABS}), along with Definition 5, imply that implies that 
\[
(s+t)\;(\overrightarrow{AB})=_{A}(t)\overrightarrow{AB}=_{A}\overrightarrow{AD}
=_{A} \overrightarrow{AA} +_{A} \overrightarrow{AD} = (s)\;(\overrightarrow{AB}) +_{A}(t)\overrightarrow{AB}.
\]

\medskip
Now suppose that $s\neq0$ and let $(s)\overrightarrow{AB}=_{A}\overrightarrow{AC}$ and
$(t)\;\overrightarrow{AB}=_{A}\overrightarrow{AD}$
for some points $C$ and $D$ in $\mathcal{P}$, with $C\neq_{P}A$.
The right hand side of the expression
\begin{equation}
(s)\overrightarrow{AB}+_{A}(t)\overrightarrow{AB}=_{A}\overrightarrow{AC}+_{A}\overrightarrow{AD},\label{eq:AX}
\end{equation}
is undefined according to Definition \ref{def:ARROWSADDITIONHEADTAIL}. 
On the other hand, if we put $s+t=r$ for some
real number $r$, then Definition \ref{def:(Scalar-Multiplication-of} implies
that

\begin{equation}
(s+t)\;(\overrightarrow{AB})  =_{A}(r)\overrightarrow{AB} =_{A}\overrightarrow{AE},\label{eq:AXL}
\end{equation}
where $E$ is a point in $\mathcal{P}$. Thus the expression in Equation (\ref{eq:AXL}) is well defined
while the right hand side of the
Equation (\ref{eq:AX}) is not well defined which makes it impossible to compare
the two left hand sides of these equations.
\end{proof}

\medskip
\begin{cor}
For any arrow $\overrightarrow{AB}$ in $\mathcal{P}_A$ with $A\neq_{P}B$,
$\overrightarrow{AB}+_{A}(-1)\;\overrightarrow{AB}$ is not well defined.
\end{cor}

\section{\label{sec:Section4}Similarities Between Arrow Spaces and Vector Spaces}
So far we have investigated some axioms of vector space $V$ which do not hold
in $\mathcal{P}_{A}$.  In this section we turn our attention to those axioms of $V$ which 
have corresponding analogs in $\mathcal{P}_A$. We first show
that arrow addition $+_{A}$ is associative.
\begin{thm}
\label{thm:ARR.ASSO}Let $\overrightarrow{AB},\;\overrightarrow{BC},\;\overrightarrow{CD}$
be arrows in $\mathcal{P}_{A}$. Then 
\[
(\overrightarrow{AB}+_{A}\overrightarrow{BC})+_{A}\overrightarrow{CD}=_{A}\overrightarrow{AB}+_{A}(\overrightarrow{BC}+_{A}\overrightarrow{CD}),
\]
 that is the addition $+_{A}$ is associative.
\end{thm}
\begin{figure}[ht]
\centering
\includegraphics[width=0.8\linewidth]{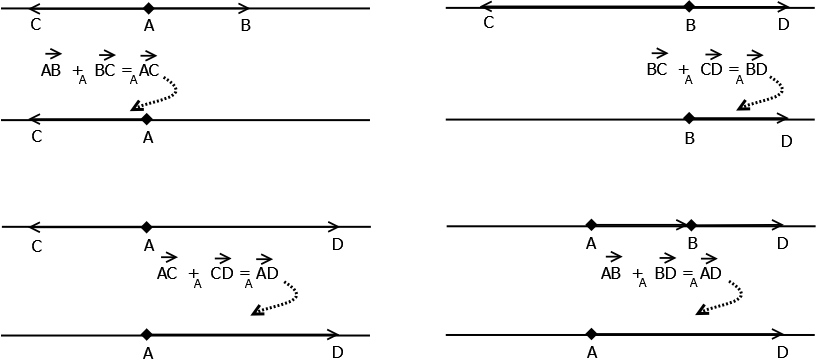}
\caption{An illustration of Theorem \ref{thm:ARR.ASSO}. The left section of this figure explains $(\protect\overrightarrow{AB}+_{A}\protect\overrightarrow{BC})+_{A}\protect\overrightarrow{CD}$, while the right one stands for $\protect\overrightarrow{AB}+_{A}(\protect\overrightarrow{BC}+_{A}\protect\overrightarrow{CD})$.}
\label{fig:THEOREM18}
\end{figure}

\begin{proof}
Given three arrows $\overrightarrow{AB},\;\overrightarrow{BC}$, and
$\overrightarrow{CD}$ in $\mathcal{P}_{A}$, Definition \ref{def:ARROWSADDITIONHEADTAIL} implies that
\begin{equation}
(\overrightarrow{AB}+_{A}\overrightarrow{BC})+_{A}\overrightarrow{CD}=_{A}\overrightarrow{AC}+_{A}\overrightarrow{CD} 
=_{A}\overrightarrow{AD}.\label{eq:*0}
\end{equation}
A similar calculation shows that
\begin{equation}
\overrightarrow{AB}+_{A}(\overrightarrow{BC}+_{A}\overrightarrow{CD}) =_{A}\overrightarrow{AB}+_{A}\overrightarrow{BD}\nonumber \\
 =_{A}\overrightarrow{AD}.\label{eq:*1}
\end{equation}
Combining the Equations (\ref{eq:*0}) and (\ref{eq:*1}) finishes the proof.
\end{proof}

\medskip
A similar result regarding occurs for arrow scalar multiplication.
But in order to prove the associativity of arrow scalar multiplication we will need the following proposition.
This proposition is important in its own right since it will be crucial to proving the existence of line between any two points; see Subsection 5.1.
\begin{prop}
\label{lem:ABAD0}Let $\overrightarrow{AB}$ and $\overrightarrow{AD}$
be two arrows of $\mathcal{P}_{A}$ such that 
$||\overrightarrow{AB}||_{A}=||\overrightarrow{AD}||_{A}$ and $\frac{\left<\overrightarrow{AB},\overrightarrow{AD}\right>_{A}}{||\overrightarrow{AB}||_{A}\;||\overrightarrow{AD}||_{A}}=1$.
Then $B=_{P}D$.
\end{prop}

\begin{proof}
Assume by way
of contradiction that $B\neq_{P}D$. By hypothesis
we have $||\overrightarrow{AB}||_{A}=||\overrightarrow{AD}||_{A}$
and $\frac{\left<\overrightarrow{AB},\overrightarrow{AD}\right>_{A}}{||\overrightarrow{AB}||_{A}\;||\overrightarrow{AD}||_{A}}=1$,
which implies that
\begin{equation}
\left<\overrightarrow{AB},\overrightarrow{AD}\right>_{A}=\left<\overrightarrow{AD},\overrightarrow{AD}\right>_{A}.\label{eq:T}
\end{equation}
An application of Axiom 1 and Definition
\ref{def:(-)Minuse.arrowAB} to Equation (\ref{eq:T}) yields
\[
\left<\overrightarrow{DA},\overrightarrow{AD}\right>_{A}+\left<\overrightarrow{AB},\overrightarrow{AD}\right>_{A} 
= \left<\overrightarrow{DB},\overrightarrow{AD}\right>_{A}=0.
\]
Since $\overrightarrow{AD}=_{A}-\overrightarrow{DA}$ (see Definition
\ref{def:(-)Minuse.arrowAB}), the preceding equation is equivalent to
\begin{equation}
\left<\overrightarrow{DB},\overrightarrow{DA}\right>_{A}=0.\label{eq:TE}
\end{equation}
On the other hand, since $||\overrightarrow{AB}||_{A}=||\overrightarrow{AD}||_{A}$
we can rewrite Equation (\ref{eq:T}) as $\left<\overrightarrow{AB},\overrightarrow{AD}\right>_{A}=\left<\overrightarrow{AB},\overrightarrow{AB}\right>_{A}$.
Then a similar argument shows that

\begin{equation}
\left<\overrightarrow{DB},\overrightarrow{AB}\right>_{A}=0.\label{eq:TA}
\end{equation}
Adding Equations (\ref{eq:TE}) and (\ref{eq:TA}) and using Axiom 1 gives us
\begin{equation}
\left<\overrightarrow{DB},\overrightarrow{DA}\right>_{A}+\left<\overrightarrow{DB},\overrightarrow{AB}\right>_{A} =
\left<\overrightarrow{DB},\overrightarrow{DB}\right>_{A}=0.\label{eq:TF}
\end{equation}
By Definition \ref{def:MeasureOf-any-arrow}, Equation (\ref{eq:TF})
means that $||\overrightarrow{DB}||_{A}^{2}=0$, that is $||\overrightarrow{DB}||_{A}=0$.
But by Lemma \ref{lem:||AB||=00003D0 IFF}, the equation $||\overrightarrow{DB}||_{A}=0$
implies that $D=_{P}B$ which contradicts our assumption that $B\neq_{P}D$.
Thus we conclude that $B=_{P}D$.
\end{proof}



 
 \begin{thm}\label{scalarmultasscocthm}
For any $s,\;t\in\mathbb{R}$ and any arrow $\overrightarrow{AB}$ in $\mathcal{P}_A$,
 $(s\;t)\overrightarrow{AB}=_{A}(s)((t)\;\overrightarrow{AB})$,
that is the arrow scalar multiplication is associative.
\end{thm}

\begin{proof}
If $A=_{P}B$ or at least one of the real numbers $s$ or $t$ is zero,
then it follows directly from Definition \ref{def:(Scalar-Multiplication-of}
that $(st)\;\overrightarrow{AB}=_{A}(s)\;((t)\;\overrightarrow{AB})=_{A}\overrightarrow{AA}.$
Now let us assume $A\neq_{P}B$ and $s,t\neq0$. Let 
\begin{equation}
(t)\;\overrightarrow{AB}  =_{A}\overrightarrow{AC},\label{eq:st(AB)a}
\end{equation}
\begin{equation}
(s)\;((t)\;\overrightarrow{AB})  =_{A}(s)\;\overrightarrow{AC}=_{A}\overrightarrow{AD},\label{eq:st(AB)b}
\end{equation}
and
\begin{equation}
(s\;t)\;\overrightarrow{AB}  =_{A}\overrightarrow{AF},\label{eq:st(AB)c}
\end{equation}
for some points $C,\;D$, and $F$ in $\mathcal{P}$. We want
to show that $\overrightarrow{AD}=_{A}\overrightarrow{AF}$. By
Definition \ref{def:(Scalar-Multiplication-of} and Equations
(\ref{eq:st(AB)a}) and (\ref{eq:st(AB)b}) we have
\[
||\overrightarrow{AC}||_{A}=|t|\;||\overrightarrow{AB}||_{A},\qquad
||\overrightarrow{AD}||_{A}=|s|\;||\overrightarrow{AC}||_{A},
\]
which implies that
\begin{equation}
||\overrightarrow{AD}||_{A}=|s\;t|\;||\overrightarrow{AB}||_{A}.\label{eq:st(AB)}
\end{equation}
Also, Definition \ref{def:(Scalar-Multiplication-of} and Equation
(\ref{eq:st(AB)c}) imply that
\begin{equation}
||\overrightarrow{AF}||_{A}=|s\;t|\;||\;\overrightarrow{AB}||_{A}.\label{ANAN}
\end{equation}
We get from Equations (\ref{eq:st(AB)}) and (\ref{ANAN}) that
\begin{equation}
||\overrightarrow{AF}||_{A}=||\overrightarrow{AD}||_{A}.\label{eq:st(AB)=00003D5}
\end{equation}
Next we show that $\left<\frac{\overrightarrow{AD}}{||\overrightarrow{AD}||_{A}},\frac{\overrightarrow{AF}}{||\overrightarrow{AF}||_{A}}\right>_{A}=1$.
By Equations (\ref{eq:st(AB)b}) and (\ref{eq:st(AB)c}), and since $||\overrightarrow{AF}||_{A}=||\overrightarrow{AD}||_{A}=|s\;t|\;||\overrightarrow{AB}||_{A}$, we
have 
\[
\left<\frac{\overrightarrow{AD}}{||\overrightarrow{AD}||_{A}},\frac{\overrightarrow{AF}}{||\overrightarrow{AF}||_{A}}\right>_{A}=\left<\frac{s\;((t)\;\overrightarrow{AB})}{|s\;t|\;||\overrightarrow{AB}||_{A}},\frac{(s\;t)\;\overrightarrow{AB}}{|s\;t|\;||\overrightarrow{AB}||_{A}}\right>_{A}.
\]
Using Equation (\ref{eq:A01}) and Definition \ref{def:MeasureOf-any-arrow}, the preceding equation simplifies to 
\begin{equation}
\left<\frac{\overrightarrow{AD}}{||\overrightarrow{AD}||_{A}},\frac{\overrightarrow{AF}}{||\overrightarrow{AF}||_{A}}\right>_{A}=\frac{(s\;t)^{2}\;\left<\overrightarrow{AB},\overrightarrow{AB}\right>_{A}}{|s\;t|^{2}\;||\overrightarrow{AB}||_{A}^{2}}=1,\label{eq:st(AB)=00003D6}
\end{equation}
Then Proposition \ref{lem:ABAD0}
implies that $D=_{P}F$ as desired.
\end{proof}


\medskip
Recall that vector space $V$ has a unique vector $0$ which is the identity element with respect to addition, i.e. 
$v + 0 = v = 0 + v = v$ for all $v\in V$.
If we think of $\overrightarrow{AA}$ as an analog to the
zero vector, then the following theorem implies that
$\overrightarrow{AA}$ is a left additive identity for arrows that
have the point $A$ as a tail.
\begin{thm}
Given an arrow $\overrightarrow{AB}$ of $\mathcal{P}_{A}$, there is a 
a unique arrow $\overrightarrow{AA}$ (similarly, a unique arrow $\overrightarrow{BB}$)
such that 
\begin{equation}
\overrightarrow{AA}+_{A}\overrightarrow{AB}=_{A}\overrightarrow{AB},\label{eq:IDE}
\end{equation}
 (and similarly $\overrightarrow{AB}+_{A}\overrightarrow{BB}=_{A}\overrightarrow{AB}$).\label{thm:fg}
\end{thm}

\begin{proof}
Let $\overrightarrow{AB}$ be an arrow in $\mathcal{P}_{A}$.
Since $\overrightarrow{AA}$ is well defined by Definition
\ref{def:ArrowDef.}, Definition \ref{def:ARROWSADDITIONHEADTAIL}
implies that $\overrightarrow{AA}+_{A}\overrightarrow{AB}=_{A}\overrightarrow{AB}.$
Now if there exists another arrow $\overrightarrow{CD}$ such that 
\begin{equation}
\overrightarrow{CD}+_{A}\overrightarrow{AB}=_{A}\overrightarrow{AB},\label{eq:TEL}
\end{equation}
then since $\overrightarrow{AB}$ is fixed,
Definition \ref{def:ARROWSADDITIONHEADTAIL} implies that
$D=_{P}A$.
Thus Equation (\ref{eq:TEL}) becomes
\begin{equation}
\overrightarrow{CA}+_{A}\overrightarrow{AB}=_{A}\overrightarrow{CB}.\label{eq:IDEN}
\end{equation}
Since Equations (\ref{eq:TEL}) and (\ref{eq:IDEN}) are equivalent,
comparing them yields $\overrightarrow{CB}=_{A}\overrightarrow{AB}$,
which means (by Definition \ref{def:equal arrows}) that we must have $C=_{P}A$.
Therefore, $\overrightarrow{CD}=_{A}\overrightarrow{AA}$.
\end{proof}

\medskip
For any $v\in V$, there is a unique element $-v$ (the additive inverse of $v$) such that
$v + (-v) =  -v + v = 0$. The following theorem is the arrow space analog of the additive inverse.
\begin{thm}
Let $\overrightarrow{AA}$ be fixed in $\mathcal{P}_{A}$.
For any arrow $\overrightarrow{AB}$ in $\mathcal{P}_{A}$ there
exists a unique arrow $\overrightarrow{BA}$ such that 
\begin{equation}
\overrightarrow{AB}+_{A}\overrightarrow{BA}=_{A}\overrightarrow{AA}.\label{eq:INV} 
\end{equation}.\label{thm:fgh}
\end{thm}
\begin{proof}
Let $\overrightarrow{AA}$ be fixed. For any point $B\in\mathcal{P}$,
 $\overrightarrow{AB}$ is in $\mathcal{P}_{A}$. By
Definition \ref{def:ArrowDef.}, $\overrightarrow{BA}$
is also in $\mathcal{P}_{A}$, and Definition \ref{def:ARROWSADDITIONHEADTAIL}
implies that
\[
\overrightarrow{AB}+_{A}\overrightarrow{BA}=_{A}\overrightarrow{AA}.
\]
Now if exists another arrow $\overrightarrow{CD}$ such that 
\begin{equation}
\overrightarrow{AB}+_{A}\overrightarrow{CD}=_{A}\overrightarrow{AA},\label{eq:KA}
\end{equation}
then since $\overrightarrow{AB}$ is fixed,
the only way for $\overrightarrow{AB}+_{A}\overrightarrow{CD}$
to be defined is to have $C=_{P}B$.
Thus if we change $C$ into $B$ in the Equation (\ref{eq:KA}) and
use Definition \ref{def:ARROWSADDITIONHEADTAIL}, we would have
\begin{equation}
\overrightarrow{AB}+_{A}\overrightarrow{BD}=_{A}\overrightarrow{AD}.\label{eq:IDEN-1}
\end{equation}
Since Equations (\ref{eq:KA}) and (\ref{eq:IDEN-1}) are equivalent,
comparing them yields $\overrightarrow{AD}=_{A}\overrightarrow{AA}$
which means ( by Definition \ref{def:equal arrows}), that 
$D=_{P}A$.
Thus, $\overrightarrow{CD}=_{A}\overrightarrow{BA}$.
\end{proof}

\medskip
The following proposition, the arrows space analog of the vector space property $s0 = 0$, is a direct consequence of Definition
\ref{def:(Scalar-Multiplication-of}. 
\begin{prop}
\label{prop:sAA}For any arrow $\overrightarrow{AA}$ in $\mathcal{P}_{A}$ and any $s\in\mathbb{R}$, $(s)\overrightarrow{AA}=_{A}\overrightarrow{AA}$.
\end{prop}


\medskip
In a vector space $V$, given $v\in V$, there is a unique scalar, denoted $1$, such that $1v = v$.
The scalar multiplication identity in an arrow space
$\mathcal{P}_{A}$ can also be identified as follows.
\begin{thm}
\label{thm:1 AB=00003DAB}For any arrow $\overrightarrow{AB}$ in
an arrow space $\mathcal{P}_{A}$, $(1)\overrightarrow{AB}=_{A}\overrightarrow{AB}$.
\end{thm}

\begin{proof}
If $A=_{P}B$, then the result follows from Proposition \ref{prop:sAA}.
So assume $A\neq_{P}B$. Lemma \ref{lem:||AB||=00003D0 IFF}
implies that $||\overrightarrow{AB}||_{A}\neq0$. 
Definition \ref{def:(Scalar-Multiplication-of} implies that $(1)\overrightarrow{AB} = \overrightarrow{AD}$ for some point $D$ such that
\begin{equation}
||\overrightarrow{AD}|| = ||(1)\overrightarrow{AB}||_{A}=||\overrightarrow{AB}||_{A},\label{eq:D01}
\end{equation}
and
\begin{equation}
\left<\frac{\overrightarrow{AD}}{||\overrightarrow{AD}||},\;\frac{\overrightarrow{AB}}{||\overrightarrow{AB}||_{A}}\right>_{A} = 1.\label{eq:D02}
\end{equation}
From Equations (\ref{eq:D01}), and (\ref{eq:D02}) we would like to conclude that $D=_{P}B$.
But this will follow from the uniqueness of the point $D$ as referenced by the remark that followed 
Definition \ref{def:(Scalar-Multiplication-of}.
\end{proof}

\medskip
In a vector space $V$, multiplication by the scalar $0$ always returns the zero vector, i.e. $0v = 0$ for any $v\in V$.
We end this section with the following proposition which provides the arrow space analog, where arrows of the form $\overrightarrow{AA}$
correspond to the zero vector.

\begin{prop}
\label{prop:zeroAB}For any arrow $\overrightarrow{AB}$ in $\mathcal{P}_{A}$, $(0)\overrightarrow{AB}=_{A}\overrightarrow{AA}$.
\end{prop}

\begin{proof}
Let $\overrightarrow{AB}$ be an arrow in $\mathcal{P}_{A}$. 
By Definition \ref{def:(Scalar-Multiplication-of}, $(0)\;\overrightarrow{AB}=_{A}\overrightarrow{AA}$.
\end{proof}

\section{\label{sec:4.Eq.Clss}Lines in an Arrow Space}

The line is a crucial concept of Euclidean geometry since it appears within three of Euclid's original five axioms.  
Because an arrow space is to be the precursor of the Euclidean plane reinterpreted 
as a two-dimensional vector space, we will need to define what is meant by a line $l$ within 
an abstract arrow space $\mathcal{P}_A$.  For any fixed 
$\overrightarrow{AB}\in \mathcal{P}_A$, we can use scalar multiplication to define a set of points with $\mathcal{P}_A$ 
determined by the ``direction" of $\overrightarrow{AB}$. This aforementioned set of points in $\mathcal{P}_A$ is the line $l$. 
This set of points on $l$, namely $\mathcal{P}_{l}$, gives rise to a one-dimensional arrow space $\mathcal{P}_{l_A}$ contained within $\mathcal{P}_{A}$.
When considered as an independent arrow space apart from $\mathcal{P}_{A}$, the arrow space $\mathcal{P}_{l_A}$ has the property that parallel postulate of the Euclidean plane is no longer an axiom but a theorem; see Subsection 5.3.  This surprising result justifies an analysis of $\mathcal{P}_{l_A}$ in its own right and is focus of the two subsections contained within this section. But all of this analysis hinges on the definition and existence of a line within an arrow space. Thus we begin with this pivotal definition.

\begin{defn}
\label{def:A LINE} Let $\mathcal{P}_A$ be a given arrow space.
Given a fixed point $A\in \mathcal{P}$ and an arrow $\overrightarrow{AB}\in \mathcal{P}_A$
with $A\neq_{P}B$, a line containing $A$ and $B$, denoted by $l_{AB}$ or $l$, 
is the set of all the heads of arrows $\overrightarrow{AC}$
such that there exists $t\in\mathbb{R}$ with $\overrightarrow{AC}=_{A}(t)\;\overrightarrow{AB}$.
The set of points on Line $l$ is denoted by $\mathcal{P}_{l_{AB}}$ or $\mathcal{P}_{l}$ when
$\overrightarrow{AB}$ is understood. 
The set of arrows determined by $\mathcal{P}_l$ is the arrow 
space $\mathcal{P}_{l_A}$.
\end{defn}


\medskip
To justify the existence of a line between any two distinct points $A$ and $B$ of $\mathcal{P}$,
we consider the next axiom which matches points on a line with
the real numbers.

\medskip 
\noindent
{\bf Axiom 3}. Let $l$ be a line containing the points $A$ and $B$. 
There exists a one-to-one correspondence between the set of real numbers and the set of points
that lie on $l$.

\medskip
Axiom 3 shows that Line $l_{AB}$ is a nonempty set of $\mathcal{P}$; see the proof of Theorem \ref{thm:WAS.Axiom (3)}.
Furthermore, since the real line is a totally ordered set, Axiom 3 provides a geometric intuition of order (or betweenness) of points in $\mathcal{P}_l$ 
associated with Line $l$. The official algebraic definition of this geometric intuition is provided below.

\medskip
\begin{defn}
\label{def:orderof2points}Let $A,\;B$, and $C$ be any three distinct
points on Line $l$. We say that the point $B$ is between $A$ and $C$
if $\left<\frac{\overrightarrow{BA}}{||\overrightarrow{BA}||_{A}},\frac{\overrightarrow{BC}}{||\overrightarrow{BC}||_{A}}\right>_{A}=-1$.
If $B$ is between $A$ and $C$, we write $A\prec B\prec C$.
\end{defn}
\begin{figure}[ht]
\centering
\includegraphics[width=0.7\linewidth]{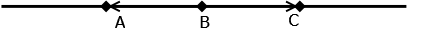}
\caption{A depiction of $A\prec B\prec C$, where $A$, $B$, and $C$ are three points on Line $l$.}
\label{fig:Definition26}
\end{figure}

\medskip
The notion of betweenness provided by Definition \ref{def:orderof2points}, along with the correspondence between the 
real line and the line $l$ determined by $\overrightarrow{AB}$, allows us to illustrate an abstract arrow $\overrightarrow{AB}$ 
as a directed finite line segment which starts at the Tail $A$ and ends at Head $B$, where the direction is recorded by the 
placement of an arrowhead on the head point $B$.  We have been using this illustration convention throughout the previous 
sections even though we had no concept of betweenness and an arrow $\overrightarrow{AB}$, in all rigor, should have 
been depicted as two discrete points $A$ and $B$.

\subsection{Existence of Line}

We are interested in proving an arrow space analog of Euclid's first axiom, namely that given any two distinct points, there exists a {\it unique} line which
contains the two given points; see Theorem \ref{thm:WAS.Axiom (3)}. Before we do this, we introduce
a lemma which is used in the proof of Theorem \ref{thm:WAS.Axiom (3)}.
This lemma provides an arrow pre-inner product constraint for two unit arrows which lie on the same line, 
namely that the unit arrows have the same direction (pre-inner product is $1$) or that the unit arrows have opposite directions (pre-inner product is $-1$).
\begin{lem}
\label{lem:LAAALM}If $A$, $B$, $M$, and $L$ are four distinct points of Line $l_{AB}$, then
\begin{equation}
\frac{\left<\overrightarrow{ML},\overrightarrow{AB}\right>_{A}}{||\overrightarrow{ML}||_{A}\;||\overrightarrow{AB}||_{A}}=\pm1.\label{eq:LAAAM}
\end{equation}
\end{lem}

\begin{figure}[ht]
\centering
\includegraphics[width=0.7\linewidth]{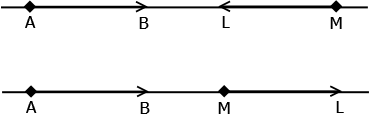}
\caption{Points on Line $l_{AB}$ illustrating Equation (\ref{eq:LAAAM}).}
\label{fig:Lemma33}
\end{figure}

\begin{proof}
Let $A$, $B$, $M$, and $L$ be four distinct points of line $l_{AB}$. By Definition \ref{def:A LINE} there exists
$t_{M},\;t_{L}\in\mathbb{R}$ such that 
\begin{equation}
\overrightarrow{AM}  =_{A}(t_{M})\overrightarrow{AB},\qquad \overrightarrow{AL}  =_{A}(t_{L})\overrightarrow{AB},\label{eq:Epua}
\end{equation}
which when combined with Definition \ref{def:(-)Minuse.arrowAB} implies that
\begin{equation}
\overrightarrow{MA}=_{A}-((t_{M})\overrightarrow{AB}).\label{eq:MATMA}
\end{equation}
Since $\overrightarrow{MA}+_{A}\overrightarrow{AL}=_{A}\overrightarrow{ML}$, we find through various applications of Axioms 1 and 2 that

\begin{equation}\label{eq:INNABML}
\frac{\left<\overrightarrow{ML},\overrightarrow{AB}\right>_{A}}{||\overrightarrow{ML}||_{A}\;||\overrightarrow{AB}||_{A}}
=\frac{(-t_M + t_L)\left<\overrightarrow{AB},\overrightarrow{AB}\right>_{A}}{||\overrightarrow{MA}+_{A}\overrightarrow{AL}||_{A}\;||\overrightarrow{AB}||_{A}}.
\end{equation}
We can also simplify the denominator of 
Equation (\ref{eq:INNABML}) in a similar manner, namely
\begin{equation}
||\overrightarrow{MA} +_A \overrightarrow{AL}||_A =  \sqrt{(t_{L}-t_{M})^{2}}\;||\overrightarrow{AB}||_{A}.\label{eq:TSQU}
\end{equation}
Equation (\ref{eq:TSQU}) implies that 
\begin{equation}
||\overrightarrow{MA}+_{A}\overrightarrow{AL}||_{A}\;||\overrightarrow{AB}||_{A}=|t_{L}-t_{M}|\;||\overrightarrow{AB}||_{A}^{2}.\label{eq:DENOM}
\end{equation}
Plugging the result of Equation (\ref{eq:DENOM}),
(notice that $t_{L}-t_{M}\neq0$ since otherwise we would have $L=_{P}M$
which is a contradiction to the hypotheses of this lemma), into the
right hand side of the Equation (\ref{eq:INNABML}) yields
\[
\frac{\left<\overrightarrow{ML},\overrightarrow{AB}\right>_{A}}{||\overrightarrow{ML}||_{A}\;||\overrightarrow{AB}||_{A}}=\frac{(t_{L}-t_{M})\;||\overrightarrow{AB}||_{A}^{2}}{|t_{L}-t_{M}|\;||\overrightarrow{AB}||_{A}^{2}}=\pm1.
\]
\end{proof}

\medskip
We next provide a pre-inner product condition for determining when a given point lies on a fixed line; see Theorems \ref{thm:ONETWO} and \ref{thm:ONETWO1}. Since the set of points $\mathcal{P}_l$ of $l_{AB}$ is
contained within underlying the set of points $\mathcal{P}$ of $\mathcal{P}_A$, we can examine the following situation.
Suppose for some point $D\in\mathcal{P}$ distinct from $A$ and $B$, we have $\frac{\left<\overrightarrow{AB},\overrightarrow{AD}\right>_{A}}{||\overrightarrow{AB}||_{A}\;||\overrightarrow{AD}||_{A}}=1$,
then $D\in\mathcal{P}_l$, see Figure 5.2.
\begin{figure}[ht]\label{fig:Points-on-a}
\includegraphics[scale=0.45]{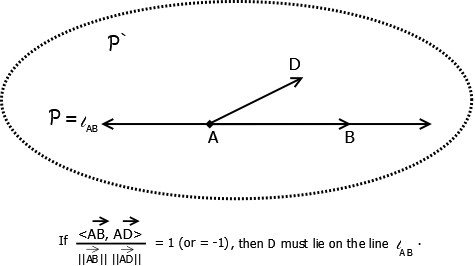}
\caption{An illustration for when a given point lies on a fixed line.}
\end{figure}

\begin{thm}
\label{thm:ONETWO}Let $A,\;B,\;D$ be three distinct points of $\mathcal{P}$
such that
$\frac{\left<\overrightarrow{AB},\overrightarrow{AD}\right>_{A}}{||\overrightarrow{AB}||_{A}\;||\overrightarrow{AD}||_{A}}=1$.
Then the following is true:
\begin{itemize}
\item[1.]
 \[
 \frac{\left<\overrightarrow{AB},\overrightarrow{BD}\right>}{||\overrightarrow{AB}||_{A}\;||\overrightarrow{BD}||_{A}}=\pm1;\]

\item[2.] either 
\[
||\overrightarrow{AD}||_{A} = ||\overrightarrow{AB}||_{A}+||\overrightarrow{BD}||_{A},\]
\begin{figure}[ht]
\centering
\includegraphics[width=0.5\linewidth]{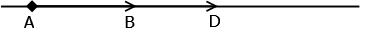}
\caption{A depiction of three points $A,\;B,\;D$ where $||\protect\overrightarrow{AD}||_{A} = ||\protect\overrightarrow{AB}||_{A}+||\protect\overrightarrow{BD}||_{A}$.}
\label{fig:THEOREM361}
\end{figure}

or
\[
||\overrightarrow{AB}||_{A} = ||\overrightarrow{AD}||_{A}+||\overrightarrow{DB}||_{A};
\]
\begin{figure}[ht]
\centering
\includegraphics[width=0.5\linewidth]{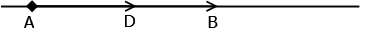}
\caption{A depiction of three points $A,\;B,\;D$ when $||\protect\overrightarrow{AB}||_{A} = ||\protect\overrightarrow{AD}||_{A}+||\protect\overrightarrow{DB}||_{A}$.}
\label{fig:THEOREM362}
\end{figure}

\item[3.]
there exists $t_{D}\in\mathbb{R}$ with $t_D > 0$ such that $\overrightarrow{AD}=_{A}(t_{D})\overrightarrow{AB}$;
thus the point $D$ lies on the line $l_{AB}$; see Figure 5.4.
\end{itemize}
\end{thm}

\begin{proof}
1) First notice that it is not given that $A,\;B$, and $D$ lie
on the same line; in fact, we need to prove this. So let $A,\;B$, and $D$ be three distinct
points of $\mathcal{P}$ such that 
\begin{equation}
\frac{\left<\overrightarrow{AB},\overrightarrow{AD}\right>_{A}}{||\overrightarrow{AB}||_{A}\;||\overrightarrow{AD}||_{A}}=1,\label{eq:A}
\end{equation}
Since $\overrightarrow{AD}=_{A}\overrightarrow{AB}+_{A}\overrightarrow{BD}$, applications of Axiom 1 and Definition \ref{def:MeasureOf-any-arrow} to 
Equation (\ref{eq:A}) implies that
\[
1 = \frac{\left<\overrightarrow{AB},\overrightarrow{AB}+_{A}\overrightarrow{BD}\right>_{A}}{||\overrightarrow{AB}||_{A}\;||\overrightarrow{AD}||_{A}}\
=\frac{||\overrightarrow{AB}||_{A}^{2}+\left<\overrightarrow{AB},\overrightarrow{BD}\right>_{A}}{||\overrightarrow{AB}||_{A}\;||\overrightarrow{AD}||_{A}}.
\]
Multiplying both sides in the above by $||\overrightarrow{AB}||_{A}\;||\overrightarrow{AD}||_{A}$
yields
\begin{equation}
||\overrightarrow{AB}||_{A}^{2}+\left<\overrightarrow{AB},\overrightarrow{BD}\right>_{A}=||\overrightarrow{AB}||_{A}\;||\overrightarrow{AD}||_{A}.\label{eq:FLD}
\end{equation}
Similar calculations show that  
\begin{equation}
||\overrightarrow{AD}||_{A}  
=\sqrt{||\overrightarrow{AB}||_{A}^{2}+2\left<\overrightarrow{AB},\overrightarrow{BD}\right>_{A}+||\overrightarrow{BD}||_{A}^{2}}.\label{eq:FLO}
\end{equation}
If we combine Equations (\ref{eq:FLD}) and (\ref{eq:FLO}),
we get
\begin{equation}
||\overrightarrow{AB}||_{A}^{2}+\left<\overrightarrow{AB},\overrightarrow{BD}\right>_{A}=||\overrightarrow{AB}||_{A}\;\sqrt{||\overrightarrow{AB}||_{A}^{2}+2\left<\overrightarrow{AB},\overrightarrow{BD}\right>_{A}+||\overrightarrow{BD}||_{A}^{2}}.\label{eq:ssass}
\end{equation}
By squaring both sides of the above, we get 
\begin{align*}
||\overrightarrow{AB}||_{A}^{4}&+2||\overrightarrow{AB}||_{A}^{2}\;\left<\overrightarrow{AB},\overrightarrow{BD}\right>_{A}
+\left<\overrightarrow{AB},\overrightarrow{BD}\right>^{2}\\
&=||\overrightarrow{AB}||_{A}^{4}+2||\overrightarrow{AB}||_{A}^{2}\;\left<\overrightarrow{AB},\overrightarrow{BD}\right>_{A}
+||\overrightarrow{AB}||_{A}^{2}\;||\overrightarrow{BD}||_{A}^{2},
\end{align*}
which is equivalent to
\[
\left<\overrightarrow{AB},\overrightarrow{BD}\right>^{2}=||\overrightarrow{AB}||_{A}^{2}\;||\overrightarrow{BD}||_{A}^{2}.
\]
Taking the square root of both sides of the above gives 
\begin{equation}
\left<\overrightarrow{AB},\overrightarrow{BD}\right>=\pm||\overrightarrow{AB}||_{A}\;||\overrightarrow{BD}||_{A}.\label{eq:FLJ}
\end{equation}
Since $A\neq_{P}B,\;B\neq_{P}D$, it follows by Lemma \ref{lem:||AB||=00003D0 IFF}
that $||\overrightarrow{AB}||_{A}\neq0$ and that $||\overrightarrow{BD}||_{A}\neq0$.
Hence, if we divide both sides of the Equation (\ref{eq:FLJ}) by
$||\overrightarrow{AB}||_{A}\;||\overrightarrow{BD}||_{A}$, we get
the desired equation. 

\medskip
2) To prove the second part of the theorem we return to Equation (\ref{eq:A}) and work on expanding the numerator $\left<\overrightarrow{AB},
\overrightarrow{AD}\right>_A$ via Axiom 1 and the fact that $\overrightarrow{AD}=_{A}\overrightarrow{AB}+_{A}\overrightarrow{BD}$
and $\overrightarrow{AB}=_{A}\overrightarrow{AD}+_{A}\overrightarrow{DB}$.
Equation (\ref{eq:A}) may then be rewritten as
\begin{align}
||\overrightarrow{AB}||_{A}\;||\overrightarrow{AD}||_{A} &=
\left<\overrightarrow{AB},\overrightarrow{AD}\right>_{A}
=
\left<\overrightarrow{AD}+_{A}\overrightarrow{DB},\overrightarrow{AD}\right>_{A}\notag\\
&=
\left<\overrightarrow{AD},\overrightarrow{AD}\right>_{A}+\left<\overrightarrow{DB},\overrightarrow{AD}\right>_{A}
= \left<\overrightarrow{AD},\overrightarrow{AD}\right>_{A}+\left<\overrightarrow{DB},\overrightarrow{AB} + \overrightarrow{BD}\right>_{A}\notag\\
&= ||\overrightarrow{AD}||_{A}^{2}+\left<\overrightarrow{DB},\overrightarrow{AB}\right>_{A}-||\overrightarrow{BD}||_{A}^{2}.\label{eq:FL}
\end{align}

\medskip
The next step is to manipulate Equation (\ref{eq:FLJ}) through Definition \ref{def:(-)Minuse.arrowAB} and repeated applications of Axiom 1.
Definition \ref{def:(-)Minuse.arrowAB} implies that $\overrightarrow{BD}=_{A}-\overrightarrow{DB}$.
Thus, it follows by Equation (\ref{eq:A01-1}) and Equation  (\ref{eq:AXIOMb})  that 
\[
\left<\overrightarrow{AB},\overrightarrow{BD}\right>
= \left<\overrightarrow{BD},\overrightarrow{AB}\right>
=-\left<\overrightarrow{AB},\overrightarrow{DB}\right>.
\]
This means that we can rewrite Equation (\ref{eq:FLJ}) (noticing
that $||\overrightarrow{DB}||_{A}=||\overrightarrow{BD}||_{A}$ by
Equation (\ref{eq:LEM11})) as follows
\begin{equation}
\left<\overrightarrow{DB},\overrightarrow{AB}\right>=\mp||\overrightarrow{AB}||_{A}\;||\overrightarrow{DB}||_{A}.\label{eq:FGH}
\end{equation}

If we combine Equations (\ref{eq:FL}) and (\ref{eq:FGH}), we
get
\[
||\overrightarrow{AD}||_{A}^{2}\mp||\overrightarrow{AB}||_{A}\;||\overrightarrow{DB}||_{A}-||\overrightarrow{DB}||_{A}^{2}=||\overrightarrow{AB}||_{A}\;||\overrightarrow{AD}||_{A}.
\]
Rearranging the above equation yields 
\[
||\overrightarrow{AD}||_{A}^{2}-||\overrightarrow{DB}||_{A}^{2}=||\overrightarrow{AB}||_{A}\;||\overrightarrow{AD}||_{A}\pm||\overrightarrow{AB}||_{A}\;||\overrightarrow{DB}||_{A},
\]
that is, 
\begin{equation}
(||\overrightarrow{AD}||_{A}-||\overrightarrow{DB}||_{A})(||\overrightarrow{AD}||_{A}+||\overrightarrow{DB}||_{A})=||\overrightarrow{AB}||_{A}\;(||\overrightarrow{AD}||_{A}\pm||\overrightarrow{DB}||_{A}).\label{eq:SIGN}
\end{equation}
Now Equation (\ref{eq:SIGN}) tells us that either 
\begin{align}
||\overrightarrow{AD}||_{A}-||\overrightarrow{BD}||_{A} & =||\overrightarrow{AB}||_{A},\qquad\text{or}\label{eq:minu}\\
||\overrightarrow{AD}||_{A}+||\overrightarrow{DB}||_{A} & =||\overrightarrow{AB}||_{A}\label{eq:posit},
\end{align}
as desired.

\medskip
3) Notice that when proving Part 2 either Equation (\ref{eq:minu}) or Equation (\ref{eq:posit}), is valid.
We will show the validity of Equation (\ref{eq:minu}) being valid
and leave the validity of Equation (\ref{eq:posit}) to the reader. Since
$A\neq_{P}B$, Lemma \ref{lem:||AB||=00003D0 IFF} implies that
$||\overrightarrow{AB}||_{A}\neq0$. Hence, Equation becomes
(\ref{eq:minu}) as $||\overrightarrow{AD}||_{A}=||\overrightarrow{AB}||_{A}\left(1+\frac{||\overrightarrow{BD}||_{A}}{||\overrightarrow{AB}||_{A}}\right)$.
If we put
\[
t_{D}=1+\frac{||\overrightarrow{BD}||_{A}}{||\overrightarrow{AB}||_{A}},
\]
 then we can rewrite the preceding equation as $||\overrightarrow{AD}||_{A}=(t_{D})||\overrightarrow{AB}||_{A}$.
Notice that $t_{D}>0$ as $||\overrightarrow{BD}||_{A},\;||\overrightarrow{AB}||_{A}>0$.
Hence, by Equation (\ref{eq:LEM00}) we have 
\begin{equation}
||\overrightarrow{AD}||_{A}  =(t_{D})||\overrightarrow{AB}||_{A} = ||(t_{D})\overrightarrow{AB}||_{A}.\label{eq:one}
\end{equation}
Also, by Equations (\ref{eq:A01-1}) and (\ref{eq:LEM00}) we have
\begin{equation}\label{eq:two}
\frac{\left<\overrightarrow{AD},(t_{D})\overrightarrow{AB}\right>}{||\overrightarrow{AD}||_{A}\;||(t_{D})\overrightarrow{AB}||_{A}}  
=\frac{t_{D}\;\left<\overrightarrow{AD},\overrightarrow{AB}\right>}{t_{D}\;||\overrightarrow{AD}||_{A}\;||\overrightarrow{AB}||_{A}}
=\frac{\left<\overrightarrow{AD},\overrightarrow{AB}\right>}{||\overrightarrow{AD}||_{A}\;||\overrightarrow{AB}||_{A}} = 1,
\end{equation}
where the last equality follows by the hypothesis of this theorem.
By Definition
\ref{def:(Scalar-Multiplication-of} we have 
$(t_{D})\overrightarrow{AB}=_{A}\overrightarrow{AK}$
 for some point $K$, and by Definition \ref{def:A LINE} the point
$K$ lies on the line $l_{AB}$. 
Equation (\ref{eq:one}) implies that $||\overrightarrow{AD}|| = ||\overrightarrow{AK}||$ while
Equation (\ref{eq:two}) implies that
$\frac{\left<\overrightarrow{AD},\overrightarrow{AK}\right>}{||\overrightarrow{AD}||_{A}\;||\overrightarrow{AK}||_{A}}  = 1$.
Thus, Proposition \ref{lem:ABAD0} is applicable and shows that $D=_{P}K$, a point on $l_{AB}$. 
\end{proof}

\medskip
An analogous result occurs when the postulated inner product in the hypothesis of Theorem \ref{thm:ONETWO} is set equal to $-1$.

\begin{figure}[ht]
\includegraphics[scale=0.4]{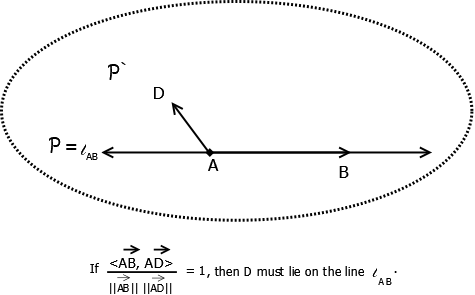}
\caption{An illustration of Theorem \ref{thm:ONETWO1}.}
\label{fig:Pnts-on-a}
\end{figure}

\begin{thm}
\label{thm:ONETWO1}Let $A,\;B$, and $D$ be three distinct points of $\mathcal{P}$
such that
$\frac{\left<\overrightarrow{AB},\overrightarrow{AD}\right>_{A}}{||\overrightarrow{AB}||_{A}\;||\overrightarrow{AD}||_{A}}=-1$.
Then the following is true:
\begin{itemize}
\item[1.]
 \[
 \frac{\left<\overrightarrow{AB},\overrightarrow{BD}\right>}{||\overrightarrow{AB}||_{A}\;||\overrightarrow{BD}||_{A}}=\pm1;\]

\item[2.] 
\[
||\overrightarrow{AD}||_{A} = -||\overrightarrow{AB}||_{A}+||\overrightarrow{BD}||_{A};\]
\begin{figure}[ht]
\centering
\includegraphics[width=0.5\linewidth]{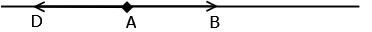}
\caption{A depiction of three points $A,\;B,\;D$ when $||\protect\overrightarrow{AD}||_{A} = -||\protect\overrightarrow{AB}||_{A}+||\protect\overrightarrow{BD}||_{A}$.}
\label{fig:THEOREM37}
\end{figure}

\item[3.]
there exists $t_{D}\in\mathbb{R}$ with $t_D < 0$ such that $\overrightarrow{AD}=_{A}(t_{D})\overrightarrow{AB}$;
thus the point $D$ lies on the line $l_{AB}$; see Figure \ref{fig:Pnts-on-a}.
\end{itemize}
\end{thm}

\begin{proof}
The proof of Part 1 is the same as the proof of Part 1 of Theorem \ref{thm:ONETWO} and is left to the reader.
The proof of Part 2 is slightly more subtle. Since 
$\left<\overrightarrow{AB},\overrightarrow{AD}\right>_{A} = -||\overrightarrow{AB}||_{A}||\overrightarrow{AD}||_{A}$, 
Equation (\ref{eq:SIGN}) becomes
\begin{equation}
(||\overrightarrow{AD}||_{A}-||\overrightarrow{DB}||_{A})(||\overrightarrow{AD}||_{A}+||\overrightarrow{DB}||_{A})=
-||\overrightarrow{AB}||_{A}\;(||\overrightarrow{AD}||_{A}\mp||\overrightarrow{DB}||_{A}).\label{eq:SIGN1}
\end{equation}
A priori, it seems like Equation (\ref{eq:SIGN1}) has two possibilities.
However, if we choose the minus sign option, Equation (\ref{eq:SIGN1}) becomes 
\[
(||\overrightarrow{AD}||_{A}-||\overrightarrow{DB}||_{A})
(||\overrightarrow{AD}||_{A}+||\overrightarrow{DB}||_{A})=
-||\overrightarrow{AB}||_{A}\;(||\overrightarrow{AD}||_{A}-||\overrightarrow{DB}||_{A})
\]
which leads to the contradiction of $-||\overrightarrow{AB}|| = ||\overrightarrow{AD}||_{A}+||\overrightarrow{DB}||_{A}$.
Thus Equation (\ref{eq:SIGN1}) really is
\[
(||\overrightarrow{AD}||_{A}-||\overrightarrow{DB}||_{A})(||\overrightarrow{AD}||_{A}+||\overrightarrow{DB}||_{A})=
-||\overrightarrow{AB}||_{A}\;(||\overrightarrow{AD}||_{A}+||\overrightarrow{DB}||_{A}),\]
which is equivalent to
\begin{equation}\label{adeq1}
||\overrightarrow{AD}||_{A}-||\overrightarrow{BD}||_{A} = -||\overrightarrow{AB}||_{A}.
\end{equation}
For the proof of Part 3 we need to first rewrite Equation (\ref{adeq1}) as
\begin{equation}\label{adeqrewrite}
||\overrightarrow{AD}||_{A}+||\overrightarrow{AB}||_{A} = ||\overrightarrow{BD}||_{A}.
\end{equation}
Equation (\ref{adeqrewrite}) shows that $||\overrightarrow{AB}||_{A}\leq ||\overrightarrow{BD}||_{A}$, which in turn implies 
\begin{equation}\label{ratio1}
1\leq \frac{||\overrightarrow{BD}||_{A}}{||\overrightarrow{AB}||_{A}}.
\end{equation}
Also, Equation (\ref{adeq1}) implies that
$-||\overrightarrow{AD}||_{A}=||\overrightarrow{AB}||_{A}\left(1-\frac{||\overrightarrow{BD}||_{A}}{||\overrightarrow{AB}||_{A}}\right)$.
If we put
\[
t_{D}=1-\frac{||\overrightarrow{BD}||_{A}}{||\overrightarrow{AB}||_{A}},
\]
 then we can rewrite the preceding equation as $-||\overrightarrow{AD}||_{A}=(t_{D})||\overrightarrow{AB}||_{A}$.
Equation (\ref{ratio1}) implies  $t_{D}\leq 0$. However if $t_{D} = 0$, we get a contradiction since by hypothesis $A$ is distinct from $D$.
So in fact $t_{D} < 0$. 
Then
\[
\frac{\left<\overrightarrow{AD},(t_{D})\overrightarrow{AB}\right>}{||\overrightarrow{AD}||_{A}\;||(t_{D})\overrightarrow{AB}||_{A}} 
 =\frac{t_{D}\;\left<\overrightarrow{AD},\overrightarrow{AB}\right>}{|t_{D}|\;||\overrightarrow{AD}||_{A}\;||\overrightarrow{AB}||_{A}}
=-\frac{\left<\overrightarrow{AD},\overrightarrow{AB}\right>}{||\overrightarrow{AD}||_{A}\;||\overrightarrow{AB}||_{A}} = 1,
\]
where the last equality made use of the hypothesis of the theorem.
Furthermore
\begin{equation}
||\overrightarrow{AD}||_{A}  =-(t_{D})||\overrightarrow{AB}||_{A} = |t_D|||\overrightarrow{AB}||_{A} = 
||(t_{D})\overrightarrow{AB}||_{A}\label{eq:onea},
\end{equation}
and the rest of the argument proceeds analogously to proof of Part 3 of Theorem \ref{thm:ONETWO}.
\end{proof}

\medskip
Finally we are in the position to prove the arrow space equivalent of Euler's first axiom, namely that for any two distinct points
there exists a unique line that contains the two points.
\begin{thm}(Existence of a Unique Line)
\label{thm:WAS.Axiom (3)}Given any two distinct points $A$ and $B$ of $\mathcal{P}$,
there exists a unique line $l_{AB}$ containing $A$ and $B$.
\end{thm}

\begin{proof}
Let $A$ and $B$ be two distinct points of $\mathcal{P}$. It follows by Axiom
3 that the set $\{F\in\mathcal{P}|\;\overrightarrow{AF}=_{A}(t)\;\overrightarrow{AB},\;\text{where}\;t\in\mathbb{R}\}$
is non-empty. By Definition \ref{def:A LINE} this set forms a line
containing $A$ and $B$. To show the uniqueness part of this
theorem, let $l_{AB}=\{F\in\mathcal{P}|\;\overrightarrow{AF}=_{A}(t)\;\overrightarrow{AB},\;\text{where}\;t\in\mathbb{R}\}$
and $l_{ML}=\{K\in\mathcal{P}|\;\overrightarrow{MK}=_{A}(s)\;\overrightarrow{ML},\;\text{where}\;s\in\mathbb{R}\}$
be two lines each of which contains $A$ and $B$, where $M$ and
$L$ are any two distinct points that lie on $l_{AB}$ other than
$A$ and $B$.  By Definition \ref{def:A LINE} there exists
$t_{M},\;t_{L}\in\mathbb{R}$ such that 
\begin{equation}
\overrightarrow{AM}=_{A}(t_{M})\;\overrightarrow{AB}\,\,\text{and}\,\,\overrightarrow{AL}=_{A}(t_{L})\;\overrightarrow{AB}.\label{eq:AMAL}
\end{equation}
To show that $l_{AB}=l_{ML}$, let $D\in l_{ML}$. It follows
by Definition \ref{def:A LINE} that there exists $s_{D}\in\mathbb{R}$
with
\begin{equation}
\overrightarrow{MD}=_{A}(s_{D})\;\overrightarrow{ML}.\label{eq:MDML}
\end{equation}
See Figure \ref{Fig5*}.
\begin{figure}[ht]\label{Fig5*}
\centering
\includegraphics[width=0.7\linewidth]{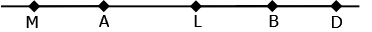}
\caption{The relationship between points $M$, $A$, $L$, $B$, and $D$ on line $l_{AB}$}
\label{fig:THEOREM38}
\end{figure}

In order to show that $D\in l_{AB}$, we need
to find some $t_{D}\in\mathbb{R}$ such that $\overrightarrow{AD}=_{A}(t_{D})\;\overrightarrow{AB}.$
Since $A,\;B,\;M$, and $L$ satisfy the hypotheses
of Lemma \ref{lem:LAAALM} we have $\frac{\left<\overrightarrow{ML},\overrightarrow{AB}\right>_{A}}{||\overrightarrow{ML}||_{A}\;||\overrightarrow{AB}||_{A}}=\pm1$
which is equivalent to 
\begin{equation}
\left<\overrightarrow{ML},\overrightarrow{AB}\right>_{A}=\pm||\overrightarrow{ML}||_{A}\;||\overrightarrow{AB}||_{A}.\label{eq:HERE}
\end{equation}
Since $\overrightarrow{AD}=_{A}\overrightarrow{AM}+_{A}\overrightarrow{MD}$, it follows by Equation (\ref{eq:AXIOMc}) and Definition \ref{def:MeasureOf-any-arrow} that
\begin{equation}
\frac{\left<\overrightarrow{AD},\overrightarrow{AB}\right>_{A}}{||\overrightarrow{AD}||_{A}\;||\overrightarrow{AB}||_{A}}=\frac{\left<\overrightarrow{AM},\overrightarrow{AB}\right>_{A}+\left<\overrightarrow{MD},\overrightarrow{AB}\right>_{A}}{\sqrt{\left<\overrightarrow{AD},\overrightarrow{AD}\right>_{A}}\;||\overrightarrow{AB}||_{A}}.\label{eq:ADMD}
\end{equation}
The numerator of the right hand side in the Equation (\ref{eq:ADMD})
can be rewritten, using Equations (\ref{eq:AMAL}), (\ref{eq:MDML}), (\ref{eq:HERE}), Definition (\ref{def:MeasureOf-any-arrow}),  
and Equation (\ref{eq:A01}) as follows
\begin{align}
\left<\overrightarrow{AM},\overrightarrow{AB}\right>_{A}+\left<\overrightarrow{MD},\overrightarrow{AB}\right>_{A} & =
t_{M}\;||\overrightarrow{AB}||_{A}^{2}+s_{D}\;\left<\overrightarrow{ML},\overrightarrow{AB}\right>_{A}\notag\\
&=
t_{M}\;||\overrightarrow{AB}||_{A}^{2}\pm s_{D}\;||\overrightarrow{ML}||_{A}\;||\overrightarrow{AB}||_{A}.\label{eq:nume}
\end{align} 
Next we simplify the denominator of the right hand side
of the Equation (\ref{eq:ADMD}) through a series of similar calculations. Equations (\ref{eq:AXIOMc}) and (\ref{eq:AXIOMb}) imply that  
\[
\left<\overrightarrow{AD},\overrightarrow{AD}\right>
=\left<\overrightarrow{AM},\overrightarrow{AM}\right>_{A}+2\left<\overrightarrow{MD},\overrightarrow{AM}\right>_{A}+\left<\overrightarrow{MD},\overrightarrow{MD}\right>_{A}.
\]
Using Equations (\ref{eq:AMAL}), (\ref{eq:MDML}), (\ref{eq:HERE}), Definition \ref{def:MeasureOf-any-arrow}, and Equation (\ref{eq:A01}), the preceding equation can be rearranged as follows
\begin{align}
\left<\overrightarrow{AD},\overrightarrow{AD}\right>&=t_{M}^{2}\;||\overrightarrow{AB}||_{A}^{2}+2t_{M}s_{D}\;\left<\overrightarrow{ML},\overrightarrow{AB}\right>_{A}+s_{D}^{2}\;||\overrightarrow{ML}||_{A}^{2}\notag\\
&=t_{M}^{2}\;||\overrightarrow{AB}||_{A}^{2}\pm2t_{M}s_{D}\;||\overrightarrow{ML}||_{A}\;||\overrightarrow{AB}||_{A}+s_{D}^{2}\;||\overrightarrow{ML}||_{A}^{2}\notag\\
&=(t_{M}\;||\overrightarrow{AB}||_{A}\pm s_{D}||\overrightarrow{ML}||_{A})^{2}.\label{eq:flsd}
\end{align}
This means that the denominator of the right hand side of the Equation
(\ref{eq:ADMD}) is
\begin{equation}
\sqrt{\left<\overrightarrow{AD},\overrightarrow{AD}\right>_{A}}\;||\overrightarrow{AB}||_{A}=|t_{M}\;||\overrightarrow{AB}||_{A}^{2}\pm s_{D}||\overrightarrow{ML}||_{A}\;||\overrightarrow{AB}||_{A}|.\label{eq:DENOMI}
\end{equation}
By Combining Equations (\ref{eq:ADMD}), (\ref{eq:nume}), and (\ref{eq:DENOMI})
we find that
\[
\frac{\left<\overrightarrow{AD},\overrightarrow{AB}\right>_{A}}{||\overrightarrow{AD}||_{A}\;||\overrightarrow{AB}||_{A}}=\frac{t_{M}\;||\overrightarrow{AB}||_{A}^{2}\pm s_{D}\;||\overrightarrow{ML}||_{A}\;||\overrightarrow{AB}||_{A}}{|t_{M}\;||\overrightarrow{AB}||_{A}^{2}\pm s_{D}||\overrightarrow{ML}||_{A}\;||\overrightarrow{AB}||_{A}|}=\pm1.
\]
Now by means of Theorem \ref{thm:ONETWO} and \ref{thm:ONETWO1}, we conclude that $D$ is a point on $l_{AB}$.
This means that $l_{AB}\subseteq l_{ML}$, and a similar argument
shows that $l_{ML}\subseteq l_{AB}$ to conclude
that $l_{AB}=l_{ML}$.
\end{proof}

\subsection{Equivalence Relation on $\mathcal{P}_{l_A}$}
Now that we have the definition of the unique $l$ determined by $\overrightarrow{AB}$, we can restrict the arrow
pre-inner product of $\mathcal{P}_A$ to $\mathcal{P}_{l_A}$ and use this restricted arrow pre-inner product to define an equivalence relation 
on the arrows of $\mathcal{P}_{l_A}$ which geometrically captures when two arrows have equal length and same direction.  
Because the arrow pre-inner product on $\mathcal{P}_{l_A}$ is the restriction of the pre-inner product of $\mathcal{P}_A$ to those arrows of $\mathcal{P}_{l_A}$, 
we will also use $\left< -, -\right> _A$ to denoted the arrow pre-inner product of $\mathcal{P}_{l_A}$. 

\medskip
The aforementioned construction is an internal construction for an arrow pre-inner product on $\mathcal{P}_{l_A}$,
 reminiscent of describing the restricted the inner product structure of direct summand in a vector space. 
 We could also describe an external construction for developing an arrow pre-inner product on $\mathcal{P}_{l_A}$ . 
Instead of using restricted arrow pre-inner product from $\mathcal{P}_A$, we focus on $\mathcal{P}_{l_A}$ 
as a separate arrow space and apply the constructions of Section 2 with a {\it new} arrow pre-inner product, which denote as $\left < -, -\right>_{A_l}$.
This new arrow pre-inner product must obey Axioms 1 and 2 and thus gives rises to a new definition 
on $P_{l_A}$ of scalar multiplication, which is denoted as $(a)_{A_l}\overrightarrow{AB}$.

\medskip 
Either construction provides the arrow space $\mathcal{P}_{l_A}$ with an arrow pre-inner product structure
 which is compatible with the propositions and theorems of this and the next subsection.  
 It is a matter of taste as to which construction is preferred by the reader.  For simplicity, we will write our results using the notation of the internal/restricted construction,
  but the reader should note that substitution of $\left< -, -\right> _A$ with 
$\left < -, -\right>_{A_l}$ and $(a)\overrightarrow{AB}$ with $(a)_{A_l}\overrightarrow{AB}$ will not change the validity of the results.

\medskip
The algebraic definition of an equivalence relation on an arrow space $\mathcal{P}_{l_A}$ is 
provided below and will be essential to proving the existence of a unique parallel arrow;
see Theorem \ref{thm:WAS.AXIOM (6)}.

\medskip
\begin{defn}
\label{def:RelatioOfArrowOnline-1}
Given two distinct points $O$ and $P$, let $\mathcal{P}_{l_A}$ denote the arrow space determined by $l_{OP}$.
Let $\overrightarrow{AB}$ and $\overrightarrow{CD}$
be two arrows in $\mathcal{P}_{l_A}$. We say that $\overrightarrow{AB}\;\Re_l\;\overrightarrow{CD}$,
if and only if either $A=_{P}B$ and $C=_{P}D$, or
\[||\overrightarrow{AB}||_{A}=||\overrightarrow{CD}||_{A}\,\,\text{and}\,\,
 \left<\frac{\overrightarrow{AB}}{||\overrightarrow{AB}||_{A}},\;\frac{\overrightarrow{CD}}{||\overrightarrow{CD}||_{A}}\right>_{A}=1.\]
\end{defn}

The following proposition and corollary are immediate consequences of Definition \ref{def:RelatioOfArrowOnline-1}. They will play a role in the proof of Theorem \ref{thm:WAS.AXIOM (5)}.
\begin{prop}
\label{prop:A=00003DB,C=00003DD} If $\overrightarrow{AB}\;\Re_l\;\overrightarrow{CD}$
with $A=_{P}B$, then  $C=_{P}D$.
\end{prop}

\begin{cor}
\label{cor:AnotequalB}If $\overrightarrow{AB}\;\Re_l\;\overrightarrow{CD}$
with $A\neq_{P}B$, then $C\neq_{P}D$.
\end{cor}

Our next goal is to show that $\Re_l$
is an equivalence
relation. To do so we need to introduce the following two theorems.
\begin{thm}
\label{thm:WAS.AXIOM (5)}Let $l_{OP}$ be a line and $\mathcal{P}_l$
be the set of points on this line. Let $\overrightarrow{AB},\;\overrightarrow{CD},\;\overrightarrow{EF}$, and $\overrightarrow{GH}$
be in $\mathcal{P}_{l_A}$ such that $\overrightarrow{AB}\;\Re_l\;\overrightarrow{CD}$
and $\overrightarrow{EF}\;\Re_l\;\overrightarrow{GH}$. Then
\begin{equation}
\left<\overrightarrow{AB},\overrightarrow{EF}\right>_{A}=\left<\overrightarrow{CD},\overrightarrow{GH}\right>_{A}.\label{eq:1DA5}
\end{equation}
\end{thm}
\begin{figure}[ht]
\centering
\includegraphics[width=0.7\linewidth]{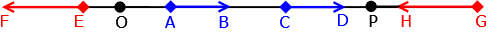}
\caption{An illustration of Theorem \ref{thm:WAS.AXIOM (5)}. Note that the equivalence between arrows is color coded where equivalent arrows share the same color.}
\label{fig:THEOREM30}
\end{figure}

\begin{proof}
 The proof of this theorem will be divided into two cases. First, if $A=_{P}B$
(similarly if $C=_{P}D$), then since $\overrightarrow{AB}\;\Re_l\;\overrightarrow{CD}$,
Proposition \ref{prop:A=00003DB,C=00003DD} implies that $C=_{P}D$.
Hence, Proposition \ref{prop:AACDZERO} implies that
\[
\left<\overrightarrow{AB},\overrightarrow{EF}\right>_{A}=\left<\overrightarrow{AA},\overrightarrow{EF}\right>_{A}=0
\,\,\text{and}\,\,\left<\overrightarrow{CD},\overrightarrow{GH}\right>_{A}=\left<\overrightarrow{CC},\overrightarrow{GH}\right>_{A}=0.
\]
Therefore, $\left<\overrightarrow{AB},\overrightarrow{EF}\right>_{A}=\left<\overrightarrow{CD},\overrightarrow{GH}\right>_{A}=0$.
 A similar argument
follows if $E=_{P}F$ (or $G=_{P}H$).

\medskip
Secondly, we consider the case where $A\neq_{P}B,\;C\neq_{P}D,\;E\neq_{P}F$,
and $G\neq_{P}H$. Then none of the quantities $||\overrightarrow{AB}||_{A},\;||\overrightarrow{CD}||_{A},\;||\overrightarrow{EF}||_{A}$,
and $||\overrightarrow{GH}||_{A}$ is zero.
Our proof technique is to write all of the quantities that appear in Definition \ref{def:RelatioOfArrowOnline-1} in terms of $\overrightarrow{OP}$.
Since $\overrightarrow{AB}=_{A}\overrightarrow{AO}+_{A}\overrightarrow{OB}$,
Definition \ref{def:MeasureOf-any-arrow} and Equation (\ref{eq:AXIOMc}) imply that
\begin{equation}
||\overrightarrow{AB}||_{A}^{2} =||\overrightarrow{AO}+_{A}\overrightarrow{OB}||_{A}^{2}
=\left<\overrightarrow{AO},\overrightarrow{AO}\right>_{A}+2\left<\overrightarrow{AO},\overrightarrow{OB}\right>_{A}
+\left<\overrightarrow{OB},\overrightarrow{OB}\right>_{A}.\label{eq:W}
\end{equation}
Since $A$
and $B$ are points in $\mathcal{P}_l$, by Definition \ref{def:A LINE}
there exist $a,\;b\in\mathbb{R}$ such that 
\begin{equation}
\overrightarrow{OA}=_{A}(a)\;\overrightarrow{OP},\qquad \overrightarrow{OB}=_{A}(b)\;\overrightarrow{OP}.\label{eq:WAAD}
\end{equation}
It follows by Definition \ref{def:(-)Minuse.arrowAB} that
\begin{equation}
\overrightarrow{AO}=_{A}-((a)\;\overrightarrow{OP}).\label{eq:WA}
\end{equation}
Now if we substitute Equation (\ref{eq:WA}) and the second equation from (\ref{eq:WAAD})
into the right hand side of the Equation (\ref{eq:W}) and use Equations (\ref{eq:A01})
and (\ref{eq:A01-1}), we get 
\[
||\overrightarrow{AB}||_{A}^{2}  =(a^{2}-2ab+b^{2})\;\left<\overrightarrow{OP},\overrightarrow{OP}\right>_{A} 
=(a-b)^{2}\;||\overrightarrow{OP}||_{A}^{2}.
\]
Taking the positive square root of both sides in the above yields
\begin{equation}
||\overrightarrow{AB}||_{A}=|a-b|\;||\overrightarrow{OP}||.\label{eq:WAS}
\end{equation}
Similarly, we can write 
\[
\overrightarrow{CD}=_{A}\overrightarrow{CO}+_{A}\overrightarrow{OD},\,\,\,
\overrightarrow{EF}=_{A}\overrightarrow{EO}+_{A}\overrightarrow{OF},\,\,\text{and}\,\,\,\overrightarrow{GH}=_{A}\overrightarrow{GO}+_{A}\overrightarrow{OH}.
\]
Since $\overrightarrow{CD}$, $\overrightarrow{EF}$, and
$\overrightarrow{GH}$ are in $\mathcal{P}_{l_A}$ we see
that the points $C,\;D,E,\;F,\;G$, and $H$ lie on the line $l_{OP}$.
Hence by Definition \ref{def:A LINE} there are real numbers $c,\;d,\;e,\;f,\;g$,
and $h$ such that 
\begin{align}
\overrightarrow{OC}=_{A}(c)\;\overrightarrow{OP},\,\,\;\overrightarrow{OD}=_{A} & (d)\;\overrightarrow{OP},\,\,\,
\overrightarrow{OE}=_{A}(e)\overrightarrow{OP},\label{eq:HOW}\\
\overrightarrow{OF}=_{A}(f)\overrightarrow{OP},\,\,\;\overrightarrow{OG}=_{A} & (g)\overrightarrow{OP},\,\,\,
\overrightarrow{OH}=_{A}(h)\overrightarrow{OH}.\label{eq:HOWA}
\end{align}
If we repeat the steps that led to Equations (\ref{eq:W}) through
(\ref{eq:WAS}) we obtain
\begin{equation}
||\overrightarrow{CD}||_{A}=|c-d|\;||\overrightarrow{OP}||,\;||\overrightarrow{EF}||_{A}=|e-f|\;||\overrightarrow{OP}||,\;||\overrightarrow{GH}||_{A}=|g-h|\;||\overrightarrow{OP}||.\label{eq:AW}
\end{equation}
Calculations similar to those used to derive Equation (\ref{eq:WAS}) show that
\begin{equation}
\left<\overrightarrow{AB},\overrightarrow{EF}\right>_{A}=  (a-b)(e-f)\left<\overrightarrow{OP},\overrightarrow{OP}\right>_{A},\,\,
\left<\overrightarrow{CD},\overrightarrow{GH}\right>_{A}=(c-d)(g-h)\left<\overrightarrow{OP},\overrightarrow{OP}\right>_{A}.
\label{eq:aww}
\end{equation}

We want to emphasize that none of the quantities $|a-b|,\;|c-d|,\;|e-f|$,
and $|g-h|$ is zero. This is because, for example, $|a-b|=0$ means
$a=b$. But by the two equations in (\ref{eq:WAAD}) $a=b$ would
imply that $\overrightarrow{OA}=_{A}\overrightarrow{OB}$. If this
is true, then we would have, by Definition \ref{def:equal arrows},
that $A=_{P}B$. This is a contradiction to our the assumption of $A\neq_{P}B$. The same argument applies for $|c-d|,\;|e-f|$,
and $|g-h|$. 

\medskip
Now by Equations (\ref{eq:WAS}), (\ref{eq:AW}), and (\ref{eq:aww}), we have 
\begin{equation}
\frac{\left<\overrightarrow{AB},\overrightarrow{EF}\right>_{A}}{||\overrightarrow{AB}||_{A}\;||\overrightarrow{EF}||_{A}}
=\frac{(a-b)(e-f)}{|a-b|\;|e-f|},\label{eq:abw}
\end{equation}
and
\begin{equation}
\frac{\left<\overrightarrow{CD},\overrightarrow{GH}\right>_{A}}{||\overrightarrow{CD}||_{A}\;||\overrightarrow{GH}||_{A}} 
 =\frac{(c-d)(g-h)}{|c-d|\;|g-h|}.\label{eq:abbw}
\end{equation}
In a similar manner we obtain
\begin{equation}
\frac{\left<\overrightarrow{AB},\overrightarrow{CD}\right>_{A}}{||\overrightarrow{AB}||_{A}\;||\overrightarrow{CD}||_{A}}=\frac{(a-b)(c-d)}{|a-b|\;|c-d|},\label{eq:bw}
\end{equation}
and
\begin{equation}
\frac{\left<\overrightarrow{EF},\overrightarrow{GH}\right>_{A}}{||\overrightarrow{EF}||_{A}\;||\overrightarrow{GH}||_{A}}=\frac{(e-f)(g-h)}{|e-f|\;|g-h|}.\label{eq:bwaa}
\end{equation}

Since $\overrightarrow{AB}\;\Re_l\;\overrightarrow{CD}$,
Definition \ref{def:RelatioOfArrowOnline-1} implies that $\frac{\left<\overrightarrow{AB},\overrightarrow{CD}\right>_{A}}{||\overrightarrow{AB}||_{A}\;||\overrightarrow{CD}||_{A}}=1$.
Hence, we conclude from Equation (\ref{eq:bw}) that $\frac{(a-b)(c-d)}{|a-b|\;|c-d|}=1$,
which means that 
\begin{equation}
(a-b)(c-d)>0.\label{eq:bwa}
\end{equation}

A similar argument shows that
\begin{equation}
(e-f)(g-h)>0.\label{eq:bwbw}
\end{equation}
The signs of the products $(a-b)(e-f)$ and $(c-d)(g-h)$
in Equations (\ref{eq:abw}) and (\ref{eq:abbw}) will decide whether
the right hand side of each of these two equations is $1$ or $-1$.
We want to prove that the products $(a-b)(e-f)$ and $(c-d)(g-h)$
have the same sign. To do so, we set up the following table.

\medskip

\begin{tabular}{|c|c|c|c|c|}
\hline 
$a-b$ & $+$ & $+$ & $-$ & $-$\tabularnewline
\hline 
$c-d$ & $+$ & $+$ & $-$ & $-$\tabularnewline
\hline 
$e-f$ & $+$ & $-$ & $+$ & $-$\tabularnewline
\hline 
$g-h$ & $+$ & $-$ & $+$ & $-$\tabularnewline
\hline 
($a-b$)($e-f$) & $+$ & $-$ & $-$ & $+$\tabularnewline
\hline 
($c-d$)($g-h$) & $+$ & $-$ & $-$ & $+$\tabularnewline
\hline 
\end{tabular}

\medskip
We conclude from the preceding table that the two products $(a-b)(e-f)$ and
$(c-d)(g-h)$ are either both positive or negative. This means that in
Equations (\ref{eq:abw}) and (\ref{eq:abbw}) we have
\begin{equation}
\frac{\left<\overrightarrow{AB},\overrightarrow{EF}\right>_{A}}{||\overrightarrow{AB}||_{A}\;||\overrightarrow{EF}||_{A}}
=\frac{\left<\overrightarrow{CD},\overrightarrow{GH}\right>_{A}}{||\overrightarrow{CD}||_{A}\;||\overrightarrow{GH}||_{A}}=\pm1.\label{eq:wsa}
\end{equation}
Now since $\overrightarrow{AB}\;\Re_l\;\overrightarrow{CD}$
and $\overrightarrow{EF}\;\Re_l\;\overrightarrow{GH}$, we have
that $||\overrightarrow{AB}||_{A}=||\overrightarrow{CD}||_{A}$
and $||\overrightarrow{EF}||_{A}=||\overrightarrow{GH}||_{A}$.
Thus Equation (\ref{eq:wsa}) simplifies to the desired conclusion.
\end{proof}

With aid of Theorem \ref{thm:WAS.AXIOM (5)} we prove the following
result which will be used directly in proving that $\Re_l$ is an equivalence
relation.
\begin{thm}
\label{thm:WAS Axiom (4)-1}Let $l_{OP}$ be a line and $\mathcal{P}_l$
be the set of points on this line. Let $\overrightarrow{AB},\;\overrightarrow{CD},\;\overrightarrow{EF}$
be three arrows in $\mathcal{P}_{l_A}$ such that $\overrightarrow{AB}\;\Re_l\;\overrightarrow{CD}$
and $\overrightarrow{EF}\;\Re_l\;\overrightarrow{CD}$. Then
$\overrightarrow{AB}\;\Re_l\;\overrightarrow{EF}$.
\end{thm}

\begin{proof}
First, if $A=_{P}B$ (similarly, if $C=_{P}D$ or $E=_{P}F$), then
since $\overrightarrow{AB}\;\Re_l\;\overrightarrow{CD}$, it follows
by Proposition \ref{prop:A=00003DB,C=00003DD} that $C=_{P}D$.
Similarly, $C=_{P}D$ and $\overrightarrow{EF}\;\Re_l\;\overrightarrow{CD}$
means that $E=_{P}F$.
We conclude that if $A=_{P}B$, then $E=_{P}F$, which
means by Definition \ref{def:RelatioOfArrowOnline-1} that $\overrightarrow{AB}\;\Re_l\;\overrightarrow{EF}$. 

\medskip
Suppose now that $A\neq_{P}B,\;C\neq_{P}D$, and $E\neq_{P}F$.
It follows by Lemma \ref{lem:||AB||=00003D0 IFF} that 
\begin{equation}
||\overrightarrow{AB}||_{A}\neq0,\,\,\;||\overrightarrow{CD}||_{A}\neq0,\,\,\;\text{and}\,\,\;||\overrightarrow{EF}||_{A}\neq0.\label{eq:KAB}
\end{equation}
Since $\overrightarrow{AB}\;\Re_l\;\overrightarrow{CD}$ and $\overrightarrow{EF}\;\Re_l\;\overrightarrow{CD}$,
Definition \ref{def:RelatioOfArrowOnline-1} implies that
\begin{equation}
\left<\frac{\overrightarrow{AB}}{||\overrightarrow{AB}||_{A}},\;\frac{\overrightarrow{CD}}{||\overrightarrow{CD}||_{A}}\right>_{A}
=\left<\frac{\overrightarrow{EF}}{||\overrightarrow{EF}||_{A}},\;\frac{\overrightarrow{CD}}{||\overrightarrow{CD}||_{A}}\right>_{A}=1,\label{eq:COR0-1}
\end{equation}
and
\begin{equation}
||\overrightarrow{AB}||_{A}=||\overrightarrow{CD}||_{A},\qquad\;||\overrightarrow{CD}||_{A}=||\overrightarrow{EF}||_{A}.\label{eq:COR1-1}
\end{equation}
Furthermore, since $\overrightarrow{AB}\;\Re_l\;\overrightarrow{CD}$ and $\overrightarrow{EF}\;\Re_l\;\overrightarrow{CD}$,
it follows by Theorem \ref{thm:WAS.AXIOM (5)} that 
\begin{equation}
\left<\overrightarrow{AB},\overrightarrow{EF}\right>_{A}=\left<\overrightarrow{CD},\overrightarrow{CD}\right>_{A}
= ||\overrightarrow{CD}||_{A}^{2}\neq0.\label{eq:qwer}
\end{equation}
Dividing both sides in Equation (\ref{eq:qwer}) by $||\overrightarrow{CD}||_{A}^{2}$
yields $\frac{\left<\overrightarrow{AB},\overrightarrow{EF}\right>_{A}}{||\overrightarrow{CD}||_{A}^{2}}=1$.
This is equivalent to $\frac{\left<\overrightarrow{AB},\overrightarrow{EF}\right>_{A}}{||\overrightarrow{AB}||_{A}\;||\overrightarrow{EF}||_{A}}=1$,
as $||\overrightarrow{AB}||_{A}=||\overrightarrow{CD}||_{A}=||\overrightarrow{EF}||_{A}$
by (\ref{eq:COR1-1}), which means by Definition \ref{def:RelatioOfArrowOnline-1}
that $\overrightarrow{AB}\;\Re_l\;\overrightarrow{EF}$.
\end{proof}

Now we show that $\Re_l$ is an equivalence relation.
\begin{thm}
\label{thm:Equvs.Relan.of.ArrowONline-1}The relation $\Re_l$ in Definition \ref{def:RelatioOfArrowOnline-1} is an
equivalence relation on $\mathcal{P}_{l_A}$.
\end{thm}

\begin{proof}
We first show that the relation $\Re_l$ is reflexive, that is $\overrightarrow{AB}\;\Re_l\;\overrightarrow{AB}$.
Let $\overrightarrow{AB}\in\mathcal{P}_{l_A}$. If $A=_{P}B$, then
it follows directly by Definition \ref{def:MeasureOf-any-arrow}
that $\overrightarrow{AB}\;\Re_l\;\overrightarrow{AB}$. Now suppose
that $A\neq_{P}B$. Then Definition \ref{def:MeasureOf-any-arrow}
and Lemma \ref{lem:tAB} imply that
\[
\left<\frac{\overrightarrow{AB}}{||\overrightarrow{AB}||_{A}},\;\frac{\overrightarrow{AB}}{||\overrightarrow{AB}||_{A}}\right>_{A}
=\left|\left|\frac{\overrightarrow{AB}}{||\overrightarrow{AB}||_{A}}\right|\right|_{A}^{2}=1.
\]
Thus, $\overrightarrow{AB}\;\Re_l\;\overrightarrow{AB}$ for any $\overrightarrow{AB}$
in $\mathcal{P}_{l_A}$.
Next we show that $\Re_l$ is symmetric. Let $\overrightarrow{AB}\;\Re_l\;\overrightarrow{CD}$.
If $A=_{P}B$ (similarly, if $C=_{P}D$), then since $\overrightarrow{AB}\;\Re_l\;\overrightarrow{CD}$,
it follows by Proposition \ref{prop:A=00003DB,C=00003DD} that $C=_{P}D$.
Now since $C=_{P}D$ and $A=_{P}B$, it follows by Definition \ref{def:RelatioOfArrowOnline-1}
that $\overrightarrow{CD}\;\Re_l\;\overrightarrow{AB}$. If $A\neq_{P}B,\;C\neq_{P}D$,
and $\overrightarrow{AB}\;\Re_l\;\overrightarrow{CD}$, then by Definition
\ref{def:RelatioOfArrowOnline-1} we have 
\begin{equation}
||\overrightarrow{AB}||_{A}=||\overrightarrow{CD}||_{A}\,\,\text{and}\,\,
\left<\frac{\overrightarrow{AB}}{||\overrightarrow{AB}||_{A}},\;\frac{\overrightarrow{CD}}{||\overrightarrow{CD}||_{A}}\right>_{A}=1\label{eq:E.R.ARROW2-1}
\end{equation}
But by Equation (\ref{eq:AXIOMb}) we have 
\begin{equation}
\left<\frac{\overrightarrow{AB}}{||\overrightarrow{AB}||_{A}},\;\frac{\overrightarrow{CD}}{||\overrightarrow{CD}||_{A}}\right>_{A}
=\left<\frac{\overrightarrow{CD}}{||\overrightarrow{CD}||_{A}},\;\frac{\overrightarrow{AB}}{||\overrightarrow{AB}||_{A}}\right>_{A} =1.\label{eq:EQV-RAL2-1}
\end{equation}
It follows by the Equations (\ref{eq:E.R.ARROW2-1}), (\ref{eq:EQV-RAL2-1}),
and Definition \ref{def:RelatioOfArrowOnline-1} that $\overrightarrow{CD}\;\Re_l\;\overrightarrow{AB}$
which means that $\Re_l$ is symmetric. Now let $\overrightarrow{AB},\;\overrightarrow{CD},\;\overrightarrow{EF}$
be three arrows in $\mathcal{P}_{l_A}$ such that $\overrightarrow{AB}\;\Re_l\;\overrightarrow{CD}$
and $\overrightarrow{CD}\;\Re_l\;\overrightarrow{EF}$. The transitivity
of $\Re_l$, that is $\overrightarrow{AB}\;\Re_l\;\overrightarrow{EF}$,
follows immediately from the reflexivity of $\Re_l$ and Theorem \ref{thm:WAS Axiom (4)-1}. Therefore,
$\Re_l$ is an equivalence relation on $\mathcal{P}_{l_A}$.
\end{proof}

\subsection{Existence of Parallel Arrow in $\mathcal{P}_{l_A}$}

Now we introduce the following important theorem which is the analog of the parallel axiom in Euclidean
geometry.  Furthermore, as we will later discover, this theorem supplants
Axiom 5 in Sections \ref{sec:6.An-Equivalence-Relation}
and \ref{sec:7.The-Equivalence-C}, which shows that fewer axioms are needed
to construct an arrow space whose underlying set of points is $\mathcal{P}_l$.
 
 \begin{thm}({\it Existence of a Unique Parallel Arrow})
\label{thm:WAS.AXIOM (6)}Given an arrow $\overrightarrow{AB}$
on a line $l_{AB}$ with $A\neq_{P}B$ and any point $P$ on $l_{AB}$,
there exists a unique point $K$ on $l_{AB}$ (and a unique
point $K'$ on $l_{AB}$ ) such that $\overrightarrow{AB}\;\Re_l\;\overrightarrow{PK}$ 
(likewise $\overrightarrow{AB}\;\Re_l\;\overrightarrow{K'P}$).
See Figure 5.10.
\end{thm}
\begin{figure}[ht]\label{paraarrow}
\includegraphics[scale=0.4]{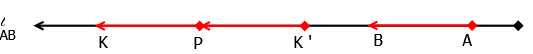}
\caption{Existence of a parallel arrow.}
\end{figure}

\begin{proof}
Let $\overrightarrow{AB}$ be an arrow of $\mathcal{P}_{l_A}$ with $A\neq_{P}B$. Let $P\in l_{AB}$
be any point. By Definition \ref{def:A LINE} and Theorem \ref{thm:aABbAB} there exists a unique
$t\in\mathbb{R}$ such that 
\begin{equation}
\overrightarrow{AP}=_{A}(t)\overrightarrow{AB}.\label{eq:LL}
\end{equation}
By using Definition \ref{def:(-)Minuse.arrowAB} this can be rewritten
as 
\begin{equation}
\overrightarrow{PA}=_{A}-((t)\overrightarrow{AB}).\label{eq:LW}
\end{equation}
We are looking for some real number $s\in\mathbb{R}$ and
a point $K\in l_{AB}$ such that 
\begin{equation}
\overrightarrow{AK}=_{A}(s)\overrightarrow{AB},\label{eq:LAAC}
\end{equation}
\begin{equation}
||\overrightarrow{AB}||_{A}=||\overrightarrow{PK}||_{A},\label{eq:LAAS}
\end{equation}
and
\begin{equation}
\left<\frac{\overrightarrow{AB}}{||\overrightarrow{AB}||_{A}},\;\frac{\overrightarrow{PK}}{||\overrightarrow{PK}||_{A}}\right>_{A}=1.\label{eq:LAAX}
\end{equation}

For the fixed points $A,\;B,\;P$, Equation (\ref{eq:LAAS}) indicates that we should
start with $||\overrightarrow{AB}||_{A}=||\overrightarrow{PK}||_{A}$
and uniquely solve for $K$ in a manner
that satisfies Equations (\ref{eq:LAAC}), (\ref{eq:LAAS}), and (\ref{eq:LAAX}).
This will involve writing all the quantities in Equation (\ref{eq:LAAS}) in terms of $\overrightarrow{AB}$.
Since $\overrightarrow{PK} = \overrightarrow{PA} + \overrightarrow{AK}$, we have 
\begin{align}
||\overrightarrow{AB}||_{A}^{2} & =||\overrightarrow{PK}||_{A}^{2} =||\overrightarrow{PA}+_{A}\overrightarrow{AK}||_{A}^{2}\notag\\
&= \left<\overrightarrow{PA},\overrightarrow{PA}\right>_{A}+2\left<\overrightarrow{PA},\overrightarrow{AK}\right>_{A}
+\left<\overrightarrow{AK},\overrightarrow{AK}\right>_{A} = (t-s)^{2}||\overrightarrow{AB}||_{A}^{2}.\label{eq:LRE}
\end{align}
Since $A\neq_{P}B$, Theorem \ref{lem:||AB||=00003D0 IFF} implies
that $||\overrightarrow{AB}||_{A}\neq0$. Hence we can divide both
sides in Equation (\ref{eq:LRE}) by $||\overrightarrow{AB}||_{A}^{2}$
to get $(t-s)^{2}=1$. Taking the square root of both sides of the
later equation yields
\begin{equation}
t-s=\pm1.\label{eq:LMBN}
\end{equation}
The above equation gives us two values of $s$, namely
\begin{equation}
s_{0}=t+1\qquad s_{1}=t-1.\label{eq:LQ}
\end{equation}
Notice that the number $t$ is fixed by the Equation (\ref{eq:LL})
which yields the uniqueness of $s_{0},\;s_{1}$ in (\ref{eq:LQ}). 
The uniqueness of $s_0$ and $s_1$, in conjunction with Axiom 3, implies the uniqueness of $K$ and $K'$, respectively. 
Therefore, we analyze the situation with $s_0$, (leaving $s_1$ to the reader, and express our claim as follows: there
exists a unique point $K$ of $l_{AB}$ given by the equation 
\begin{equation}
\overrightarrow{AK}=_{A}(s_{0})\overrightarrow{AB},\label{eq:claim}
\end{equation}
where $s_{0}=t+1$. By construction this point satisfies Equation
(\ref{eq:LAAS}).
It remains to confirm that this $K$ also satisfies Equation (\ref{eq:LAAX}). Indeed,
since $||\overrightarrow{PK}||_{A}=||\overrightarrow{AB}||_{A}$, we have 
\begin{equation}
\left<\frac{\overrightarrow{AB}}{||\overrightarrow{AB}||_{A}},\;\frac{\overrightarrow{PK}}{||\overrightarrow{PK}||_{A}}\right>_{A}=\left<\frac{\overrightarrow{AB}}{||\overrightarrow{AB}||_{A}},\;\frac{\overrightarrow{PK}}{||\overrightarrow{AB}||_{A}}\right>_{A}
= \frac{\left<\overrightarrow{AB},\overrightarrow{PK}\right>_{A}}{\left<\overrightarrow{AB},\overrightarrow{AB}\right>_{A}}.\label{eq:SECO}
\end{equation}
Hence, in order to prove the Equation (\ref{eq:LAAX}), it is clear
from the Equation (\ref{eq:SECO}) that we need only to show that
\begin{equation}
\left<\overrightarrow{AB},\overrightarrow{PK}\right>_{A}=\left<\overrightarrow{AB},\overrightarrow{AB}\right>_{A}.\label{eq:SEC}
\end{equation}
Now since $\overrightarrow{PK}=\overrightarrow{PA}+_{A}\overrightarrow{AK}$,
by Equation (\ref{eq:AXIOMc}) we have 
\begin{align*}
\left<\overrightarrow{AB},\overrightarrow{PK}\right>_{A} 
=\left<\overrightarrow{AB},\overrightarrow{PA}+_{A}\overrightarrow{AK}\right>_{A}
=\left<\overrightarrow{AB},\overrightarrow{PA}\right>_{A}+\left<\overrightarrow{AB},\overrightarrow{AK}\right>_{A}.
\end{align*}
Since $\overrightarrow{AK}=_{A}(s_{0})\overrightarrow{AB}=_{A}(t+1)\;\overrightarrow{AB}$
and $\overrightarrow{PA}=_{A}-((t)\overrightarrow{AB)}$,
the right hand side in the above becomes
\begin{align*}
\left<\overrightarrow{AB},\overrightarrow{PK}\right>_{A}
&=\left<\overrightarrow{AB},-((t)\overrightarrow{AB})\right>_{A}+\left<\overrightarrow{AB},(t+1)\;\overrightarrow{AB}\right>_{A}\\
&=-t\left<\overrightarrow{AB},\overrightarrow{AB}\right>_{A}+(t+1)\left<\overrightarrow{AB},\overrightarrow{AB}\right>_{A}.
\end{align*}
where the final equality follows from Equation (\ref{eq:A01}).
Rearranging the above yields $\left<\overrightarrow{AB},\overrightarrow{PK}\right>_{A}=\left<\overrightarrow{AB},\overrightarrow{AB}\right>_{A}$.
This means that Equation (\ref{eq:SEC}) holds and combining Equations
(\ref{eq:SECO}) and (\ref{eq:SEC}) yields Equation (\ref{eq:LAAX}) as desired.
\end{proof}

\section{\label{sec:6.An-Equivalence-Relation}An Equivalence Relation on
Arrow Space}

In this section we consider $\mathcal{P}$ to be any infinite set
of points (including a set of points of a line). All definitions,
axioms, and results from Sections 1 through 5 will be considered
here unless otherwise is stated. 
For an arbitrary arrow space $\mathcal{P}_{A}$, we need to extend the definition
of an equivalence relation on a line introduced in Section \ref{sec:4.Eq.Clss}, which means
restating Definition
\ref{def:RelatioOfArrowOnline-1} in the context of an arbitrary arrow space and defining $\Re$.
 However, we need to make some
changes to show that $\Re$ is an equivalence relation. In Section
\ref{sec:4.Eq.Clss} we had all points contained in Line $l_{OP}$ and this restriction 
in direction allowed us to write all arrows as a scalar multiple of 
a fixed arrow $\overrightarrow{OP}$ and directly prove Theorem \ref{thm:WAS.AXIOM (5)}.
This result was key to proving the transitivity of the relation
$\Re_l$; see Theorem \ref{thm:WAS Axiom (4)-1}. However, if we try
to prove the analog of Theorem \ref{thm:WAS.AXIOM (5)} for a general arrow space we would face the difficulty expressing arrows
in terms of one fixed arrow. Therefore, since this statement is crucial
to prove transitivity of the relation $\Re$, we express it as an axiom,
namely Axiom 4.  Also, Theorem \ref{thm:WAS.AXIOM (6)} from Section
\ref{sec:4.Eq.Clss}, namely the existence of a parallel arrow, must also be restated as Axiom 5 for
the same reason. 
Once $\Re$ is shown to be an equivalence relation we can supplement $\mathcal{P}_{A}$
with vector algebra and form the associated vector space $\mathcal{P}_v$. 
The construction of $\mathcal{P}_v$ from $\mathcal{P}_A$ can also be applied to 
$\mathcal{P}_{l_A}$, taking into account that the statements
of Axiom 4 and 5 from this section will be replaced by Theorems \ref{thm:WAS.AXIOM (5)}
and  \ref{thm:WAS.AXIOM (6)} respectively. Thus, less axioms are required in the case of line than the general case. 

\medskip
We start by defining a relation $\Re$ on $\mathcal{P}_A$.

\begin{defn}
\label{def:RelatioOfArrowOnline}Let $\overrightarrow{AB}$ and $\overrightarrow{CD}$
be two arrows in $\mathcal{P}_{A}$. We say that $\overrightarrow{AB}\;\Re\;\overrightarrow{CD}$,
if and only if either $A=_{P}B$ and $C=_{P}D$, or
\[
||\overrightarrow{AB}||_{A}=||\overrightarrow{CD}||_{A}\,\,\text{and}\,\,
\left<\frac{\overrightarrow{AB}}{||\overrightarrow{AB}||_{A}},\;\frac{\overrightarrow{CD}}{||\overrightarrow{CD}||_{A}}\right>_{A}=1.
\]
\end{defn}
We need the following axiom to prove transitivity of the relation
$\Re$.

Axiom 4. 
Given $\overrightarrow{AB}$, $\overrightarrow{CD}$,
$\overrightarrow{EF}$, and $\overrightarrow{GH}$ such that $\overrightarrow{AB}\;\Re\;\overrightarrow{CD}$
and $\overrightarrow{EF}\;\Re\;\overrightarrow{GH}$, then 
\begin{equation}
\left<\overrightarrow{AB},\overrightarrow{EF}\right>_{A}=\left<\overrightarrow{CD},\overrightarrow{GH}\right>_{A}.\label{eq:A5}
\end{equation}

\begin{figure}[ht]
\label{fig:K1P1L1}
\includegraphics[scale=0.48]{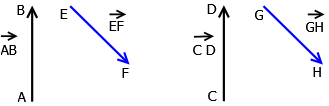}
\caption{An illustration of Axiom 4 where equivalent arrows share the same color.}
\end{figure}

The proof of the following theorem is technically the same as that
of Theorem \ref{thm:WAS Axiom (4)-1},
where Equation (\ref{eq:A5}) of Axiom 4 is to be used
instead of Theorem \ref{thm:WAS.AXIOM (5)}.
\begin{thm}
\label{cor:WAS Axiom (4)}Let $\overrightarrow{AB},\;\overrightarrow{CD},\;\overrightarrow{EF}$
be three arrows in $\mathcal{P}_{A}$ such that $\overrightarrow{AB}\;\Re\;\overrightarrow{CD}$
and $\overrightarrow{CD}\;\Re\;\overrightarrow{EF}$. Then we have
$\overrightarrow{AB}\;\Re\;\overrightarrow{EF}$.
\end{thm}

Since proving that the relation $\Re$ is an equivalence relation
is similar to the proof of Theorem \ref{thm:Equvs.Relan.of.ArrowONline-1}
(notice that transitivity of $\Re$ follows directly from Theorem
\ref{cor:WAS Axiom (4)}) we will state the theorem and skip the
proof to avoid repetition.
\begin{thm}
\label{thm:Equvs.Relan.of.ArrowONline}The relation $\Re$ in Definition
\ref{def:RelatioOfArrowOnline} is an equivalence relation on $\mathcal{P}_{A}$.
\end{thm}

\medskip
Now we are ready to introduce vectors.
\begin{defn}
\label{def: (A Vector)}Consider the family of all equivalence classes
that are obtained from Theorem \ref{thm:Equvs.Relan.of.ArrowONline}
and denote it by $\mathcal{P}_{v}$. We call each equivalence class
$[\overrightarrow{AB}]\in\mathcal{P}_{v}$ a vector, denoted by $v$.
In particular, if $A=_{P}B$, then we call the equivalence class $[\overrightarrow{AA}]$
the zero vector and denote it by $\overrightarrow{0}$.
\end{defn}

The following axiom represents a statement that combines equivalence
classes (vectors), points, and arrows. Basically, it says that for
any vector if a point is given, then there exist two unique arrows
(one has the given point as a tail and the other has it as a head),
and these two arrows are representatives of the given vector.

\medskip
Axiom 5. ({\it Existence of a Unique Parallel Arrow}) Given an arrow $\overrightarrow{AB}$ with $A\neq_{P}B$ and
any point $P\in \mathcal{P}$, there exists a unique point $K$ (and
a unique point K') in $\mathcal{P}$ such that $\overrightarrow{AB}\;\Re\;\overrightarrow{PK}$,
that is

\begin{equation}
||\overrightarrow{AB}||_{A}=||\overrightarrow{PK}||_{A}\;\;\,\text{and}\;\;\,\left<\frac{\overrightarrow{AB}}{||\overrightarrow{AB}||_{A}},\;\frac{\overrightarrow{PK}}{||\overrightarrow{PK}||_{A}}\right>_{A}=1,\label{eq:AXIOM 6}
\end{equation}
where $P,\;K$ are the tail and head of $\overrightarrow{PK}$, respectively,
(likewise, $||\overrightarrow{AB}||_{A}=||\overrightarrow{K'P}||_{A}$, 
$\left<\frac{\overrightarrow{AB}}{||\overrightarrow{AB}||_{A}},\;\frac{\overrightarrow{K'P}}{||\overrightarrow{K'P}||_{A}}\right>_{A}=1$,
where $K',\;P$ are the tail and head of $\overrightarrow{K'P}$,
respectively.)
\begin{figure}[ht]
\includegraphics[scale=0.4]{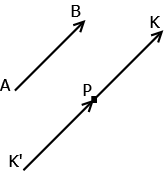}
\caption{A depiction of Axiom 5 where $\protect\overrightarrow{PK}$ and $\protect\overrightarrow{K'P}$ are equivalent to $\protect\overrightarrow{AB}$.}
\end{figure}

Next we use Axiom 5 and Definition \ref{def:ARROWSADDITIONHEADTAIL}
to define an addition of the equivalence classes (vector addition).
We use different notations for equality and addition of equivalence
classes than that of arrows, namely $=_{V}$ and $+_{V}$, respectively.
\begin{defn}
\label{def:VectorAddition}Let $[\overrightarrow{AB}]$ and $[\overrightarrow{CD}]\in\mathcal{P}_{v}$.
Let $P$ be any point in $\mathcal{P}$. Consider the two unique arrows
$\overrightarrow{KP}$ and $\overrightarrow{PL}$ that we get when
we apply Axiom 5 to $\overrightarrow{AB}$ with the point $P$, and
$\overrightarrow{CD}$ with the point $P$, respectively, such that
$\overrightarrow{KP}\;\Re\;\overrightarrow{AB}$ and $\overrightarrow{PL}\;\Re\;\overrightarrow{CD}$.
We define $[\overrightarrow{AB}]+_{V}[\overrightarrow{CD}]:=_{V}\;[\overrightarrow{KP}+_{A}\overrightarrow{PL}]=_{V}[\overrightarrow{KL}]$.
\begin{figure}[ht]
\includegraphics[scale=0.4]{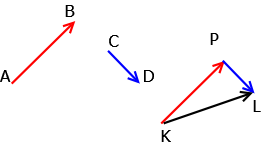}
\caption{An illustration of vector addition where equivalent arrows share the same color.}
\end{figure}
\end{defn}

Now we want to show that the addition $[\overrightarrow{AB}]+_{V}[\overrightarrow{CD}]$
given in Definition \ref{def:VectorAddition} is independent
of the choice of the point $P$. This will be an easy task after we
introduce the three following results.
\begin{lem}
\label{lem:A22}Let $\overrightarrow{K_{1}P_{1}}$, $\overrightarrow{P_{1}L_{1}}$,
$\overrightarrow{K_{2}P_{2}}$, and $\overrightarrow{P_{2}L_{2}}$
be such that (as in Figure (\ref{fig:K2P2L2})) $\overrightarrow{K_{1}P_{1}}\;\Re\;\overrightarrow{K_{2}P_{2}}$
and $\overrightarrow{P_{1}L_{1}}\;\Re\;\overrightarrow{P_{2}L_{2}}$,
where $\{P_{1},\;K_{1},\;L_{1}\}$ and $\{P_{2},\;K_{2},\;L_{2}\}$
are disjoint sets. Then we have
\begin{equation}
||\overrightarrow{K_{1}L_{1}}||_{A}=||\overrightarrow{K_{2}L_{2}}||_{A}.\label{eq:B7}
\end{equation}
\end{lem}

\begin{figure}[ht]
\includegraphics[scale=0.5]{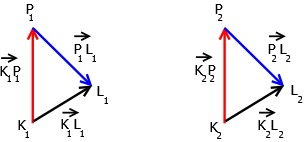}
\caption{\label{fig:K2P2L2}An illustration of Lemma \ref{lem:A22}. 
Note that the equivalence between arrows is color coded where equivalent arrows share the same color.}
\end{figure}

\begin{proof} 
Since $\overrightarrow{K_{1}P_{1}}\;\Re\;\overrightarrow{K_{2}P_{2}}$
and $\overrightarrow{P_{1}L_{1}}\;\Re\;\overrightarrow{P_{2}L_{2}}$,
it follows by Definition \ref{def:RelatioOfArrowOnline} that 
\[
||\overrightarrow{K_{1}P_{1}}||_{A}=||\overrightarrow{K_{2}P_{2}}||_{A}\,\,\text{and}\,\,
||\overrightarrow{P_{1}L_{1}}||_{A}=||\overrightarrow{P_{2}L_{2}}||_{A}.
\]
Squaring both sides of each Equation in the above and using  Definition \ref{def:MeasureOf-any-arrow} yields 
 
\begin{equation}
\left<\overrightarrow{K_{1}P_{1}},\overrightarrow{K_{1}P_{1}}\right>_{A}=\left<\overrightarrow{K_{2}P_{2}},\overrightarrow{K_{2}P_{2}}\right>_{A}\,\text{and}\, 
\left<\overrightarrow{P_{1}L_{1}},\overrightarrow{P_{1}L_{1}}\right>_{A}=\left<\overrightarrow{P_{2}L_{2}},\overrightarrow{P_{2}L_{2}}\right>_{A}
\label{eq:L4-1}
\end{equation}
Also, since $\overrightarrow{K_{1}P_{1}}\;\Re\;\overrightarrow{K_{2}P_{2}}$
and $\overrightarrow{P_{1}L_{1}}\;\Re\;\overrightarrow{P_{2}L_{2}}$
we have, by Axiom 4 and Equation (\ref{eq:AXIOMb}), that 
\begin{align}
\left<\overrightarrow{K_{1}P_{1}},\overrightarrow{P_{1}L_{1}}\right>_{A}=\left<\overrightarrow{K_{2}P_{2}},\overrightarrow{P_{2}L_{2}}\right>_{A}=\left<\overrightarrow{P_{2}L_{2}},\overrightarrow{K_{2}P_{2}}\right>_{A}.\label{eq:sqsq}
\end{align}
Now by Definitions \ref{def:ARROWSADDITIONHEADTAIL} and \ref{def:MeasureOf-any-arrow} and Equations (\ref{eq:L4-1}), (\ref{eq:sqsq}), and (\ref{eq:AXIOMc})  we have
\begin{align}
||\overrightarrow{K_{2}L_{2}}||_{A}^{2} & =
\left<\overrightarrow{K_{2}P_{2}}+_{A}\overrightarrow{P_{2}L_{2}},\overrightarrow{K_{2}P_{2}}+_{A}\overrightarrow{P_{2}L_{2}}\right>_{A}\nonumber\\
&=\left<\overrightarrow{K_{2}P_{2}},\overrightarrow{K_{2}P_{2}}\right>_{A}+2\left<\overrightarrow{K_{2}P_{2}},\overrightarrow{P_{2}L_{2}}\right>_{A}+\left<\overrightarrow{P_{2}L_{2}},\overrightarrow{P_{2}L_{2}}\right>_{A}\nonumber\\
&=\left<\overrightarrow{K_{1}P_{1}},\overrightarrow{K_{1}P_{1}}\right>_{A}
+2\left<\overrightarrow{K_{1}P_{1}},\overrightarrow{P_{1}L_{1}}\right>_{A}+\left<\overrightarrow{P_{1}L_{1}},\overrightarrow{P_{1}L_{1}}\right>_{A}\nonumber\\
&=\left<\overrightarrow{K_1P_1}+_{A}\overrightarrow{P_1L_1}, \overrightarrow{K_1P_1}+_{A}\overrightarrow{P_1L_1}\right>
= ||\overrightarrow{K_{1}L_{1}}||_{A}^{2}.\label{eq:L10-1}
\end{align}
Taking the positive square root of both sides of Equation (\ref{eq:L10-1})
gives the desired result. 
\end{proof}
\begin{thm}
\label{thm:WAS AXIOM (7)} Let $\overrightarrow{K_{1}P_{1}}$,
$\overrightarrow{P_{1}L_{1}}$, $\overrightarrow{P_{2}L_{2}}$, and
$\overrightarrow{K_{2}P_{2}}$ be such that (see Figure (\ref{fig:K2P2L2}))
$\overrightarrow{K_{1}P_{1}}\;\Re\;\overrightarrow{K_{2}P_{2}}$ and
$\overrightarrow{P_{1}L_{1}}\;\Re\;\overrightarrow{P_{2}L_{2}}$,
where $\{P_{1},\;K_{1},\;L_{1}\}$ and $\{P_{2},\;K_{2},\;L_{2}\}$
are disjoint sets. Then 
\begin{equation}
\frac{\left<\overrightarrow{K_{1}L_{1}},\overrightarrow{K_{2}L_{2}}\right>_{A}}{||\overrightarrow{K_{1}L_{1}}||_{A}||\overrightarrow{K_{2}L_{2}}||_{A}}=1.\label{eq:A7}
\end{equation}
\end{thm}

\begin{proof}
Since $\overrightarrow{K_{1}P_{1}}\;\Re\;\overrightarrow{K_{2}P_{2}}$, it follows 
by Definition \ref{def:RelatioOfArrowOnline} that
\begin{equation}
\frac{\left<\overrightarrow{K_{1}P_{1}},\overrightarrow{K_{2}P_{2}}\right>_{A}}{||\overrightarrow{K_{1}P_{1}}||_{A}||\overrightarrow{K_{2}P_{2}}||_{A}}=1,
\qquad
||\overrightarrow{K_{1}P_{1}}||_{A}=||\overrightarrow{K_{2}P_{2}}||_{A}.\label{eq:EXP}
\end{equation}
The two equations above together with Definition \ref{def:MeasureOf-any-arrow} imply that  
\begin{align}
\left<\overrightarrow{K_{1}P_{1}},\overrightarrow{K_{2}P_{2}}\right>_{A}=||\overrightarrow{K_{1}P_{1}}||_{A}^{2}=\left<\overrightarrow{K_{1}P_{1}},\overrightarrow{K_{1}P_{1}}\right>_{A}.\label{eq:km}
\end{align}
Similarly, since $\overrightarrow{P_{1}L_{1}}\;\Re\;\overrightarrow{P_{2}L_{2}}$, we can get
\begin{equation}
\left<\overrightarrow{P_{1}L_{1}},\overrightarrow{P_{2}L_{2}}\right>_{A}=\left<\overrightarrow{P_{1}L_{1}},\overrightarrow{P_{1}L_{1}}\right>_{A}.\label{eq:weex}
\end{equation}
Notice that Equation (\ref{eq:A7}) is equivalent to
\begin{equation}
\left<\overrightarrow{K_{1}L_{1}},\overrightarrow{K_{2}L_{2}}\right>_{A}=||\overrightarrow{K_{1}L_{1}}||_{A}||\overrightarrow{K_{2}L_{2}}||_{A}.\label{eq:K1L1-0-1}
\end{equation}
Also, we have by Lemma \ref{lem:A22} that $||\overrightarrow{K_{1}L_{1}}||_{A}=||\overrightarrow{K_{2}L_{2}}||_{A}$.
Hence, using this and Definition \ref{def:MeasureOf-any-arrow}, Equation (\ref{eq:K1L1-0-1}) becomes
\begin{equation}
\left<\overrightarrow{K_{1}L_{1}},\overrightarrow{K_{2}L_{2}}\right>_{A}=||\overrightarrow{K_{1}L_{1}}||_{A}^{2}=\left<\overrightarrow{K_{1}L_{1}},\overrightarrow{K_{1}L_{1}}\right>_{A}.\label{eq:K1L1-1-1}
\end{equation}
Thus to prove this theorem it is enough to show that Equation (\ref{eq:K1L1-1-1}) holds.
Now by Definition \ref{def:MeasureOf-any-arrow}  and Equation (\ref{eq:AXIOMc})  we have

\begin{align}
\left<\overrightarrow{K_{1}L_{1}},\overrightarrow{K_{2}L_{2}}\right>_{A} &=\left<\overrightarrow{K_{1}P_{1}}+_{A}\overrightarrow{P_{1}L_{1}},\overrightarrow{K_{2}P_{2}}+_{A}\overrightarrow{P_{2}L_{2}}\right>_{A}\nonumber\\
&=\left<\overrightarrow{K_{1}P_{1}},\overrightarrow{K_{2}P_{2}}\right>_{A}
+2\left<\overrightarrow{K_{1}P_{1}},\overrightarrow{P_{2}L_{2}}\right>_{A}
+\left<\overrightarrow{P_{1}L_{1}},\overrightarrow{P_{2}L_{2}}\right>_{A}\nonumber\\
&=\left<\overrightarrow{K_{1}P_{1}},\overrightarrow{K_{1}P_{1}}\right>_{A}
+2\left<\overrightarrow{K_{1}P_{1}},\overrightarrow{P_{1}L_{1}}\right>_{A}
+\left<\overrightarrow{P_{1}L_{1}},\overrightarrow{P_{1}L_{1}}\right>_{A}\label{eq:Kgsr},
\end{align}
where we used three applications of Axiom 4 to obtain the final equality.
Applications of (\ref{eq:AXIOMc}) and Definition \ref{def:ARROWSADDITIONHEADTAIL}, to the preceding equation yields
\[
\left<\overrightarrow{K_{1}L_{1}},\overrightarrow{K_{2}L_{2}}\right>_{A}=\left<\overrightarrow{K_{1}L_{1}},\overrightarrow{K_{1}L_{1}}\right>_{A}.
\]
This is exactly the Equation (\ref{eq:K1L1-1-1}) which ends the proof of this theorem.
\end{proof}
\begin{rem}
\label{rem:R1}In the Equations (\ref{eq:B7}) and (\ref{eq:A7})
above, it is important to mention that the ruling of writing these
equation is, for example, for the two arrows $\overrightarrow{K_{1}P_{1}}$,
$\overrightarrow{P_{1}L_{1}}$, the corresponding arrow is to be written
by taking the tail of the arrow whose head is the common point, and
the head to be chosen as the head of the arrow whose tail is the common
point. In this case the resultant arrow is $\overrightarrow{K_{1}L_{1}}$.
Similarly, for $\overrightarrow{P_{2}L_{2}}$, and $\overrightarrow{K_{2}P_{2}}$, 
the resultant arrow is $\overrightarrow{K_{2}L_{2}}$\label{rem:remaark49}.
\end{rem}

\begin{cor}
\label{cor:COR I,2}Given $\overrightarrow{K_{1}P_{1}}$,
$\overrightarrow{P_{1}L_{1}}$, $\overrightarrow{K_{2}P_{2}}$, and
$\overrightarrow{P_{2}L_{2}}$ in an arrow space $\mathcal{P}_{A}$
where $\overrightarrow{K_{1}P_{1}}\;\Re\;\overrightarrow{K_{2}P_{2}},\;and\;\overrightarrow{P_{1}L_{1}}\;\Re\;\overrightarrow{P_{2}L_{2}}$,
and where $\{P_{1},\;K_{1},\;L_{1}\}\cap \{P_{2},\;K_{2},\;L_{2}\} = \emptyset$, then $\overrightarrow{K_{1}L_{1}}\;\Re\;\overrightarrow{K_{2}L_{2}}$.
\end{cor}

\begin{proof}
By Lemma \ref{lem:A22} we have 
\begin{equation}
||\overrightarrow{K_{1}L_{1}}||_{A}=||\overrightarrow{K_{2}L_{2}}||_{A}.\label{eq:c1}
\end{equation}
Also, by Theorem \ref{thm:WAS AXIOM (7)} we have 
\begin{equation}
\frac{\left<\overrightarrow{K_{1}L_{1}},\overrightarrow{K_{2}L_{2}}\right>_{A}}{||\overrightarrow{K_{1}L_{1}}||_{A}||\overrightarrow{K_{2}L_{2}}||_{A}}=1.\label{eq:c2}
\end{equation}
By means of Definition \ref{def:RelatioOfArrowOnline}, Equations (\ref{eq:c1}), (\ref{eq:c2}) imply that $\overrightarrow{K_{1}L_{1}}\;\Re\;\overrightarrow{K_{2}L_{2}}$,
as desired.
\end{proof}
Next we use Corollary \ref{cor:COR I,2} to prove that the addition
of vectors in Definition \ref{def:VectorAddition} is independent
from the choice of the point $P$.
\begin{thm}
\label{thm:ForAddition.Independnc.}Given any two equivalence classes
$[\overrightarrow{AB}],\;[\overrightarrow{CD}]\in\mathcal{P}_{v}$,
the addition, $[\overrightarrow{AB}]+_{V}[\overrightarrow{CD}]$ as
given in Definition \ref{def:VectorAddition} is independent
of the choice of the point $P$.
\begin{figure}[ht]\label{fig:Independent}
\includegraphics[scale=0.4]{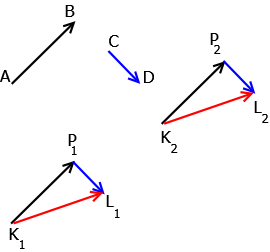}
\caption{Independent Vector Addition.  Note that the equivalence between arrows is color coded where equivalent arrows share the same color.}
\end{figure}
\end{thm}

\begin{proof}
Let $[\overrightarrow{AB}],\;[\overrightarrow{CD}]\in\mathcal{P}_{v}$ and
let $P_{1},\;P_{2}$ be any two distinct points in $\mathcal{P}$.
Using Axiom 5 let $\overrightarrow{K_{1}P_{1}}$ and $\overrightarrow{P_{1}L_{1}}$
be such that $\overrightarrow{K_{1}P_{1}}\;\Re\;\overrightarrow{AB}$
and $\overrightarrow{P_{1}L_{1}}\;\Re\;\overrightarrow{CD}$; and
also let $\overrightarrow{K_{2}P_{2}}$ and $\overrightarrow{P_{2}L_{2}}$
be such that $\overrightarrow{K_{2}P_{2}}\;\Re\;\overrightarrow{AB}$
and $\overrightarrow{P_{2}L_{2}}\;\Re\;\overrightarrow{CD}$ (see
Figure 6.5). Then since the relation $\Re$
is transitive, it follows that
\begin{equation}
\overrightarrow{K_{1}P_{1}}\;\Re\;\overrightarrow{K_{2}P_{2}},\;\,\text{and}\;\,\overrightarrow{P_{1}L_{1}}\;\Re\;\overrightarrow{P_{2}L_{2}}.\label{eq:T0}
\end{equation}
 To show that Definition \ref{def:VectorAddition} is independent
of the choice of any point, it is enough to show that $\overrightarrow{K_{1}L_{1}}\;\Re\;\overrightarrow{K_{2}L_{2}}$. This
is an immediate result of Remark \ref{rem:remaark49} and Corollary \ref{cor:COR I,2}.
\end{proof}

\medskip
Now we define vector scalar multiplication
\begin{defn}
\label{def:Vector-Scalar Multiplication}Let $t\in\mathbb{R}$ and
$u=[\overrightarrow{AB}]$ be any vector, where $\overrightarrow{AB}$
is some representative of an equivalence class. We define the scalar
multiplication 
\[
t\;u=t\;[\overrightarrow{AB}]:=_{V}[(t)\;\overrightarrow{AB}].
\]
\end{defn}

In order to show that
product $t\;u$ in Definition \ref{def:Vector-Scalar Multiplication}
is independent of the choice of the arrow $\overrightarrow{AB}$, we need the following lemma.

\begin{lem}
\label{lem:AB R CD -> tAB R tCD}For any two arrows $\overrightarrow{AB},\;\overrightarrow{CD}$
and any $t\in\mathbb{R}$ if $\overrightarrow{AB}\;\Re\;\overrightarrow{CD}$,
then $(t)\;\overrightarrow{AB}\;\Re\;(t)\;\overrightarrow{CD}$.
\end{lem}

\begin{proof}
If $A=_{P}B$, then since $\overrightarrow{AB}\;\Re\;\overrightarrow{CD}$,
we would have by Proposition \ref{prop:A=00003DB,C=00003DD} that
$C=_{P}D$. Then for any $t\in\mathbb{R}$ we have by Definition \ref{def:(Scalar-Multiplication-of}
that 
\[
(t)\;\overrightarrow{AB}=_{A}(t)\;\overrightarrow{AA}=_{A}\overrightarrow{AA},\;\,
\text{and}\,\;(t)\;\overrightarrow{CD}=_{A}(t)\;\overrightarrow{CC}=_{A}\overrightarrow{CC}.
\]
Thus, since we have by Definition \ref{def:RelatioOfArrowOnline}
that $\overrightarrow{AA}\;\Re\;\overrightarrow{CC}$, we conclude
that $(t)\;\overrightarrow{AB}\;\Re\;(t)\;\overrightarrow{CD}$. Now let
$\overrightarrow{AB},\;\overrightarrow{CD}\in\mathcal{P}_{A}$, with
$A\neq_{P}B,\;C\neq_{P}D$, such that $\overrightarrow{AB}\;\Re\;\overrightarrow{CD}$.
If $t=0$, the if follows by Definition \ref{def:(Scalar-Multiplication-of}
that 
\[
(t)\;\overrightarrow{AB}=_{A}(0)\;\overrightarrow{AB}=_{A}\overrightarrow{AA},\;\,
\text{and}\,\;(t)\;\overrightarrow{CD}=_{A}(0)\;\overrightarrow{CD}=_{A}\overrightarrow{CC}.
\]
Thus, again we have by Definition \ref{def:RelatioOfArrowOnline}
that $\overrightarrow{AA}\;\Re\;\overrightarrow{CC}$ and we conclude
that $(t)\;\overrightarrow{AB}\;\Re\;(t)\;\overrightarrow{CD}$. We now
consider that $\overrightarrow{AB},\;\overrightarrow{CD}\in\mathcal{P}_{A}$,
with $A\neq_{P}B,\;C\neq_{P}D$, such that $\overrightarrow{AB}\;\Re\;\overrightarrow{CD}$
and $t\neq0$. By Definition \ref{def:RelatioOfArrowOnline}, we know that 
\begin{equation}
||\overrightarrow{AB}||_{A}=||\overrightarrow{CD}||_{A}\,\,\text{and}\,\,
\left<\frac{\overrightarrow{AB}}{||\overrightarrow{AB}||_{A}},\;\frac{\overrightarrow{CD}}{||\overrightarrow{CD}||_{A}}\right>_{A}=1.\label{eq:<ab,cd>=00003D1-1}
\end{equation}
By Lemma \ref{lem:tAB} and (\ref{eq:<ab,cd>=00003D1-1}) we have
\begin{equation}
||(t)\;\overrightarrow{AB}||_{A}=|t|\;||\overrightarrow{AB}||_{A}=|t|\;||\overrightarrow{CD}||_{A}=||(t)\;\overrightarrow{CD}||_{A}.\label{eq:closee}
\end{equation}
Also, by Lemma \ref{lem:tAB}, (\ref{eq:A01}), and Equation (\ref{eq:<ab,cd>=00003D1-1}) we have 
\begin{equation}
\left<\frac{(t)\;\overrightarrow{AB}}{||(t)\;\overrightarrow{AB}||_{A}},\;\frac{(t)\;\overrightarrow{CD}}{||(t)\;\overrightarrow{CD}||_{A}}\right>_{A}
=\left<\frac{(t)\;\overrightarrow{AB}}{|t|\;||\overrightarrow{AB}||_{A}},\;\frac{(t)\;\overrightarrow{CD}}{|t|\;||\overrightarrow{CD}||_{A}}\right>_{A}=1.\label{eq:laj}
\end{equation}
Combining 
(\ref{eq:closee}) and (\ref{eq:laj}) with Definition \ref{def:RelatioOfArrowOnline} shows
that
 \[(t)\;\overrightarrow{AB}\;\Re\;(t)\;\overrightarrow{CD}.\]
\end{proof}

\section{\label{sec:7.The-Equivalence-C}The Equivalence Classes Of arrows
as a Vector space}

In this section we show that the set $\mathcal{P}_{v}$ of
all equivalence classes of arrows with the two operations, addition
and scalar multiplication, from Definitions \ref{def:VectorAddition}
and \ref{def:Vector-Scalar Multiplication} satisfy all the axioms
of vector space. Definitions \ref{def:VectorAddition}
and \ref{def:Vector-Scalar Multiplication} imply that $\mathcal{P}_{v}$
is closed under these two operations, namely that adding two equivalence 
classes and multiplying a scalar in any
equivalence class is again an equivalence class. We now prove the
remaining eight axioms of a vector space. We start by showing that the
addition $+_{V}$ is commutative.
\begin{thm}
\label{thm:poi}The addition $+_{V}$ is commutative, that is for any $[\overrightarrow{AB}]$ and $[\overrightarrow{CD}]$
in $\mathcal{P}_{v}$ we have

\begin{equation}
[\overrightarrow{AB}]+_{V}[\overrightarrow{CD}]=_{V}[\overrightarrow{CD}]+_{V}[\overrightarrow{AB}].\label{eq:COMMTV}
\end{equation}
\end{thm}

\begin{proof}
Let $P_{1}$ be any point and use Axiom 5 to find the arrows $\overrightarrow{K_{1}P_{1}}$
and $\overrightarrow{P_{1}L_{1}}$ such that (see Figure 7.1)
\begin{equation}
\overrightarrow{K_{1}P_{1}}\;\Re\;\overrightarrow{AB}\qquad\text{and}
\qquad\overrightarrow{P_{1}L_{1}}\;\Re\;\overrightarrow{CD}.\label{eq:COMM0}
\end{equation}
Then by Definition \ref{def:VectorAddition} we have 
\begin{equation}
[\overrightarrow{AB}]+_{V}[\overrightarrow{CD}]=[\overrightarrow{K_{1}P_{1}}+_{A}\overrightarrow{P_{1}L_{1}}]=[\overrightarrow{K_{1}L_{1}}].\label{eq:COMM1}
\end{equation}
On the other hand, let $P_{2}$ be any point and use Axiom 5 to
find the arrows $\overrightarrow{K_{2}P_{2}}$ and $\overrightarrow{P_{2}L_{2}}$
such that (see Figure (\ref{fig:Z})
below)
\begin{equation}
\overrightarrow{P_{2}L_{2}}\;\Re\;\overrightarrow{AB}\qquad\text{and}
\qquad\overrightarrow{K_{2}P_{2}}\;\Re\;\overrightarrow{CD}.\label{eq:COMM2}
\end{equation}
\begin{figure}[ht]
\label{fig:Z}
\includegraphics[scale=0.4]{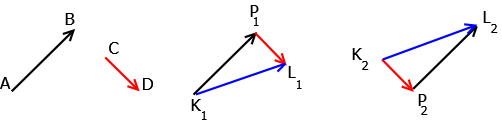}
\caption{An illustration of Theorem \ref{thm:poi}. Note that the equivalence between arrows is color coded where equivalent arrows share the same color.}
\end{figure}
Then we have again by Definition \ref{def:VectorAddition} that
\begin{equation}
[\overrightarrow{CD}]+_{V}[\overrightarrow{AB}]=[\overrightarrow{K_{2}P_{2}}+_{A}\overrightarrow{P_{2}L_{2}}]=[\overrightarrow{K_{2}L_{2}}].\label{eq:COMM3}
\end{equation}
Equations (\ref{eq:COMM1}) and (\ref{eq:COMM3}) imply 
that we need to prove
\begin{equation}
\overrightarrow{K_{1}L_{1}}\;\Re\;\overrightarrow{K_{2}L_{2}}.\label{eq:COMM4}
\end{equation}
By Equations (\ref{eq:COMM0}) and (\ref{eq:COMM2}), and the transitivity of the relation $\Re$, we have 
\begin{equation}
\overrightarrow{K_{1}P_{1}}\;\Re\;\overrightarrow{P_{2}L_{2}}\qquad\text{and}
\qquad\overrightarrow{P_{1}L_{1}}\;\Re\;\overrightarrow{K_{2}P_{2}}.\label{eq:COMM5}
\end{equation}
Applications of Remark \ref{rem:R1} and Corollary \ref{cor:COR I,2}
to Equation (\ref{eq:COMM5}) yields the desired result of Equation (\ref{eq:COMM4}).
\end{proof}

We show also that the addition $+_{V}$ is associative and identify
the additive identity, namely $[\overrightarrow{PP}]$ where $P$
is any point in $\mathcal{P}$.
\begin{thm}
For any $[\overrightarrow{AB}],\;[\overrightarrow{CD}]$, and $[\overrightarrow{EF}]$
in $\mathcal{P}_{v}$ we have
\begin{itemize}
\item[1.] addition $+_{V}$ is associative, that is 
\begin{equation}
([\overrightarrow{AB}]+_{V}[\overrightarrow{CD}])+_{V}[\overrightarrow{EF}]=_{V}[\overrightarrow{AB}]+_{V}([\overrightarrow{CD}]+_{V}[\overrightarrow{EF}]);\label{eq:ASSOTV}
\end{equation}
\item[2.]
for any point $P$ and any $[\overrightarrow{AB}]$, $[\overrightarrow{AB}]+_{V}[\overrightarrow{PP}]=_{V}[\overrightarrow{AB}].$
\end{itemize}
\end{thm}

\begin{proof}
(1) Let $P_{1}$ be any point in $\mathcal{P}$. Then by using Axiom
5, we let $\overrightarrow{K_{1}P_{1}}$ and  $\overrightarrow{P_{1}L_{1}}$
be in $\mathcal{P}_{A}$ such
that 
\[
\overrightarrow{K_{1}P_{1}}\;\Re\;\overrightarrow{AB}\qquad\text{and}
\qquad\overrightarrow{P_{1}L_{1}}\;\Re\;\overrightarrow{CD}.
\]
Also, by an application of Axiom 5 to the arrow $\overrightarrow{EF}$ and the point $L_{1}$, we get the arrow $\overrightarrow{L_{1}T_{1}}$ such that
\[\overrightarrow{L_{1}T_{1}}\;\Re\;\overrightarrow{EF}.\]
See Figure \ref{fig:Theorem55}.
\begin{figure}[ht]
\centering
\includegraphics[width=0.5\linewidth]{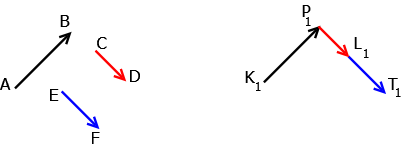}
\caption{An illustration used in the proof of the associativity of vector addition. As usual, arrow equivalence is denoted via color coding. }
\label{fig:Theorem55}
\end{figure}

Now by Definition \ref{def:VectorAddition} and Theorem \ref{thm:ARR.ASSO} we have 
\begin{align}
([\overrightarrow{AB}]+_{V}[\overrightarrow{CD}])+_{V}[\overrightarrow{EF}]&
=_{V}[\overrightarrow{K_{1}P_{1}}+_{A}\overrightarrow{P_{1}L_{1}}]+_{V}[\overrightarrow{L_{1}T_{1}}]\nonumber\\
&=_{V}[(\overrightarrow{K_{1}P_{1}}+_{A}\overrightarrow{P_{1}L_{1}})+_{A}\overrightarrow{L_{1}T_{1}}]\nonumber\\
&=_{V}[\overrightarrow{K_{1}P_{1}}+_{A}(\overrightarrow{P_{1}L_{1}}+_{A}\overrightarrow{L_{1}T_{1}})]\nonumber\\
&=_{V}[\overrightarrow{K_{1}P_{1}}]+_{V}[\overrightarrow{P_{1}L_{1}}+_{A}\overrightarrow{L_{1}T_{1}}]\nonumber\\
&=_{V}[\overrightarrow{AB}]+_{V}([\overrightarrow{CD}]+_{V}[\overrightarrow{EF}])\nonumber.
\end{align}
Thus, the addition $+_{V}$ is associative.

\medskip
(2) Let $P$ be any point and $[\overrightarrow{AB}]$ be any given
vector. Then by Axiom 5 there exists a unique arrow, say
$\overrightarrow{KP}$, such that $\overrightarrow{KP}\;\Re\;\overrightarrow{AB}$.
By Definition \ref{def:VectorAddition} we have 
\begin{align}
[\overrightarrow{AB}]+_{V}[\overrightarrow{PP}] & =_{V}[\overrightarrow{KP}+_{A}\overrightarrow{PP}]
=_{V}[\overrightarrow{KP}]=_{V}[\overrightarrow{AB}].\nonumber
\end{align}
\end{proof}

Next we prove that for any $[\overrightarrow{AB}]$ in $\mathcal{P}_{v}$
there exists the additive inverse which is $[-\overrightarrow{AB}]$.
\begin{thm}
For any $[\overrightarrow{AB}]$ in $\mathcal{P}_{v}$ we have 
\[
[\overrightarrow{AB}]+_{V}[-\overrightarrow{AB}]=_{V}[\overrightarrow{AA}].
\]
\end{thm}

\begin{proof}
By Definition \ref{def:(-)Minuse.arrowAB}
we have $-\overrightarrow{AB}=_{A}\overrightarrow{BA}$,
and since the relation $\Re$ is reflexive we have  $-\overrightarrow{AB}\;\Re\;\overrightarrow{BA}.$
It follows by Definitions \ref{def:ARROWSADDITIONHEADTAIL} and \ref{def:VectorAddition} and that
\[
[\overrightarrow{AB}]+_{V}[-\overrightarrow{AB}]  =_{V}[\overrightarrow{AB}]+_{V}[\overrightarrow{BA}]=_{V}[\overrightarrow{AB}+_{A}\overrightarrow{BA}]=_{V}[\overrightarrow{AA}]\nonumber.
\]
\end{proof}

In the following theorem we prove that the scalar multiplication is
associative and distributive.
\begin{thm}\label{thm:iaddedthislabel}
For any $[\overrightarrow{AB}]$ in $\mathcal{P}_{v}$ and $s,\;t\in\mathbb{R}$
we have
\begin{itemize}
\item[1.] $t\;s\;[\overrightarrow{AB}]=_{V}t\;[(s)\;\overrightarrow{AB}]$,

\item[2.]
for any arrow $\overrightarrow{AB}$ we have $(t+s)\;[\overrightarrow{AB}]=_{V}t\;[\overrightarrow{AB}]+_{V}s\;[\overrightarrow{AB}]$.
\end{itemize}
\end{thm}

\begin{proof}
(1) Let $t,\;s\in\mathbb{R}$ and $\overrightarrow{AB}$ be any arrow.
By Definition \ref{def:Vector-Scalar Multiplication} we have 
\begin{equation}
s\;[\overrightarrow{AB}]=_{V}[(s)\;\overrightarrow{AB}].\label{eq:MULTP0}
\end{equation}
An application of Definition \ref{def:Vector-Scalar Multiplication}
to Equation (\ref{eq:MULTP0}) yields
\[
t\;s\;[\overrightarrow{AB}]=_{V}t\;[(s)\;\overrightarrow{AB}].
\]

\medskip
(2) If $A =_{P} B$, the identity holds trivially. So without loss of generality assume $A\neq_{P} B$.
Let $K$ and $L$ be some points that lie on the line $l_{AB}$
such that
\begin{equation}
(t)\;\overrightarrow{AB}=_{A}\overrightarrow{AK}\qquad\text{and}
\qquad(s)\;\overrightarrow{AB}=_{A}\overrightarrow{AL},\label{eq:D0}
\end{equation}
By Definition \ref{def:Vector-Scalar Multiplication} and
Equation (\ref{eq:D0}) we have
\begin{align}
t\;[\overrightarrow{AB}]+_{V}s\;[\overrightarrow{AB}] & =_{V}[(t)\;\overrightarrow{AB}]+_{V}[(s)\;\overrightarrow{AB}]=_{V}[\overrightarrow{AK}]+_{V}[\overrightarrow{AL}].\label{eq:D1}
\end{align}
If we use Axiom 5 for the arrow $\overrightarrow{AL}$ and the point $K$, we get the arrow $\overrightarrow{KD}$ where
$\overrightarrow{AL}\;\Re\;\overrightarrow{KD}$.
Thus by Definition \ref{def:VectorAddition} we can rewrite Equation (\ref{eq:D1}) as 
\begin{equation}
t\;[\overrightarrow{AB}]+_{V}s\;[\overrightarrow{AB}]=_{V}[\overrightarrow{AK}+_{A}\overrightarrow{KD}]=_{V}[\overrightarrow{AD}].\label{eq:D2}
\end{equation}
On the other hand, by Definition \ref{def:Vector-Scalar Multiplication}
we have 
\begin{equation}
(t+s)\;[\overrightarrow{AB}]=_{V}[(t+s)\;\overrightarrow{AB}].\label{eq:D3}
\end{equation}
Equations (\ref{eq:D2}) and (\ref{eq:D3}) 
imply that proving $(t+s)\;[\overrightarrow{AB}]=_{V}t\;[\overrightarrow{AB}]+_{V}s\;[\overrightarrow{AB}]$
is equivalent to proving
\begin{equation}
\overrightarrow{AD}\;\Re\;(t+s)\;\overrightarrow{AB}.\label{eq:D4}
\end{equation}
Various applications of Definition \ref{def:MeasureOf-any-arrow} and Equation (\ref{eq:AXIOMc}) show that
\begin{align}
||\overrightarrow{AD}||_{A}^{2}&
=\left<\overrightarrow{AD},\overrightarrow{AD}\right>_{A}
=\left<\overrightarrow{AK}+_{A}\overrightarrow{KD},\overrightarrow{AK}+_{A}\overrightarrow{KD}\right>_{A}\nonumber\\
&=\left<\overrightarrow{AK},\overrightarrow{AK}\right>_{A}+ 2\left<\overrightarrow{AK},\overrightarrow{KD}\right>_{A}
+\left<\overrightarrow{KD},\overrightarrow{KD}\right>_{A}.\label{eq:RA}
\end{align}
Moreover,  Equations  (\ref{eq:A01}) and (\ref{eq:D0}) and  Definition\;\ref{def:MeasureOf-any-arrow} imply that 
\begin{equation}
\left<\overrightarrow{AK},\overrightarrow{AK}\right>_{A}=\left<(t)\;\overrightarrow{AB},(t)\;\overrightarrow{AB}\right>_{A}=t^{2}\;\left<\overrightarrow{AB},\overrightarrow{AB}\right>_{A}=t^{2}\;||\overrightarrow{AB}||_{A}^{2}.\label{eq:D9}
\end{equation}
Equation (\ref{eq:D0}) and the fact that $\overrightarrow{AL}\;\Re\;\overrightarrow{KD}$ shows that 
$\overrightarrow{KD}\;\Re\;(s)\;\overrightarrow{AB}$,
and since the relation $\Re$ is reflexive we have
$\overrightarrow{AK}\;\Re\;\overrightarrow{AK}$.
Then an application of  Axiom 4 implies that 
\begin{equation}
\left<\overrightarrow{AK},\overrightarrow{KD}\right>_{A}
=\left<\overrightarrow{AK},(s)\;\overrightarrow{AB}\right>_{A} 
=\left<(t)\;\overrightarrow{AB},(s)\;\overrightarrow{AB}\right>_{A}
=t\;s\;||\overrightarrow{AB}||_{A}^{2}.\label{eq:D10}
\end{equation}

Similarly, we find that
\begin{equation}
\left<\overrightarrow{KD},\overrightarrow{KD}\right>_{A}=s^{2}\;||\overrightarrow{AB}||_{A}^{2}.\label{eq:D11}
\end{equation}
If we plug the Equations (\ref{eq:D9}), (\ref{eq:D10}), and (\ref{eq:D11}) into Equation (\ref{eq:RA}),
we get 
\[
||\overrightarrow{AD}||_{A}^{2}
=t^{2}\;||\overrightarrow{AB}||_{A}^{2}+2t\;s\;||\overrightarrow{AB}||_{A}^{2}+s^{2}\;||\overrightarrow{AB}||_{A}^{2}=(t+s)^{2}\;||\overrightarrow{AB}||_{A}^{2}.
\]
Taking the positive square root of both sides in the above yields
\begin{align}
||\overrightarrow{AD}||_{A}&
=(t+s)\;||\overrightarrow{AB}||_{A}.\label{eq:north}
\end{align}
Furthermore  Equation (\ref{eq:A01}), Lemma \ref{lem:tAB}, and Definition \ref{def:ARROWSADDITIONHEADTAIL} imply that
\begin{align}
&\frac{\left<(t+s)\;\overrightarrow{AB},\overrightarrow{AD}\right>_{A}}{||(t+s)\;\overrightarrow{AB}||_{A}\;||\overrightarrow{AD}||_{A}}
=\frac{(t+s)\;\left<\overrightarrow{AB},\overrightarrow{AD}\right>_{A}}{|t+s|\;||\overrightarrow{AB}||_{A}\;||\overrightarrow{AD}||_{A}}\nonumber\\
&=\frac{(t+s)}{|t+s|}\frac{\left<\overrightarrow{AB},\overrightarrow{AK}+_{A}\overrightarrow{KD}\right>_{A}}{||\overrightarrow{AB}||_{A}\;||\overrightarrow{AD}||_{A}}
=\frac{(t+s)}{|t+s|}\left(\frac{\left<\overrightarrow{AB},\overrightarrow{AK}\right>_{A}}{||\overrightarrow{AB}||_{A}\;||\overrightarrow{AD}||_{A}}+\frac{\left<\overrightarrow{AB},\overrightarrow{KD}\right>_{A}}{||\overrightarrow{AB}||_{A}\;||\overrightarrow{AD}||_{A}}\right).\label{eq:D13}
\end{align}
By definition,
\begin{equation}
\left<\overrightarrow{AB},\overrightarrow{AK}\right>_{A}  =\left<\overrightarrow{AB},(t)\;\overrightarrow{AB}\right>_{A}=t\left<\overrightarrow{AB},\overrightarrow{AB}\right>_{A},\label{eq:ZOBA}
\end{equation}
and by an application of Axiom 4
\begin{equation}
\left<\overrightarrow{AB},\overrightarrow{KD}\right>_{A}  =\left<\overrightarrow{AB},(s)\;\overrightarrow{AB}\right>_{A}=s\left<\overrightarrow{AB},\overrightarrow{AB}\right>_{A}.\label{eq:GOBA}
\end{equation}
If we plug Equations (\ref{eq:north}), (\ref{eq:ZOBA}), and (\ref{eq:GOBA})
into Equation (\ref{eq:D13}), we get
\begin{align}
\frac{\left<(t+s)\;\overrightarrow{AB},\overrightarrow{AD}\right>_{A}}{||(t+s)\;\overrightarrow{AB}||_{A}\;||\overrightarrow{AD}||_{A}}&
=\frac{(t+s)}{|t+s|}\left(\frac{t\left<\overrightarrow{AB},\overrightarrow{AB}\right>_{A}}{||\overrightarrow{AB}||_{A}\;||\overrightarrow{AD}||_{A}}+\frac{s\left<\overrightarrow{AB},\overrightarrow{AB}\right>_{A}}{||\overrightarrow{AB}||_{A}\;||\overrightarrow{AD}||_{A}}\right)\nonumber\\
&=\frac{(t+s)^{2}}{|t+s|^{2}}\frac{\left<\overrightarrow{AB},\overrightarrow{AB}\right>_{A}}{||\overrightarrow{AB}||_{A}^{2}}=1,\label{eq:south}
\end{align}
where the last equality follows from Definition \ref{def:MeasureOf-any-arrow}. Now by means of Definition \ref{def:RelatioOfArrowOnline},  Equations (\ref{eq:north}) and (\ref{eq:south}) imply that  $\overrightarrow{AD}\;\Re\;(t+s)\;\overrightarrow{AB}$, that is Equation (\ref{eq:D4}) holds.\end{proof}

Finally we show another distributive property of the scalar multiplication
and identify the vector scalar multiplicative identity as the real number
$1$.
\begin{thm}\label{thm:Hussmade}
For any $[\overrightarrow{AB}],\;[\overrightarrow{CD}]$ in $\mathcal{P}_{v}$
and any $t\in\mathbb{R}$ we have

\begin{itemize}
\item[1.]
$t\;([\overrightarrow{AB}]+_{V}[\overrightarrow{CD}])=_{V}t\;[\overrightarrow{AB}]+_{V}t\;[\overrightarrow{CD}]$,

\item[2.]
$1\;[\overrightarrow{AB}]=_{V}[\overrightarrow{AB}]$.
\end{itemize}
\end{thm}

\begin{proof}
(1) If $A = _{P}B$ or $C =_{P} D$, the identity holds trivially.
So assume $A\neq_{P}B$ and $C\neq_{P}D$.
Using Axiom 5 for the arrow $\overrightarrow{CD}$
and the point $B$ we get the arrow $\overrightarrow{BK}$ such that
$\overrightarrow{CD}\;\Re\;\overrightarrow{BK}$.
Definitions \ref{def:VectorAddition} and \ref{def:Vector-Scalar Multiplication} 
imply that 
\[
t\;([\overrightarrow{AB}]+_{V}[\overrightarrow{CD}])
=_{V}t\;[\overrightarrow{AB}+_{A}\overrightarrow{BK}]=_{V}t\;[\overrightarrow{AK}]
=[(t)\;\overrightarrow{AK}]=[\overrightarrow{AH}].\label{eq:east}
\]
for some point $H$ that lies on the line $l_{AK}$ with
$(t)\;\overrightarrow{AK}\;\Re\;\overrightarrow{AH}$.
It follows from the preceding relation and Lemma \ref{lem:tAB} that 
$|t|\;||\overrightarrow{AK}||_{A}=_{A}||\overrightarrow{AH}||_{A}$.
Now let 
\begin{equation}
(t)\;\overrightarrow{AB}=_{A}\overrightarrow{AM}\qquad\text{and}\qquad
(t)\;\overrightarrow{CD}=_{A}\overrightarrow{CN},\label{eq:DT3}
\end{equation}
for some points $M$ and $N$ that lie on the lines $l_{AB}$ and
$l_{CD}$, respectively. Then 
\begin{equation}
t\;[\overrightarrow{AB}]+_{V}t\;[\overrightarrow{CD}]
 =_{V}[(t)\;\overrightarrow{AB}]+_{V}[(t)\;\overrightarrow{CD}]=_{V}[\overrightarrow{AM}]+_{V}[\overrightarrow{CN}].\label{eq:DT4}
\end{equation}
Applying Axiom 5 to the arrow $\overrightarrow{CN}$
and the point $M$ gives the unique arrow $\overrightarrow{ML}$ such
that
$\overrightarrow{ML}\;\Re\;\overrightarrow{CN}$.
It follows by Definition \ref{def:VectorAddition},
Equation (\ref{eq:DT4}), and the preceding relation that
\begin{equation}
t\;[\overrightarrow{AB}]+_{V}t\;[\overrightarrow{CD}]=_{V}[\overrightarrow{AM}+_{A}\overrightarrow{ML}]=_{V}[\overrightarrow{AL}].\label{eq:DT6}
\end{equation}
To complete the proof we need to show that
\begin{equation}
\overrightarrow{AH}\;\Re\;\overrightarrow{AL}.\label{eq:ton}
\end{equation}
By Definitions \ref{def:ARROWSADDITIONHEADTAIL} and \ref{def:MeasureOf-any-arrow} and Equation (\ref{eq:AXIOMc}) we have 
\begin{align}
||\overrightarrow{AL}||_{A}^{2}&
=\left<\overrightarrow{AL},\overrightarrow{AL}\right>_{A}=\left<\overrightarrow{AM}+_{A}\overrightarrow{ML},\overrightarrow{AM}+_{A}\overrightarrow{ML}\right>_{A}\nonumber\\
&=\left<\overrightarrow{AM},\overrightarrow{AM}\right>_{A}+2\left<\overrightarrow{AM},\overrightarrow{ML}\right>_{A}+\left<\overrightarrow{ML},\overrightarrow{ML}\right>_{A}.\label{eq:DT9}
\end{align}
Equations (\ref{eq:A01}) and (\ref{eq:DT3}) imply that 
\begin{equation}
\left<\overrightarrow{AM},\overrightarrow{AM}\right>_{A} =\left<(t)\;\overrightarrow{AB},(t)\;\overrightarrow{AB}\right>_{A}=t^{2}\;\left<\overrightarrow{AB},\overrightarrow{AB}\right>_{A}.\label{eq:DT10}
\end{equation}
Since Equation (\ref{eq:DT3}) and $\overrightarrow{ML}\;\Re\;\overrightarrow{CN}$ imply that
$\overrightarrow{ML}\;\Re\;(t)\;\overrightarrow{CD}$, we also find that
\begin{equation}
\left<\overrightarrow{ML},\overrightarrow{ML}\right>_{A}=\left<(t)\;\overrightarrow{CD},(t)\;\overrightarrow{CD}\right>_{A}=t^{2}\left<\overrightarrow{CD},\overrightarrow{CD}\right>_{A}=t^{2}\;\left<\overrightarrow{BK},\overrightarrow{BK}\right>_{A},\label{eq:DT14}
\end{equation}
where the last equality made use of Axiom 4 and the fact that $\overrightarrow{CD}\;\Re\;\overrightarrow{BK}$.
Similar calculations show that 
\begin{equation}
\left<\overrightarrow{AM},\overrightarrow{ML}\right>_{A}
=t^{2}\;\left<\overrightarrow{AB},\overrightarrow{CD}\right>_{A}
=t^{2}\;\left<\overrightarrow{AB},\overrightarrow{BK}\right>_{A}
=t^{2}\;\left<\overrightarrow{BK},\overrightarrow{AB}\right>_{A},\label{eq:DT11}
\end{equation}
Thus, if we plug Equations (\ref{eq:DT10}), (\ref{eq:DT14}), and (\ref{eq:DT11}) into Equation (\ref{eq:DT9}) and use (\ref{eq:AXIOMc}), we get 
\begin{align}
||\overrightarrow{AL}||_{A}^{2}&
=t^{2}\;\left<\overrightarrow{AB},\overrightarrow{AB}\right>_{A}+2t^{2}\;\left<\overrightarrow{AB},\overrightarrow{BK}\right>_{A}+t^{2}\;\left<\overrightarrow{BK},\overrightarrow{BK}\right>_{A}\nonumber\\
&=t^{2}\;\left<\overrightarrow{AB}+_{A}\overrightarrow{BK},\overrightarrow{AB}+_{A}\overrightarrow{BK}\right>_{A}
=t^{2}\;\left<\overrightarrow{AK},\overrightarrow{AK}\right>_{A}
=t^{2}\;||\overrightarrow{AK}||_{A}^{2},\nonumber
\end{align}
where the last equality comes from Definition \ref{def:MeasureOf-any-arrow}.
Taking the positive square root of both sides in the preceding Equation yields 
$||\overrightarrow{AL}||_{A}=|t|\;||\overrightarrow{AK}||_{A}.$
We conclude from the preceding calculation that  
\begin{equation}
|t|\;||\overrightarrow{AK}||_{A}=_{A}||\overrightarrow{AH}||_{A}=||\overrightarrow{AL}||_{A}.\label{eq:DT17}
\end{equation}
By the Equations (\ref{eq:A01}) and (\ref{eq:DT17}) and the fact that $(t)\overrightarrow{AK}\;\Re\;\overrightarrow{AH}$, we have
\begin{equation}
\frac{\left<\overrightarrow{AH},\overrightarrow{AL}\right>_{A}}{||\overrightarrow{AH}||_{A}\;||\overrightarrow{AL}||_{A}} =\frac{\left<(t)\;\overrightarrow{AK},\overrightarrow{AL}\right>_{A}}{|t|^{2}\;||\overrightarrow{AK}||_{A}^{2}}=\frac{t\;\left<\overrightarrow{AK},\overrightarrow{AL}\right>_{A}}{|t|^{2}\;||\overrightarrow{AK}||_{A}^{2}}.\label{eq:DT18}
\end{equation}
To simplify the numerator of the right term in Equation (\ref{eq:DT18}) we use
Definition \ref{def:ARROWSADDITIONHEADTAIL} and Equation (\ref{eq:AXIOMc}) as follows:
\begin{align}
\left<\overrightarrow{AK},\overrightarrow{AL}\right>_{A}&
=\left<\overrightarrow{AB}+_{A}\overrightarrow{BK},\overrightarrow{AM}+_{A}\overrightarrow{ML}\right>_{A}\nonumber\\
&=\left<\overrightarrow{AB},\overrightarrow{AM}\right>_{A}+\left<\overrightarrow{AB},\overrightarrow{ML}\right>_{A}+\left<\overrightarrow{BK},\overrightarrow{AM}\right>_{A}+\left<\overrightarrow{BK},\overrightarrow{ML}\right>_{A}.\label{eq:sun}
\end{align}
Equations (\ref{eq:A01}) and (\ref{eq:DT3}) imply that
\begin{align}
\left<\overrightarrow{AB},\overrightarrow{AM}\right>_{A}=\left<\overrightarrow{AB},(t)\;\overrightarrow{AB}\right>_{A}=t\;\left<\overrightarrow{AB},\overrightarrow{AB}\right>_{A}.\label{eq:DT21}
\end{align}
Since $\overrightarrow{ML}\;\Re\;(t)\;\overrightarrow{CD}$, an application of Axiom 4 shows that
\begin{equation}
\left<\overrightarrow{AB},\overrightarrow{ML}\right>_{A}=\left<\overrightarrow{AB},(t)\;\overrightarrow{CD}\right>_{A}.\label{eq:tem}
\end{equation}
Furthermore, we have $\overrightarrow{CD}\;\Re\;\overrightarrow{BK}$.
It follows by Lemma \ref{lem:AB R CD -> tAB R tCD} that 
$(t)\;\overrightarrow{CD}\;\Re\;(t)\;\overrightarrow{BK}$.
Thus Equation (\ref{eq:tem}) is equivalent to
\begin{equation}
\left<\overrightarrow{AB},\overrightarrow{ML}\right>_{A}=\left<\overrightarrow{AB},(t)\;\overrightarrow{BK}\right>_{A}=t\;\left<\overrightarrow{AB},\overrightarrow{BK}\right>_{A}.\label{eq:DT25}
\end{equation}
Similarly, we find that
\begin{equation}
\left<\overrightarrow{BK},\overrightarrow{ML}\right>_{A}=t\;\left<\overrightarrow{BK},\overrightarrow{BK}\right>_{A}.\label{eq:DT28}
\end{equation}
Using Equation (\ref{eq:A01}) and noticing that $(t)\;\overrightarrow{AB}=_{A}\overrightarrow{AM}$, we can write
\begin{equation}
\left<\overrightarrow{BK},\overrightarrow{AM}\right>_{A}=\left<\overrightarrow{BK},(t)\;\overrightarrow{AB}\right>_{A}=t\;\left<\overrightarrow{BK},\overrightarrow{AB}\right>_{A}.\label{eq:DT27}
\end{equation}
Now if we plug Equations (\ref{eq:DT21}), (\ref{eq:DT25}), (\ref{eq:DT28}), and (\ref{eq:DT27}) into Equation (\ref{eq:sun}), we find that
\[
\left<\overrightarrow{AK},\overrightarrow{AL}\right>_{A}=t\;\left(\left<\overrightarrow{AB},\overrightarrow{AB}\right>_{A}+\left<\overrightarrow{AB},\overrightarrow{BK}\right>_{A}+\left<\overrightarrow{BK},\overrightarrow{AB}\right>_{A}+\left<\overrightarrow{BK},\overrightarrow{BK}\right>_{A}\right),
\]
which can be simplified (via Definitions \ref{def:ARROWSADDITIONHEADTAIL} and  \ref{def:MeasureOf-any-arrow} and Equation (\ref{eq:AXIOMc})) to
\begin{equation}
\left<\overrightarrow{AK},\overrightarrow{AL}\right>_{A}=t\;||\overrightarrow{AK}||_{A}^{2}.\label{eq:DT30}
\end{equation}
Now by Equations (\ref{eq:DT18}) and (\ref{eq:DT30}) we obtain 
\begin{align}
\frac{\left<\overrightarrow{AH},\overrightarrow{AL}\right>_{A}}{||\overrightarrow{AH}||_{A}\;||\overrightarrow{AL}||_{A}}& =\frac{t\;\left<\overrightarrow{AK},\overrightarrow{AL}\right>_{A}}{|t|^{2}\;||\overrightarrow{AK}||_{A}^{2}}=\frac{t^{2}\;||\overrightarrow{AK}||_{A}^{2}}{|t|^{2}\;||\overrightarrow{AK}||_{A}^{2}}=1.\label{eq:tom}
\end{align}
 Definition \ref{def:RelatioOfArrowOnline}, along with Equations (\ref{eq:DT17}) and (\ref{eq:tom}), shows that $\overrightarrow{AH}\;\Re\;\overrightarrow{AL}$ as desired.

\medskip
(2) By Definition \ref{def:Vector-Scalar Multiplication} we have
$1\;[\overrightarrow{AB}]=_{V}[(1)\;\overrightarrow{AB}]$. Thus we need to show that
$\overrightarrow{AB}\;\Re\;(1)\;\overrightarrow{AB}$. By
Theorem \ref{thm:1 AB=00003DAB} we have $(1)\;\overrightarrow{AB}=_{A}\overrightarrow{AB}$,
and since the relation $\Re$ is reflexive it follows that $\overrightarrow{AB}\;\Re\;(1)\;\overrightarrow{AB}$.
\end{proof}

\begin{rem}
Axioms 1, 2, and 4 show that $\left< -, -\right>_{A}$ induces an inner product $\left< -, -\right>_{V}$ on $\mathcal{P}_v$, where
\[
\left<[\overrightarrow{AB}],[\overrightarrow{CD}]\right>_{V}:= \left<\overrightarrow{AB}, \overrightarrow{CD}\right>_{A}.
\]
\end{rem}

\section{Applications of Arrow Spaces to Affine Geometry}
\label{sectionaffinegeomex}
We close this article by showing how the structure of an arrow space provides a useful approach to affine geometry by solving two 
problems using the constructions that we have built so far. We start
with the following theorem which defines the existence of a projection of a given point onto a given line.
\begin{thm}
\label{thm:POINT.LINE}Let $O,G$ be two distinct points. Let $l_{OG}$
be the line containing the points $O$ and $G$. Given any
point $P\notin l_{OG}$, there exists a unique point $W\in l_{OG}$
such that 
\begin{equation}
\left<\overrightarrow{WO},\overrightarrow{WP}\right>_{A}=0.\label{eq:PERPEND}
\end{equation}
\begin{figure}[ht]
\includegraphics[scale=0.45]{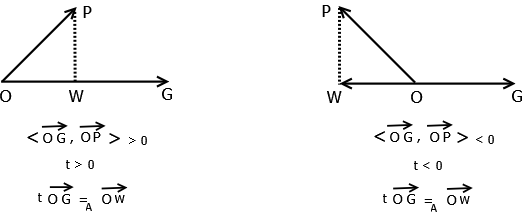}
\caption{\label{fig:WO.WP}An illustration of Theorem \ref{thm:POINT.LINE}, 
the existence of a point $W$ where $\left<\protect\overrightarrow{WO},\protect\overrightarrow{WP}\right>_{A}=0.$}
\end{figure}
\end{thm}

\begin{proof}
Let $O,G$ be two given points with $O\neq_{P}G$. Let $P$ be any point such that $P\notin l_{OG}$.
We want to find $t\in\mathbb{R}$ where 
\begin{equation}
(t)\;\overrightarrow{OG}=_{A}\overrightarrow{OW},\label{eq:PERPEND0}
\end{equation}
 for some $W\in l_{OG}$ which satisfies Equation (\ref{eq:PERPEND}).
See Figure \ref{fig:WO.WP}. By using Definition \ref{def:(-)Minuse.arrowAB}
we can rewrite Equation (\ref{eq:PERPEND0})  as 
$-((t)\;\overrightarrow{OG})=_{A}\overrightarrow{WO}$.
This means that Equation (\ref{eq:PERPEND}) is equivalent to
\begin{equation}
\left<-((t)\;\overrightarrow{OG}),\overrightarrow{WP}\right>_{A}=0.\label{eq:PERPEND1}
\end{equation}
If $t=0$, then it follows by Definition \ref{def:(Scalar-Multiplication-of} and Proposition \ref{prop:AACDZERO} that $O = W$ and the preceding equation
holds trivially. For  $t\neq0$, Definition \ref{def:ARROWSADDITIONHEADTAIL}, Equations (\ref{eq:A01}) and (\ref{eq:AXIOMc}) imply that
\begin{align*}
0 &= \left<-((t)\;\overrightarrow{OG}),\overrightarrow{WP}\right>_{A}
=-t\;\left<\overrightarrow{OG},\overrightarrow{WO}+_{A}\overrightarrow{OP}\right>_{A}\\
&=-t\;\left<\overrightarrow{OG},\overrightarrow{WO}\right>_{A}-t\;\left<\overrightarrow{OG},\overrightarrow{OP}\right>_{A}\\
&=-t\;\left<\overrightarrow{OG},-((t)\;\overrightarrow{OG})\right>_{A}-t\;\left<\overrightarrow{OG},\overrightarrow{OP}\right>_{A}\\
&=t^{2}\;\left<\overrightarrow{OG},\overrightarrow{OG}\right>_{A}-t\;\left<\overrightarrow{OG},\overrightarrow{OP}\right>_{A}=0.
\end{align*}
Solving the preceding equation for $t$ yields
\begin{equation}
t=\frac{\left<\overrightarrow{OG},\overrightarrow{OP}\right>_{A}}{\left<\overrightarrow{OG},\overrightarrow{OG}\right>_{A}}.\label{eq:PERPEND5}
\end{equation}
The two quantities  $\left<\overrightarrow{OG},\overrightarrow{OP}\right>_{A}$ and $\left<\overrightarrow{OG},\overrightarrow{OG}\right>_{A}$ are uniquely determined as $O,\;G$, and $P$ are fixed. Hence $t$ in Equation
(\ref{eq:PERPEND5}) is unique. Therefore, Equation (\ref{eq:PERPEND0}) implies
that there exists a unique point $W\in l_{OG}$
such that Equation (\ref{eq:PERPEND}) holds.
\end{proof}

Theorem \ref{thm:POINT.LINE}  allows for a geometric proof of the arrow space Cauchy-Schwartz analog.
\begin{thm}
\label{csarrow}
 (Cauchy-Schwartz Inequality)
 Given $\overrightarrow{OG}$ and $\overrightarrow{AB}$ in $\mathcal{P}_A$,
 \begin{equation}
 \label{cseq1}
 \left<\overrightarrow{OG}, \overrightarrow{AB}\right>^2_A\leq \left<\overrightarrow{OG}, \overrightarrow{OG}\right>_{A}
 \left<\overrightarrow{AB}, \overrightarrow{AB}\right>_{A},
 \end{equation}
 with equality holding in (\ref{cseq1}) if and only if $\overrightarrow{AB} =_{A} (s)\overrightarrow{OG}$ for some real number $s$.
 \end{thm}
 \begin{proof}
 If $O =_{P} G$, the result is trivially true. So assume $O \neq_{P} G$ and let $l_{OG}$ be the line through the points $O$ and $G$.
 By Axiom 5, set $\overrightarrow{OP}$ to be the unique arrow such that $\overrightarrow{AB}\,\,\Re\,\,\overrightarrow{OP}$. 
 Since Axiom 4 implies that  $\left<\overrightarrow{OG}, \overrightarrow{AB}\right>_A = \left<\overrightarrow{OG}, \overrightarrow{OP}\right>_A$, we see
(\ref{cseq1}) is equivalent to
\begin{equation}
\label{cseq2}
 \left<\overrightarrow{OG}, \overrightarrow{OP}\right>^2_A\leq \left<\overrightarrow{OG}, \overrightarrow{OG}\right>_{A}
 \left<\overrightarrow{OP}, \overrightarrow{OP}\right>_{A}.
 \end{equation}
 Thus it suffices to verify (\ref{cseq2}).
 
 \medskip
 If $P \in l_{OG}$, then $\overrightarrow{OP} = (s)\overrightarrow{OG}$ for some $s\in \mathbb{R}$. In this case
 \begin{align*}
 \left<\overrightarrow{OG}, \overrightarrow{OP}\right>^2_A&= \left<\overrightarrow{OG}, (s)\overrightarrow{OG}\right>^2_A 
 = s^2 \left<\overrightarrow{OG}, \overrightarrow{OG}\right>^2_{A}\\
 &=\left<\overrightarrow{OG}, \overrightarrow{OG}\right>_A\left<(s)\overrightarrow{OG}, (s)\overrightarrow{OG}\right>_{A}\\
 &=\left<\overrightarrow{OG}, \overrightarrow{OG}\right>_A\left<\overrightarrow{OP}, \overrightarrow{OP}\right>_{A}.
 \end{align*}
 
 \medskip
Now assume $P\notin  l_{OG}$. Then by Theorem \ref{thm:POINT.LINE} there is a unique point $W \in l_{OG}$, namely
 $\overrightarrow{OW} = (t)\overrightarrow{OG}$ with 
 $t=\frac{\left<\overrightarrow{OG},\overrightarrow{OP}\right>_{A}}{\left<\overrightarrow{OG},\overrightarrow{OG}\right>_{A}}$,
 such that $\left<\overrightarrow{WO}, \overrightarrow{WP}\right>_{A} = 0$.  In other words, 
 Triangle OPW is a right triangle with hypotenuse $\overrightarrow{OP}$; see Figure \ref{fig:WO.WP}. The Pythagorean theorem implies that
 \begin{align*}
 \left<\overrightarrow{OP},\overrightarrow{OP}\right>_{A} &= \left<\overrightarrow{OW},\overrightarrow{OW}\right>_{A} 
 + \left<\overrightarrow{WP},\overrightarrow{WP}\right>_{A}\\
 &\geq \left<\overrightarrow{OW},\overrightarrow{OW}\right>_{A} = \left<(t)\overrightarrow{OG},(t)\overrightarrow{OG}\right>_{A}\\
 &= t^2\left<\overrightarrow{OG},\overrightarrow{OG}\right>_{A} 
 = \frac{\left<\overrightarrow{OG},\overrightarrow{OP}\right>_{A}^2}{\left<\overrightarrow{OG},\overrightarrow{OG}\right>_{A}},
 \end{align*}
 which after multiplication by $\left<\overrightarrow{OG},\overrightarrow{OG}\right>_{A}$ is equivalent to Inequality (\ref{cseq2}).
 \end{proof}

Next we define the barycentric coordinates of a point $M$ in $\mathcal{P}$, a concept crucial to the definition of affine maps (see Definitions 2.2 and 2.6 of \cite{key-20}).
To do so we fix an origin $O$ in $\mathcal{P}$ and introduce the
following definition
\begin{defn}\label{defbar}
Let $\left\{ P_{i}\right\} _{i=1}^{i=n}$ be a set of points in $\mathcal{P}$
such that $[\overrightarrow{OM}]=\sum_{i=1}^{i=n}\;\lambda_{i}\;[\overrightarrow{OP_{i}}]$,
where $\sum_{i=1}^{i=n}\;\lambda_{i}=1$. We write $M:=_{P}\sum_{i=1}^{i=n}\;\lambda_{i}\;P_{i}$
and call the real numbers $\lambda_{i}\geq0,\;i=1,...,\;n$ the barycentric
coordinates of a point $M$.
\end{defn}

To show that Definition \ref{defbar} is well defined and independent of the chosen origin, we need the following theorem.
\begin{thm}
\label{thm:unq.end.poont}Let $\{P_{i}\}_{i=1}^{n}$ be a family of
distinct points in $\mathcal{P}$. Let $O$ be a coordinate free origin
and $\{\lambda_{i}\}_{i=1}^{n}$ be a family of real numbers such
that $\sum_{i=1}^{n}\lambda_{i}=1$. There exists a unique point
$M$ such that $\sum_{i=1}^{n}[(\lambda_{i})\;\overrightarrow{OP_{i}}]=_{V}[\overrightarrow{OM}]$.
Moreover, the point $M$ is independent from the choice of $O$. That is, if $\tilde{O}$ is any other point and if
\begin{equation}
\sum_{i=1}^{n}[(\lambda_{i})\;\overrightarrow{\tilde{O}P_{i}}]=_{V}[\overrightarrow{\tilde{O}L}],\label{eq:SUM}
\end{equation}
then we must have $L=_{P}M$. See Figure 8.2.
\end{thm}
\begin{figure}[ht]
\label{bartot}
\includegraphics[width=3.0in,height=4.0in]{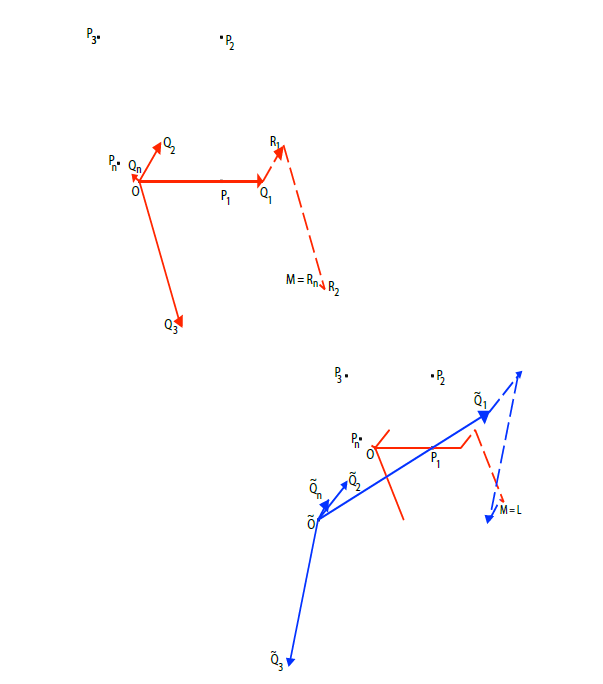}
\caption{The top figure with the red arrows illustrates the barycenter from $O$ while the bottom figure with blue arrows represents the barycenter from $\tilde{O}$.}
\end{figure}

\begin{proof}
Let $(\lambda_{i})\;\overrightarrow{OP_{i}}=_{A}\overrightarrow{OQ_{i}}$, 
$1\leq i\leq n$, for some points $\{Q_{i}\}_{i=1}^{n}$. The existence and uniqueness
of a point $M$, such that 
\begin{equation}
\sum_{i=1}^{n}[(\lambda_{i})\;\overrightarrow{OP_{i}}]=_{V}\sum_{i=1}^{n}(\lambda_{i})[\overrightarrow{OP_{i}}]=_{V}\sum_{i=1}^{n}[\overrightarrow{OQ_{i}}]=_{V}[\overrightarrow{OM}],\label{eq:SUMA}
\end{equation}
follows from $n-1$ applications Definition \ref{def:VectorAddition} as follows
\[
(n-1)\;\text{times}\;\begin{cases}
[\overrightarrow{OQ_{1}}]+_{V}[\overrightarrow{OQ_{2}}]=_{V}[\overrightarrow{OQ_{1}}+_{A}\overrightarrow{Q_{1}R_{1}}]=_{V}[\overrightarrow{OR_{1}}],\\{}
[\overrightarrow{OR_{1}}]+_{V}[\overrightarrow{OQ_{3}}]=_{V}[\overrightarrow{OR_{1}}+_{A}\overrightarrow{R_{1}R_{2}}]=_{V}[\overrightarrow{OR_{2}}],\\
:\\{}
[\overrightarrow{OR_{n-2}}]+_{V}[\overrightarrow{OQ_{n}}]=_{V}[\overrightarrow{OR_{n-2}}+_{A}\overrightarrow{R_{n-2}M}]=_{V}[\overrightarrow{OM}],
\end{cases}
\]
where $\{R_{i}\}_{i=1}^{n-1}$ (we put $M=R_{n-1}$) are the unique points that we get when we apply Axiom 5 in Definition \ref{def:VectorAddition}. See Figure 8.3.
\begin{figure}[ht]
\label{fig:bary}
\includegraphics[scale=0.5]{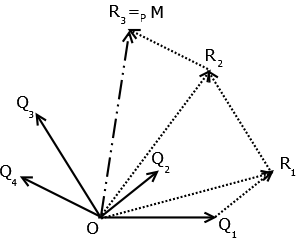}
\caption{$\sum_{i=1}^{n}[(\lambda_{i})\;\protect\overrightarrow{OP_{i}}]=_{V}\sum_{i=1}^{n}[\protect\overrightarrow{OQ_{i}}]=_{V}[\protect\overrightarrow{OM}],\;n=4$}
\end{figure}

\medskip
To show that $M$ is independent of the choice of $O$  notice that by Definitions \ref{def:ARROWSADDITIONHEADTAIL} and Theorem \ref{thm:Hussmade} (1) we have
\begin{equation}
[\overrightarrow{OM}]=_{V}
\sum_{i=1}^{n}\lambda_{i}[\overrightarrow{OP_{i}}]
=_{V}\sum_{i=1}^{n}\lambda_{i}[\overrightarrow{O\tilde{O}}+_{A}\overrightarrow{\tilde{O}P_{i}}]
=_{V}[\overrightarrow{O\tilde{O}}]\sum_{i=1}^{n}\lambda_{i} +_{V}
\sum_{i=1}^{n}\lambda_{i}[\overrightarrow{\tilde{O}P_{i}}],\label{eq:SUMB}
\end{equation}
for any point $\tilde{O}$ other than $O$. Since $\sum_{i=1}^{n}\lambda_{i}=1$, by Theorem \ref{thm:iaddedthislabel}(2) and Equations (\ref{eq:SUMA}) and (\ref{eq:SUMB}) we find that
\begin{equation}
[\overrightarrow{OM}] = _{V}
[\overrightarrow{O\tilde{O}}]+_{V}[\overrightarrow{\tilde{O}L}]=_{V}[\overrightarrow{O\tilde{O}}+_{A}\overrightarrow{\tilde{O}L}]=_{V}[\overrightarrow{OL}],
\end{equation} 
where $[\overrightarrow{\tilde{O}L}]=_{V}\sum_{i=1}^{n}(\lambda_{i})[\overrightarrow{\tilde{O}P_{i}}]$
for some point $L$. This means that $\overrightarrow{OL}\;\Re\;\overrightarrow{OM}$ which  implies by Definition \ref{def:RelatioOfArrowOnline}
and Proposition \ref{lem:ABAD0} that $L=_{P}M$, as desired.
\end{proof}

\end{document}